\documentclass[a4paper]{gtart}

\usepackage{amsmath, amsthm, amssymb, amsfonts, accents, mathrsfs, thmtools, mathtools, eucal}
\usepackage[makeroom]{cancel}

\usepackage[inner=20mm, outer=20mm, textheight=235mm]{geometry}

\usepackage{parskip}

\makeatletter
\def\thm@space@setup{%
  \thm@preskip=\parskip \thm@postskip=0pt
}
\makeatother

\usepackage{graphicx}
\graphicspath{ {png_figures/}}

\usepackage{xcolor}

\usepackage{hyperref}
\hypersetup{
	colorlinks = true,
	linktoc = all,
	linkcolor = red
}

\usepackage[margin=10pt,font=small,labelfont=bf, labelsep=endash]{caption}
\usepackage{xfrac}

\usepackage[color=orange!30]{todonotes}

\usepackage{enumitem}

\usepackage{bbm}

\definecolor{blue1}{rgb}{0.25,0.25,1}
\definecolor{red1}{rgb}{1,0.25,0.25}

\usepackage{tikz}
\usepackage{tikz-cd}

\usetikzlibrary{arrows.meta}
\usetikzlibrary{decorations.markings}

\tikzstyle directed=[postaction={decorate, decoration={markings, mark=at position 0.5 with {\arrow{Latex[length=3mm, width = 3mm]} } } }, very thick ]

\tikzstyle rightarrow=[postaction={decorate, decoration={markings, mark=at position 1 with {\arrow{Stealth[length=3mm, width = 3mm]} } } }, very thick ]

\usetikzlibrary{calc, matrix}

\DeclareMathOperator{\dev}{dev}
\DeclareMathOperator{\hol}{hol}
\DeclareMathOperator{\hull}{hull}

\DeclareMathOperator{\Id}{Id}
\DeclareMathOperator{\Out}{Out}
\DeclareMathOperator{\conv}{conv}
\DeclareMathOperator{\pos}{pos}
\DeclareMathOperator{\bd}{\partial}

\DeclareMathOperator{\im}{im}

\DeclareMathOperator{\MCG}{Mod}

\def\tri{\Delta}

\DeclareMathOperator{\PGL}{PGL}
\DeclareMathOperator{\GL}{GL}
\DeclareMathOperator{\PSL}{PSL}
\DeclareMathOperator{\SL}{SL}
\DeclareMathOperator{\SO}{SO}

\def\A{\mathcal{A}}
\def\X{\mathcal{X}}

\def\Cell{\mathring{\mathcal{C}}}
\def\clCell{\mathcal{C}} 
\def\T{\mathcal{T}}
\def\Tt{\mathcal{T}}
\def\Ttf{\mathcal{T}_f}

\def\decTtf{{\mathcal{T}}_f^\ddag}
\def\singdecTtf{{\mathcal{T}}_f^\dag}
\def\decTeich{{\mathcal{T}}_{\hyp}^\dag}
\def\Teich{\mathcal{T}_{\hyp}}

\def\End{\mathcal{E}}

\def\Z{\mathbb{Z}}
\def\N{\mathbb{N}}
\def\R{\mathbb{R}}
\def\Q{\mathbb{Q}}
\def\RP{\mathbb{RP}}
\def\HH{\mathbb{H}}

\newcommand{\vtx}{C}
\newcommand{\covtx}{r}
\def\vec_dec{\mathcal{V}}
\def\covec_dec{\mathcal{R}}

\newcommand{\trCoords}{\bar{A}}
\newcommand{\eCoords}{\bar{a}}
\newcommand{\eCoordsb}{\bar{b}}
\newcommand{\hyp}{\text{hyp}}

\theoremstyle{plain}
	
	\newtheorem{thm}{Theorem}[section]
	\newtheorem*{thm*}{Theorem}
	\newtheorem{lem}[thm]{Lemma}
	\newtheorem*{lem*}{Lemma}
	\newtheorem{cor}[thm]{Corollary}
	\newtheorem*{cor*}{Corollary}
	
	\newtheorem*{cla*}{Claim}
	
	\newtheorem*{pro*}{Proposition}
	\newtheorem{rem}[thm]{Remark}
	\newtheorem*{rem*}{Remark}
	\newtheorem{exm}[thm]{Example}
	\newtheorem*{exm*}{Example}
	
	\newtheorem*{defn*}{Definition}

\begin{document}

\title{On Moduli Spaces of Convex Projective Structures on Surfaces: \\ Outitude and Cell-Decomposition in Fock-Goncharov Coordinates}
\author{Robert Haraway, Robert L\"owe, Dominic Tate and Stephan Tillmann}

\begin{abstract}
Generalising a seminal result of Epstein and Penner for cusped hyperbolic manifolds, Cooper and Long showed that each decorated strictly convex projective cusped manifold has a canonical cell decomposition. Penner used the former result to describe a natural cell decomposition of decorated Teichm\"uller space of punctured surfaces. We extend this cell decomposition to the moduli space of decorated strictly convex projective structures of finite volume on punctured surfaces.

The proof uses Fock and Goncharov's $\A$--coordinates for doubly decorated structures. In addition, we describe a simple, intrinsic edge-flipping algorithm to determine the canonical cell decomposition associated to a point in moduli space, and show that Penner's centres of Teichm\"uller cells are also natural centres of the cells in moduli space. We show that in many cases, the associated holonomy groups are semi-arithmetic.
\end{abstract}



\makeshorttitle


\section{Introduction}
\label{sec:intro}


Classical Teichm\"uller space is the moduli space of marked hyperbolic structures of finite volume on a surface. In the case of a punctured surface, many geometrically meaningful \emph{ideal cell decompositions} for its 
Teichm\"uller space are known. For instance, quadratic differentials are used for the construction attributed to Harer, Mumford 
and Thurston~\cite{Harer-cohomological-1986}; hyperbolic geometry and geodesic laminations are used by Bowditch and Epstein~\cite{BoEp-natural-1988}; and Penner~\cite{Penner-Decorated-1987} uses Euclidean cell decompositions associated to the 
points in \emph{decorated} Teichm\"uller space. The decoration arises from associating a positive real number to each cusp of the surface, and this is
used in the key construction of Epstein and Penner~\cite{EpPe-Euclidean-1988} that constructs a Euclidean cell decomposition of a decorated and marked hyperbolic surface.
 All of these decompositions of (decorated) Teichm\"uller space are natural in the sense that they are invariant under the action of the mapping class group 
(and hence descend to a cell decomposition of the moduli space of unmarked structures) and that they do not involve any arbitrary choices.

A hyperbolic structure is an example of a \emph{strictly convex projective structure}, and two hyperbolic structures are equivalent as 
hyperbolic structures if and only if they are equivalent as projective structures. Let 
$S_{g,n}$ denote the surface  of genus $g$ with $n$ punctures. We will always assume that $2g+n>2,$ 
so that the surface has negative Euler characteristic. Whereas the classical Teichm\"uller space $\Teich(S_{g,n})$ is homeomorphic with $\R^{6g-6+2n},$ Marquis~\cite{Marquis-espaces-2010} has shown that the analogous moduli space $\Ttf(S_{g,n})$ 
of marked strictly convex projective structures of finite volume on $S_{g,n}$ is homeomorphic with $\R^{16g-16+6n}.$ 

There is a natural \emph{decorated} moduli space $\singdecTtf(S_{g,n})$, again obtained by associating a positive real number to each cusp of the surface.
Cooper and Long~\cite{CoLo-Epstein-Penner-2015} generalised the construction of Epstein and Penner~\cite{EpPe-Euclidean-1988}, thus
associating to each point in the decorated moduli space $\singdecTtf(S)$ of strictly convex projective structures of finite volume an ideal cell decomposition of $S.$
Cooper and Long~\cite{CoLo-Epstein-Penner-2015} state that their construction can be used to define a decomposition of the decorated moduli 
space $\singdecTtf(S),$ but that it is not known whether all components of this decomposition are cells. The main result of this paper establishes the fact that this is indeed always the case.  This was previously only known  in the case where $S$ is the once-punctured torus or a sphere with three punctues~\cite{HaTi-tessellating-2017}.

As in the classical setting, there is a principal $\R^n_+$ foliated fibration 
$\singdecTtf(S_{g,n}) \to {\Ttf}(S_{g,n}),$ and different points in a fibre above a point of 
${\Ttf}(S_{g,n})$ may lie in different components of the cell decomposition of $\singdecTtf(S_{g,n}).$ 
However, if there is only one cusp, then all points in a fibre lie in the same component, and one obtains a 
decomposition of ${\Ttf}(S_{g,1}).$

A strictly convex projective surface is a quotient $\Omega/\Gamma,$ where $\Omega \subset \RP^2$ is a strictly convex domain and $\Gamma$ is a discrete group preserving $\Omega.$ One technical difficulty in working with strictly convex projective structures arises from the fact that as one varies a point in moduli space, not only the associated holonomy group $\Gamma$ varies, but also the domain $\Omega$ varies. The key in our proof is to use a particularly nice coordinate system for the space $\decTtf(S_{g,n})$
of \emph{doubly-decorated}  strictly convex projective structures due to Fock and Goncharov \cite{FoGo-moduli-2006}. This has a principal $\R^n_+$ foliated fibration $\decTtf(S_{g,n}) \to \singdecTtf(S_{g,n}),$ so different points in a fibre above a point of $\singdecTtf(S_{g,n})$ have the same canonical cell decomposition.

Fock and Goncharov \cite{FoGo-moduli-2006} devise parameterisations of $\Ttf(S_{g,n})$ and $\decTtf(S_{g,n})$ by choosing an ideal triangulation of the surface and associating to each triangle one positive parameter and to each edge two. We distinguish the edge parameters by associating them to the edge with different orientations. This gives a parameter space diffeomorphic with $\R^{16g-16+8n},$ which is given two different interpretations.

The \emph{$\mathcal{X}$-coordinates} arise from flags in the projective plane $\RP^2.$ This is shown in \cite{FoGo-moduli-2007} to
parameterise the space of strictly convex projective structures on $S$ with a framing at each end. This amounts to a finite branched cover of the space of strictly convex projective structures on $S$ studied by Goldman, and parameterises structures that may have finite or infinite area. The space $\Ttf(S_{g,n})$ is identified with a subvariety of $\R_{>0}^{16g-16+8n}.$ See \cite{CTT-moduli-2018} for a complete discussion of these facts.

The \emph{$\A$-coordinates} arise from flags in $\R^3.$ Indeed, they describe the lift of the developing map for a strictly convex projective structures from $\RP^2$ to $\R^3,$ and are shown to parameterise the space of finite-area structures with the additional data of a vector and covector decoration at each cusp of $S$. We denote this double decorated space by 
$\decTtf(S_{g,n}).$ We give a complete treatment in \S\ref{sec:decorated_higher_teichmueller_space}, where we also show that the space of independent decorated structures 
$\singdecTtf(S_{g,n})/\R_{>0}$ can be identified with a subset of $\decTtf(S_{g,n})$ that is a product of $n+1$ open simplices. The details for the relationship between $\A$-coordinates and lifts of developing maps to $\R^3$ are probably well known to experts.

In \S\ref{sec:outitude}, we introduce the key player in our approach, the \emph{outitude} associated to an edge of the ideal triangulation of $S.$ We show that, using $\A$--coordinates, there is a simple edge flipping algorithm resulting in the canonical ideal cell decomposition associated to a point in $\decTtf(S_{g,n}).$ 
The outitude is then used to prove that we obtain a cell decomposition of $\decTtf(S_{g,n})$ in \S\ref{sec:main_proof}. The main idea of the proof is to determine a natural product structure for each putative cell in the decomposition of $\decTtf(S_{g,n})$, and to show that each level set in this product structure is star-shaped. This is first done for the case of a triangulation (Theorem~\ref{thm:main}), and the proof for more general cell decompositions of $S$ is more involved (Theorem~\ref{thm:main2}).

In \S\ref{sec:duality}, we analyse projective duality in $\A$-coordinates, identify classical Teichm\"uller space in these coordinates and show that our cell decomposition is generally not invariant under duality. Benoist and Hulin \cite{BenoistHulin:2013} showed that $\Ttf(S_{g,n})$ is the product of classical Teichm\"uller space $\T_\hyp(S_{g,n})$ and the vector space of 
cubic holomorphic differential on the surface with poles of order at most 2 at the cusps. In particular, any cell decomposition of classical Teichm\"uller space gives rise to a cell decomposition of the spaces $\Ttf(S_{g,n})$, $\singdecTtf(S_{g,n})$ and $\decTtf(S_{g,n})$. Duality is used to show that our cell decomposition generally does not arise in this way in  \S\ref{sec:Other cell decompositions}.

Penner~\cite{Penner-Decorated-1987} describes a natural centre for each of the cells in classical Teichm\"uller space, and shows that the centres of top-dimensional cells correspond to arithmetic Fuchsian groups. In \S\ref{sec:cells}, we show that Penner's centres are also natural centres of the cells in $\decTtf(S_{g,n})$, and that they correspond to semi-arithmetic Fuchsian groups in many (but not all) cases. For instance, if the associated cell decomposition of the surface only involves polygons with an odd number of sides, then the associated group is semi-arithmetic.

The study of \emph{higher Teichm\"uller spaces} has emerged through the examples of \emph{Hitchin components} and \emph{maximal representations} (see \cite{Burger-higher-2014, Wienhard-invitation-2019} for an overview of the field and references). 
Generalisations from classical to higher Teichm\"uller theory typically have a geometric interpretation for convex real projective structures in the case of $\Ttf(S_{g,n})$ and $\decTtf(S_{g,n})$. 
The spaces $\Ttf(S_{g,n})$ and $\decTtf(S_{g,n})$ can thus be viewed as stepping stones between classical Teichm\"uller space and the higher Teichm\"uller spaces. 
It would be interesting to construct a generalisation of the outitude for arbitrary semisimple Lie groups that leads to cell decompositions of all higher Teichm\"uller spaces.\\


\textbf{Acknowledgements}

Research by L\"owe is supported by the DFG via SFB-TRR 109: ``Discretization in Geometry and Dynamics''.
Tate acknowledges support by the Australian Government Research Training Program.
Research of Tillmann is supported in part under the Australian Research Council's ARC Future Fellowship FT170100316.

\section{Preliminaries}
\label{sec:Background}

This section fixes some useful notation and terminology for ideal cell decompositions of punctured surfaces (\S\ref{subsec:icd} and \S\ref{subsec:straight triangs}), convex projective structures of finite volume on punctured surfaces (\S\ref{subsec:convex_projective_structures}); and decorations and double-decorations of such structures (\S\ref{subsec:Decorated and doubly decorated}).
For further background on these topics, the reader is referred to 
Hatcher~\cite{Hatcher-triangulations-1991}, Mosher~\cite{Mosher-tiling-1988}, Penner~\cite{Penner-Decorated-1987}, Goldman~\cite{Goldman-convex-1990, Goldman-geometric-1988}, Marquis~\cite{Marquis-espaces-2010}, Cooper and Long \cite{CoLo-Epstein-Penner-2015}, and also \cite{CTT-moduli-2018}.


\subsection{Ideal Cell Decompositions} \label{subsec:icd}

Let $S_g$ be a closed, oriented, smooth surface of genus $g.$ Let $D_1, \ldots, D_n$ be open discs on $S_g$ with pairwise disjoint closures. Choose points $p_k \in D_k.$ Then $S_{g,n} = S_g \setminus \{p_1, \ldots, p_n\}$ is a \emph{punctured surface}. A \emph{compact core} of $S_{g,n}$ is $S_{g,n}^c = S_g \setminus \cup D_k.$ We also write $S_g = S_{g,0}$ and $\mathcal{P} = \{p_1,\ldots,p_n\}$.
We always assume that $S_{g,n}$ has negative Euler characteristic.

We call $\End_k = \overline{D_k} \setminus \{p_k\}$ an \emph{end} of $S_{g,n}.$ The inclusion $\End_k \subset S_{g,n}$ of oriented surfaces defines a conjugacy class of \emph{peripheral subgroups} corresponding to $\im(\pi_1(\End_k) \to \pi_1(S_{g,n})) \cong \Z$. To simplify notation, we pick one representative of the conjugacy class and denote it by $\pi_1(\End_k) \le \pi_1(S_{g,n}).$ Any conjugate to one of the subgroups $\pi_1(\End_k)$ in $\pi_1(S_{g,n})$ is called a \emph{peripheral subgroup} of $\pi_1(S_{g,n}).$

An \emph{essential arc} in $S_{g}$ is an arc $\alpha \subset S_g$ whose endpoints are in $\mathcal{P}$ (they may coincide) and whose interior is embedded in $S_{g,n},$ with the condition that if $S_g \setminus \alpha$ has two components, then neither component is a disc meeting $\mathcal{P}$ only in the endpoints of $\alpha.$ We will also require that the intersection of $\alpha$ with the compact core $S_{g,n}^c$ is connected.

An \emph{ideal triangulation} $\tri$ of $S_{g,n}$ is the intersection with $S_{g,n}$ of a \emph{maximal} union of pairwise disjoint essential arcs in $S_g$, no two of which are ambient isotopic fixing $\mathcal{P}.$ More generally, an \emph{ideal cell decomposition} $\Delta$ of $S_{g,n}$ is the intersection with $S_{g,n}$ of a union $\Delta$ of pairwise disjoint essential arcs in $S_g$, no two of which are ambient isotopic fixing $\mathcal{P},$ and with the property that each component of $S_{g,n} \setminus \Delta$ is an open cell. We regard two ideal cell decompositions of $S_{g,n}$ as equivalent if they are isotopic via an isotopy of $S_{g}$ that fixes $\mathcal{P}$.

The essential arcs in an ideal cell decomposition cut $S_{g,n}$ into \emph{ideal polygons.} Each component of $S_{g,n}\setminus \Delta$ is the interior of an ideal polygon, and the closure of this component in $S_{g,n}$ is an ideal polygon. For example, as a subsurface of $S_{g,n}$, an ideal triangle is either a closed disc with three punctures on its boundary, or a closed disc with one puncture on its boundary and one puncture in its interior. The edges of an ideal polygon $P$ are the arcs in $\Delta$ incident with $P$, counted with multiplicity. The \emph{corners} of $P$ are the components of the intersection of $P$ with the ends of $S_{g,n}.$
For example, the polygon depicted in Figure \ref{fig:Standard_Triangulation1} has eight edges and eight corners regardless of how those edges and ideal vertices are identified in $S_{g,n}$.

Any ideal cell decomposition $\Delta$ of $S_{g,n}$ can be extended to an ideal triangulation $\Delta' \supseteq \Delta$ of $S_{g,n}$. We say that $\Delta'$ is a \emph{standard subdivision of $\Delta$} if on each polygon $P$ there exists a corner that meets every arc in $\Delta' \setminus \Delta$ that is contained in $P$. (Recall that the intersection of each arc with the compact core is assumed to be connected.) A polygon with a standard ideal triangulation is depicted in Figure~\ref{fig:Standard_Triangulation1}.

\begin{figure}
	\center
	\begin{tikzpicture}[scale=0.6]
	\draw (2, 0) node [anchor = north]{$V_0$};
	\draw (4, 0) node [anchor = north]{$V_1$};
	\draw (6, 2) node [anchor = west]{$V_2$};
	\draw (6, 4) node [anchor = west]{$V_3$};
	\draw (4, 6) node [anchor = south]{$V_4$};
	\draw (2, 6) node [anchor = south]{$V_5$};
	\draw (0, 4) node [anchor = east]{$V_6$};
	\draw (0, 2) node [anchor = east]{$V_7$};
	\draw (2, 0) -- (4, 0);
	\draw (4, 0) -- (6, 2);
	\draw (6, 2) -- (6, 4);
	\draw (6, 4) -- (4, 6);
	\draw (4, 6) -- (2, 6);
	\draw (2, 6) -- (0, 4);
	\draw (0, 4) -- (0, 2);
	\draw (0, 2) -- (2, 0);
	\draw (2, 0) edge[red] (6, 2);
	\draw (2, 0) edge[red] (6, 4);
	\draw (2, 0) edge[red] (4, 6);
	\draw (2, 0) edge[red] (6, 4);
	\draw (2, 0) edge[red] (2, 6);
	\draw (2, 0) edge[red] (0, 4);
	\end{tikzpicture}
	\caption{A standard triangulation of an $8$-gon.} \label{fig:Standard_Triangulation1}
\end{figure}
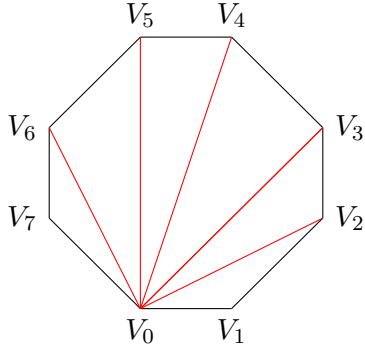

An \emph{edge flip} on an ideal triangulation consists of picking two distinct ideal triangles sharing an edge, removing that edge and replacing it with the other diagonal of the square thus formed (see Figure~\ref{fig:edge_flip}). Hatcher~\cite{Hatcher-triangulations-1991} and Mosher~\cite{Mosher-tiling-1988} showed that any two ideal triangulations of $S_{g,n}$ differ by finitely many edge flips.

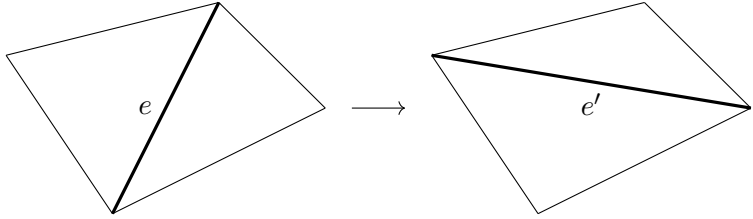
\begin{figure}
	\center
	\begin{tikzpicture}[scale=0.7]
	\draw (4, 4) -- (6, 2);
	\draw (6, 2) -- (2,0);
	\draw (2, 0) edge[very thick, "$e$" left, pos = 0.5] (4, 4);
	\draw (4, 4) -- (0, 3);
	\draw (0, 3) -- (2, 0);
	
	\draw[->] (6.5, 2) -- (7.5, 2);
	
	\draw (12, 4) -- (14, 2);
	\draw (14, 2) -- (10,0);
	\draw (8, 3) edge[very thick, "$e'$" below, pos = 0.5] (14, 2);
	\draw (12, 4) -- (8, 3);
	\draw (8, 3) -- (10, 0);
	\end{tikzpicture}
	\caption{An edge flip along $e$.} 
	\label{fig:edge_flip}
\end{figure}


\subsection{Real Convex Projective Structures} \label{subsec:convex_projective_structures}

This paper is concerned with a special class of $(\PGL_3(\R), \RP^2)$--structures on $S = S_{g,n}$, namely \emph{strictly convex projective structures of finite volume}. 

There is a natural analytic isomorphism $\PGL_3(\R) \rightarrow \SL_3(\R)$ defined by 
\[ A \mapsto (\det A)^{-1/3} \ A.\]
We may (and will) therefore work with the group $\SL_3(\R)$, and note that it is isomorphic as an analytic group with $\PGL_3(\R).$ In particular, any arguments involving continuity, topology or analyticity that arise are not affected by working with the special linear group. 

An open subset $\Omega \subset \RP^2$ is said to be \emph{properly convex} if its closure $\overline{\Omega}$ is strictly contained in an affine patch of $\RP^2$ and it is convex as a subset of that affine patch. We say that $\Omega$ is \emph{strictly convex} if it is properly convex and $\partial \overline{\Omega}$ does not contain an interval. We denote by $\SL_3(\Omega)$ the subgroup of $\SL_3(\R)$ which preserves $\Omega$.

A \emph{marked properly convex projective structure} on $S$ is a triple $(\Omega,\Gamma,\phi)$ consisting of
\begin{itemize}
\item a properly convex domain $\Omega \subset \RP^2$,
\item a discrete, torsion-free group $\Gamma < \SL_3(\Omega)$, called the \emph{holonomy group}, and
\item a diffeomorphism $\phi : S \rightarrow \sfrac{\Omega}{\Gamma}$, commonly referred to as the \emph{marking}.
\end{itemize}
A surface endowed with a properly convex projective structure is a 
\emph{properly convex projective surface}. Two marked properly convex projective structures $(\Omega_0,\Gamma_0,\phi_0)$ and $(\Omega_1,\Gamma_1,\phi_1)$ are \emph{equivalent} if there is
$M \in \SL_3(\R)$ such that 
\begin{itemize}
	\item $M \cdot \Omega_0 = \Omega_1$ \enspace ,
	\item $M \cdot \Gamma_0 \cdot M^{-1} = \Gamma_1$ and 
	\item $\phi_1$ is isotopic to $M^\ast \circ \phi_0$ , where $M^\ast$ is the diffeomorphism $\sfrac{\Omega_0}{\Gamma_0} \rightarrow \sfrac{\Omega_1}{\Gamma_1}$ induced by $M$.
\end{itemize}
The set $\Tt(S)$ of all such equivalence classes is the \emph{moduli space of marked properly convex projective structures} on $S$.

Fix once and for all a universal cover $\pi: \widetilde{S} \rightarrow S$. Let $( \Omega , \Gamma , \phi ) \in \Tt(S) $. The \emph{developing map} of this structure is a lift of $\phi$ to 
a diffeomorphism whose image is $\Omega$, denoted by
\(
\dev\co \widetilde{S} \rightarrow \Omega.
\)
The \emph{holonomy representation} $\hol\co \pi_1(S) \rightarrow \Gamma$ is the push-forward via $\dev$ of the action of $\pi_1(S)$ on $\widetilde{S}$. It follows that, having fixed a universal cover, the pair $(\dev, \hol)$ uniquely determines the marked structure $( \Omega , \Gamma , \phi ).$ The induced action of $\SL_3(\R)$ is given by $M \cdot (\dev, \hol) = (M \circ \dev, M\hol M^{-1})$ for all $M \in \SL_3(\R).$

The \emph{Hilbert metric} is a natural Finsler metric on the properly convex domain $\Omega$ that is invariant under the action of $\SL_3(\Omega)$ (the definition of this metric is irrelevant for this paper; for details see \cite{Goldman-convex-1990}). If, in addition, $\Omega$ is strictly convex then there is a unique geodesic between any two points. If $(\Omega, \Gamma, \phi)$ is a properly convex projective structure then, since $\Gamma$ acts on $\Omega$ by isometries, the Hilbert metric descends to a metric on $\Omega / \Gamma $.

Taking the Hausdorff measure with respect to this metric endows $S$ with a volume. In order to make use of  the canonical cell decomposition of $S$ due to Cooper and Long~\cite{CoLo-Epstein-Penner-2015} we restrict our attention to those strictly convex projective structures on $S$ that have finite volume. It is shown 
by Marquis~\cite{Marquis-espaces-2010} 
that this is equivalent to the requirement that for all non-trivial peripheral elements $\gamma \in \pi_1(S)$ the holonomy $\hol(\gamma)$ is conjugate to the standard parabolic, i.e.
\[
\hol (\gamma) \sim \begin{pmatrix} 1 & 1 & 0 \\ 0 & 1 & 1 \\ 0 & 0 & 1 \end{pmatrix} \enspace .
\]
The moduli space of all properly convex projective structures on $S$ having finite volume is denoted by $\Ttf(S)$.


\subsection{Straight ideal triangulations}
\label{subsec:straight triangs}

Suppose $S$ has a strictly convex real projective structure of finite volume $(\Omega, \Gamma, \phi)$.  An ideal triangulation of $S$ is \emph{straight} if each ideal edge is the image of the intersection of $\Omega$ with a projective line. Every ideal triangulation of $S$ is isotopic to a straight ideal triangulation (see \cite{TiWo-algorithm-2016}).

In the following, we use an ideal triangulation $\tri$ of $S$ to obtain parametrisations of moduli spaces. This requires us to clarify the relationship between the smooth structure on $S$ and the ideal triangulation $\tri.$ We take the following viewpoint.
The surface $S$ is given by an ideal triangulation. We identify each ideal triangle with the regular euclidean triangle of side length one and with its vertices removed, and we realise the face pairings of triangles by euclidean isometries. The resulting identification space has a natural analytic structure, and we assume that $S$ is given with this structure. This construction lifts to the universal cover of $S$. So given a real analytic surface $S$ with an ideal triangulation $\tri,$ we may assume that each ideal triangle in $\widetilde{S}$ is identified with a euclidean triangle (and hence parameterised using barycentric coordinates), and the face pairings between triangles are realised by euclidean isometries. 

Now suppose that $\dev \co \widetilde{S} \to \Omega \subset \RP^2$ is a developing map for a properly convex projective structure on $S.$ By definition, this is an analytic map, but the images of the ideal edges may not be straight in $\Omega.$ We may compose $\dev$ with an isotopy to obtain a map that is linear on each ideal triangle and continuous. We will term the resulting map a \emph{PL developing map}\index{developing map ! PL}. 

Conversely, suppose $\Omega$ is a properly convex domain and $\hol\co \pi_1(S)\to \SL(\Omega)$ is a discrete and faithful representation. Assume that 
$\dev'\co \widetilde{S} \to \Omega \subset \RP^2$ is a continuous bijection that is linear on each ideal triangle and equivariant with respect to $\hol.$ Choose a regular neighbourhood $N$ of the union of the ideal edges. We may equivariantly approximate $\dev'$ by an analytic map, only changing it in $N.$ This gives a developing map $\dev \co \widetilde{S} \to \Omega \subset \RP^2$ that agrees with $\dev'$ in the complement of $N,$ and hence is uniquely determined by $\dev'.$ Moreover, the images of the ideal endpoints of the edges will remain unchanged, and hence applying the straightening procedure to $\dev$ again results in $\dev'.$ 

The upshot of this discussion is that we can work interchangeably with developing maps and PL developing maps, and the latter are uniquely determined by the limiting endpoints of the ideal edges on the boundary of the properly convex domain.


\subsection{Decorated and doubly decorated real projective structures}
\label{subsec:Decorated and doubly decorated}

Let $\tau  = (\Omega, \Gamma, \phi ) \in  \Ttf(S)$ and denote its developing map and holonomy representation by $\dev_\tau$ and $\hol_\tau$ respectively.
Recall that we assume that $\Omega$ is contained in some affine patch $A$. There is a natural embedding 
\[
\iota : A \rightarrow \{ (x, y, 1) \in \R^3 \}  \enspace ,
\] 
which is unique up to an affine transformation of the codomain. 
We identify $\Omega$ with its image via $\iota$. 
The \emph{positive light-cone} of $\Omega$ is the set $\mathcal{L}^+ = \R_{>0} \cdot \partial \Omega \subset \{  (x, y, z) \in \R^3 \mid z > 0 \}$.

For each end $\End_i$ let $p_i$ be 
the unique fixed point of $\hol_\tau(\pi_1(\End_i))$.
Choose a light-cone representative $V_{i} \in \mathcal{L}^+$ that projects to $p_i$, i.e. $\left[ V_i \right] = p_i$.
The union of $\Gamma$-orbits
\begin{equation}\label{def:vec_dec}
\vec_dec = \bigcup_{i=1}^n \Gamma \cdot V_i
\end{equation}
is called a \emph{vector decoration}, or simply a \emph{decoration} of $\tau$.

With the above notation, for each $V_i$ we additionally choose a covector $\covtx_i \in (\R^3)^*$ whose kernel projects to a supporting line of $\overline{\Omega}$ and such that $\covtx_i \cdot V_i = 0$ and $\covtx_i \cdot W > 0$ for all $W \in \vec_dec \setminus \{ V_i \}$.
The union of $\Gamma$-orbits 
\begin{equation}\label{def:covec_dec}
\covec_dec = \bigcup_{i=1}^n \covtx_i \cdot \Gamma^{-1}
\end{equation}
is called a \emph{covector decoration} of $\tau \in \Ttf(S)$.  

A \emph{double decoration} of a marked properly convex projective structure $\tau \in \Ttf(S)$ is a tuple $( \covec_dec , \vec_dec )$ consisting of a vector decoration $\vec_dec$ and a covector decoration $\covec_dec$ of $\tau$.

Two doubly-decorated structures $( \Omega_0 , \Gamma_0 , \phi_0 , ( \covec_dec_0 , \vec_dec_0 ) ) $ and $ ( \Omega_1 , \Gamma_1 , \phi_1 , ( \covec_dec_1 , \vec_dec_1 ) ) $ on $S$ are equivalent if $(\Omega_0, \Gamma_0, \phi_0)$ and $(\Omega_1, \Gamma_1, \phi_1)$ are equivalent as marked convex projective structures and the transformation 
$M \in \SL_3(\R)$ that induces this equivalence also satisfies the 
condition
\[
( \covec_dec_0 \cdot M^{-1} ,  M \cdot \vec_dec_0  ) = (\covec_dec_1, \vec_dec_1)  \enspace.
\]
The set of doubly-decorated equivalence classes of marked real convex projective structures on $S$ is called the \emph{moduli space of doubly-decorated convex real projective structures} and is denoted $\decTtf(S)$. 

The \emph{mapping class group} of $S$ is the group of isotopy classes of diffeomorphisms of $S$, denoted by
\[
\MCG (S) = \text{Diff}^+(S) ~ / ~ \text{Diff}_0(S) \enspace ,
\] 
where $\text{Diff}^+(S)$ denotes the set of orientation preserving diffeomorphisms of $S$ and $\text{Diff}_0(S)$ denotes the subset of those diffeomorphisms which are isotopic to the identity.
$\MCG (S)$ naturally acts on $\decTtf (S)$ by precomposition of the marking, i.e.
\begin{align*}
\MCG (S) \times \decTtf(S) & \rightarrow \decTtf(S) \\
\left( \mu , (\Omega, \Gamma, \phi, (\covec_dec, \vec_dec)) \right) & \mapsto (\Omega, \Gamma, \phi \circ \mu^{-1}, (\covec_dec, \vec_dec)) \enspace .
\end{align*}
If $\mu \in \MCG(S)$ then we refer to the induced map $\mu^*: \decTtf(S) \rightarrow \decTtf(S)$ as the \emph{change of marking} map induced by $\mu$.

We now describe a natural topology on $\decTtf(S)$. Let $C_w^{\infty}(\widetilde{S}, \RP^2)$ denote the space of smooth maps from $\widetilde{S}$ to $\RP^2$ with the weak topology (see Hirsch~\cite{Hirsch-differential-1976}). Since $\SL_3(\R)$ acts by diffeomorphisms on $\RP^2$ there is a natural action of $\SL_3(\R)$ on $C_w^{\infty}(\widetilde{S}, \RP^2 )$ by postcomposition. We give $C_w^{\infty}(\widetilde{S}, \RP^2) /  \SL_3(\R)$ the quotient topology. We now make use of the map
\begin{align*}
\Ttf(S) & \rightarrow C_w^{\infty}(\widetilde{S}, \RP^2) / \SL_3(\R) \enspace , \\
(\Omega, \Gamma, \phi) & \mapsto [\dev] \enspace,
\end{align*}
to pull back a natural topology on $\Ttf(S)$. Having fixed a preferred peripheral subgroup for each end,
the double decoration of a convex projective structure of finite volume amounts to a single choice of vector in $\R^3$ and covector in $(\R^3)^\ast$. Therefore we have an inclusion
\begin{align*}
\decTtf(S) & \hookrightarrow \Ttf(S) \times \left( (\R^3)^\ast \times \R^3 \right)^n \enspace .
\end{align*}
The right hand side is given the product topology, and $\decTtf(S)$ is given the subspace topology from the inclusion. This topology does not depend on the choice of peripheral subgroup for each end.


\section{A parameterisation of the moduli space}
\label{sec:decorated_higher_teichmueller_space}


Following Fock and Goncharov~\cite{FoGo-moduli-2006}, this section gives an explicit parameterisation of the moduli space of marked doubly-decorated convex real projective structures on $S_{g,n}$ using an ideal triangulation of the surface. The parameters are called \emph{$\A$-coordinates}. We spell out a number of details that are implicit in \cite{FoGo-moduli-2006, FoGo-moduli-2007}.
The main result of this section is Theorem~\ref{thm:A-coords} in which it is shown that the space of $\A$-coordinates is naturally homeomorphic to $\decTtf(S)$. 
We also analyse the action of the mapping class group and use this to give the moduli space a natural analytic structure (see Corollary~\ref{cor:analytic structure}).
We conclude with a discussion of the relationship between the two coordinate systems given by Fock and Goncharov in \cite{FoGo-moduli-2006}, the $\A$-coordinates and the $\X$-coordinates, and an identification of the space of independent decorated structures $\singdecTtf(S_{g,n})/\R_{>0}$ with a subset of $\decTtf(S_{g,n})$ that is a product of $n+1$ open simplices.

For convenience, we often identify a triangle or (oriented) edge with its parameter.


\subsection{Concrete decorated triangulations} \label{subsec:concretes}

A \emph{concrete flag} is a pair $(r, V)$, where $V \in \R^3 \setminus \{ 0 \} $ 
and $ \covtx \in (\R^3)^{\ast}\setminus \{ 0 \}$ satisfy $\covtx \cdot V = 0$. We write elements of $\R^3$ as column vectors and elements of $(\R^3)^{\ast}$ as row vectors, and call them \emph{vectors} and {covectors} respectively.
An \emph{affine flag} is the projectivisation of a concrete flag.

A \emph{concrete decorated triangle} is a pair $(R, C) \in \GL_3(\R) \times \GL_3(\R)$ such that
\begin{itemize}
\item $R \cdot C$ is positive counter-diagonal, that is to say its diagonal entries are zero and all other entries are positive, and
\item $\det (C) > 0$.
\end{itemize}
A simple calculation shows that positive counter-diagonal $3 \times 3$ real matrices have positive determinant.

Let $(R, C)$ be a concrete decorated triangle.
The \emph{vertices} of $(R, C)$ are the columns of $C$ and the \emph{covertices} of $(R, C)$ are the rows of $R$. Thereby we identify a concrete decorated triangle with an ordered choice of three concrete flags $(\covtx_0,\vtx_0), (\covtx_1,\vtx_1)$ and  $(\covtx_2,\vtx_2)$, called its \emph{concrete decorated vertices}, such that
\begin{itemize}
\item $\covtx_i \cdot \vtx_j > 0$ , for $i \neq j$, and
\item $\det ( \vtx_0 \mid \vtx_1 \mid \vtx_2 ) >0 $ .
\end{itemize}
where $\mid$ denotes concatenation of columns, so $C = ( \vtx_0 \mid \vtx_1 \mid \vtx_2)$. To simplify notation we denote the matrix $R$ by $( \covtx_0 \mid\mid \covtx_1 \mid \mid \covtx_2)$ where $\mid \mid $ denotes concatenation of rows. The \emph{hull} of $(R, C)$ is the convex hull of its vertices, denoted $\hull(R, C)$. This is a flat triangle in $\R^3$ and it is endowed with a canonical orientation by the ordering of its vertices, induced by the columns of $C$. Given that $\hull(R, C)$ is a flat triangle we may refer to an ordered pair of two columns of $C$ as an \emph{oriented edge} of $(R, C)$.

\begin{figure}
	\center
	\begin{tikzpicture}[scale=0.8]
	\draw (5.5, 4) node [anchor = south east]{$\vtx_2$};
	\draw (8, 2) node [anchor = west]{$\vtx_1$};
	\draw (2.8, 0) node [anchor = north west]{$\vtx_0$};
	\draw [-{Latex[length=2mm, width=2mm]}, very thick] ([shift = (30:1cm)] 4.1, 1.7) arc (180:400:0.5cm);
	\draw (5, 4) -- (8, 2);
	\draw (8, 2) -- (3,0);
	\draw (3, 0) edge["$e$" left, pos = 0.5] (5, 4);
	\draw (2, 4) -- ( 7, 4 );
	\draw (1, 0) -- ( 6, 0 );
	\draw (7, 0) -- ( 9, 4 );
	\draw (2, 4) edge["$\ker(\covtx_2)$" above, pos = 0.1]  ( 7, 4 );
	\draw (1, 0) edge["$\ker(\covtx_0)$" below, pos = 0.1]  ( 6, 0 );
	\draw (7, 0) edge["$\ker(\covtx_1)$" right, pos = 0.1]  ( 9, 4 );
	\end{tikzpicture}
	\caption{An affine image of a concrete decorated triangle $(R, C)$. } \label{fig:positive_parameters}
\end{figure}
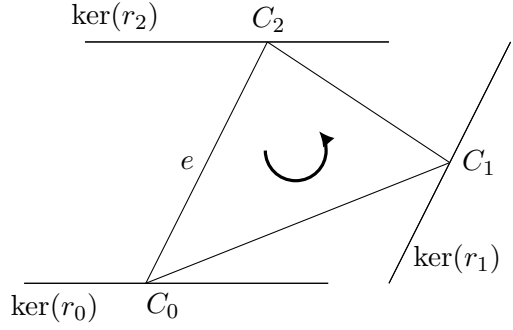

There is a natural action of $\SL_3(\R)$ on the set of concrete decorated triangles defined as follows. Let $M \in \SL_3(\R)$ and let $(R, C)$ be a concrete decorated triangle. Then we define
\[
M \cdot (R, C) := ( R \cdot M^{-1} , M \cdot C) \enspace .
\]

\begin{lem} \label{lem:abstract_triangle_equivalence}
Let $(R, C)$ and $(R', C')$ be concrete decorated triangles. The following are equivalent.
\begin{enumerate}
\item There exists $M \in \SL_3(\R)$ such that $M \cdot (R, C) = (R', C')$.
\item $\det(C) = \det(C')$ and $R\cdot C = R' \cdot C'$.
\end{enumerate}
\end{lem}
\begin{proof}
Clearly if $M \cdot (R, C) = (R', C')$ then $C' = M \cdot C$ so $\det(C') = \det(C)$. Moreover, in this case
\[
R' \cdot C' = R \cdot M^{-1} \cdot M \cdot C = R \cdot C \enspace .
\]
Conversely, suppose that $\det(C) = \det(C')$ and $ R \cdot C = R' \cdot C'$. We define $M := C' C^{-1}$. Then clearly $M \in \SL_3(\R)$, and we have $M \cdot C = C'$ as well as 
\[
R \cdot M^{-1} = R \cdot C \cdot C'^{-1} =  R' \cdot C' \cdot C'^{-1} =  R' \enspace .
\]
Therefore $M \cdot (R, C) = (R', C')$ as required.
\end{proof}

Given a concrete decorated triangle $(R, C)$ we refer to $\det(C)$ as the \emph{triangle parameter} assigned to 
(R, C).
The value $\covtx_i \cdot \vtx_j$ is the \emph{edge parameter} associated to the oriented edge $(~ \vtx_i \mid \vtx_j ~ )$.
To each concrete decorated triangle we may therefore 
assign one triangle parameter to its (oriented) hull and one edge parameter to each of its six oriented edges.
Lemma~\ref{lem:abstract_triangle_equivalence} shows that a concrete decorated triangle is uniquely determined by these seven values up to the action of $\SL_3(\R).$

\begin{lem} \label{lem:next_vertex}
Fix a concrete decorated triangle $(R, C)$, where
\[
C= ( \vtx_0 \mid  \vtx_1  \mid \vtx_2 ) \quad \text{and} \quad R = ( \covtx_0 \mid\mid \covtx_1 \mid \mid \covtx_2  ) \enspace .
\]
Then each choice of values $e_{03}, e_{23}, A_{023} \in \R$ determines a unique vector $\vtx_3 \in \R^3$ satisfying
\begin{enumerate}[label=(\roman*)]
\item $e_{03} = \covtx_0 \cdot \vtx_3$ ,
\item $e_{23} = \covtx_2 \cdot \vtx_3$ and
\item $A_{023} = \det(\vtx_0 \mid \vtx_2 \mid \vtx_3)$ .
\end{enumerate}
The so obtained map $(e_{03}, e_{23}, A_{023}) \mapsto \vtx_3$ defines a linear automorphism of $\R^3$.
If $X$ denotes a nonzero vector in $\ker ( \covtx_0 ) \cap \ker ( \covtx_2 ) $ such that $\det(C_0 \mid C_2 \mid X) > 0$, then the values $e_{03}$, $e_{23}, A_{023}$ are strictly positive if and only if $C_3$ is contained in the interior of the positive cone 
\[\pos (C_0, C_2, X) = \{ a C_0 + b C_2 + c X \mid a, b, c >0\}.\]
\end{lem}
\begin{proof}
First note that condition (iii) may be equivalently expressed by using the Euclidean cross product as $A_{023} = \vtx_3 \cdot (\vtx_0 \times \vtx_2)$.
Thus $\vtx_3$ is determined by the 3$\times$3 linear system
\[
\underbrace{
\left(
\begin{array}{ccc}
\rule[.5ex]{2.5ex}{0.5pt} & \covtx_0 & \rule[.5ex]{2.5ex}{0.5pt} \\
\rule[.5ex]{2.5ex}{0.5pt} & \covtx_2 & \rule[.5ex]{2.5ex}{0.5pt} \\
\rule[.5ex]{2.5ex}{0.5pt} & \vtx_0 \times \vtx_2 & \rule[.5ex]{2.5ex}{0.5pt}
\end{array}
\right) }_{=: M}
 \cdot \vtx_3 = 
\left(
\begin{array}{c}
e_{03}\\
e_{23}\\
A_{023}
\end{array}
\right) \enspace .
\]
By definition of concrete decorated triangles $M$ is invertible with inverse 
\[
M^{-1} =
\left(
\begin{array}{ccc}
\tfrac{1}{e_{03}} & 0 & 0 \\
0 & \tfrac{1}{e_{23}} &  0 \\
0 & 0 & \tfrac{1}{e_{03} e_{23}}
\end{array}
\right) \cdot
\left(
\begin{array}{ccc}
\rule[-1ex]{0.5pt}{2.5ex} & \rule[-1ex]{0.5pt}{2.5ex} & \rule[-1ex]{0.5pt}{2.5ex} \\
\vtx_2 & \vtx_0 & \covtx_2 \times \covtx_0 \\
\rule[-1ex]{0.5pt}{2.5ex} & \rule[-1ex]{0.5pt}{2.5ex} & \rule[-1ex]{0.5pt}{2.5ex}
\end{array}
\right) \enspace .
\]
This proves the first statement.
The second claim follows since in fact $X$ can be chosen as $\covtx_2 \times \covtx_0$.
Hence, in case of positivity of $e_{03}, e_{23}, A_{023}$ the vector $\vtx_3=M^{-1}\cdot (e_{03}, e_{23}, A_{023})$ is a positive linear combination of the columns of $M^{-1}$.
\end{proof}

	\begin{rem}
		In the situation of Lemma~\ref{lem:next_vertex}. Figure~\ref{fig:edge_parameter} depicts an affine plane in $\RP^2$ containing the projections of the $\vtx_i$'s and $X$. The condition that $e_{03}$, $e_{30}$ and $A_{023}$ are all strictly positive is equivalent to the condition that $\vtx_3$ lies inside the convex hull of the projections of $\vtx_0$, $\vtx_2$ and $X$. 
	\end{rem}

\begin{figure}[ht]
	\centering
	\begin{tikzpicture}[scale=0.8]
	\draw [-{Latex[length=2mm, width=2mm]}, very thick] ([shift = (30:1cm)] 1, 1.5) arc (180:400:0.5cm);
	\draw [-{Latex[length=2mm, width=2mm]}, very thick] ([shift = (30:1cm)] 4.3, 1.9) arc (180:400:0.5cm);
	
	\draw (0, 2) -- (5, 4)
	node [circle, fill=black, inner sep = 0pt, minimum size = 4pt, label = left: {$C_3$}, pos = 0] () {};
	
	\draw (5, 4) -- (8, 2)
	node [circle, fill=black, inner sep = 0pt, minimum size = 4pt, label = above: {$C_2$}, pos = 0] () {};
	
	\draw (8, 2) -- (3,0)
	node [circle, fill=black, inner sep = 0pt, minimum size = 4pt, label = right: {$C_1$}, pos = 0] () {};
	
	\draw (3, 0) -- (0, 2);
	
	\draw (3, 0) -- (5,4)
	node [circle, fill=black, inner sep = 0pt, minimum size = 4pt, label = below: {$\vtx_0$}, pos = 0] () {};
	
	\draw ( 4, -0.2148257142) -- (-5, 1.7142857143) 
	node [ label = below: {$\ker(r_0)$}, pos = 0.5] () {}
	node [circle, fill=black, inner sep = 0pt, minimum size = 4pt, label = below: {$X$}, pos = 0.883] () {};
	
	\draw (6, 4.277777) -- (-5, 1.222222)
	node [ label = above: {$\ker(r_2)$}, pos = 0.5] () {};
	
	\end{tikzpicture}
	\caption{The location of $\vtx_3$ is determined by the values of $\covtx_0 \cdot \vtx_3$, $\covtx_2 \cdot \vtx_3$ and $\det(\vtx_0 \mid \vtx_2 \mid \vtx_3)$.} \label{fig:edge_parameter}
\end{figure}
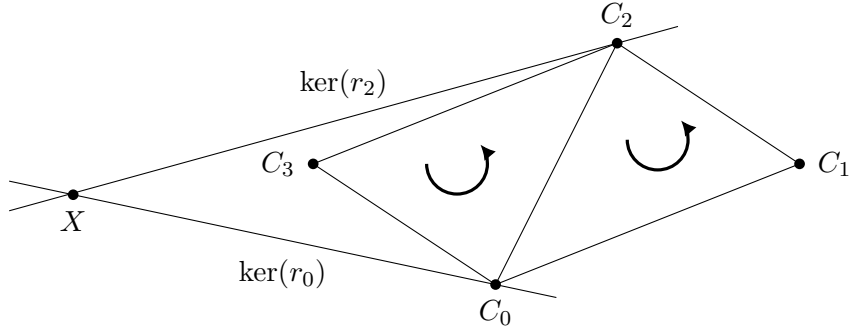

A \emph{concrete decorated edge} of a concrete decorated triangle $(R, C)$ is an \emph{unordered} pair of concrete decorated vertices of $(R, C).$
Figure \ref{fig:edge_parameter} depicts two concrete decorated triangles
\begin{align*}
( R, C ) &= (~ ( \covtx_0 \mid \mid \covtx_1 \mid \mid \covtx_2) ~,~ ( \vtx_0 \mid \vtx_1 \mid \vtx_2 )~ ) \enspace ,\\
( R', C') &= ( ~( \covtx_0 \mid \mid \covtx_2 \mid \mid \covtx_3)~ , ~( \vtx_0 \mid \vtx_2 \mid \vtx_3 )~ ) \enspace . 
\end{align*}
Both contain the concrete decorated edge 
$\{~ (\covtx_0, \vtx_0) ~,~ (\covtx_2, \vtx_2) ~\}.$
In this case we say that the concrete decorated triangles $( R, C )$ and $( R', C') $ \emph{share an edge}.

Note here that, by definition, the two adjacent triangles $( R, C )$ and $( R', C') $ induce opposite orientations on the concrete decorated edge they share.
Geometrically the latter condition ensures that the plane through the origin  in $\R^3$ containing $C_0$ and $C_1$ separates $C_2$ and $C_3$.

A collection $\Lambda = \{ (R, C)_i \}_{i \in I}$ of concrete decorated triangles is called a \emph{concrete decorated triangulation} if:
\begin{enumerate}
\item Each concrete decorated edge of $\Lambda$ is shared by exactly two concrete decorated triangles of $\Lambda$, and;
\item Any two concrete decorated triangles $(R, C)$ and $(R', C')$ in $\Lambda$ are connected by a finite sequence $(R, C) = (R_0, C_0)$, $(R_1, C_1)$, \ldots , $(R_k, C_k) =(R', C') $ such that consecutive members share a concrete decorated edge.
\end{enumerate}

Recall that the hull of a concrete decorated triangle is the flat triangle in $\R^3$ spanned by its vertices. 
A \emph{facet} of $\Lambda$ is the hull of a concrete decorated triangle in $\Lambda$. 
The \emph{hull} of $\Lambda$ is defined as union of its facets, i.e.~
\[
\hull (\Lambda) = \bigcup_{ i \in I } \hull (R, C)_i \enspace .
\]
It follows from the definition of $\Lambda$ that $\hull (\Lambda)$ is a connected subset of $\R^3.$

\begin{lem}[Edge flip] \label{lem:edge_flip}
Suppose that the following pairs are concrete decorated triangles.
\begin{align*}
( R, C ) &= (~ ( \covtx_0 \mid \mid \covtx_1 \mid \mid \covtx_2) ~,~ ( \vtx_0 \mid \vtx_1 \mid \vtx_2 )~ ) \enspace ,\\
( R', C') &= ( ~( \covtx_0 \mid \mid \covtx_2 \mid \mid \covtx_3)~ , ~( \vtx_0 \mid \vtx_2 \mid \vtx_3 )~ ) \enspace .
\end{align*}
Then the following pairs are also concrete decorated triangles.
\begin{align*}
( S, D ) = (~ (\covtx_3 \mid \mid \covtx_1 \mid \mid \covtx_2 ) ~, ~(\vtx_3 \mid \vtx_1 \mid \vtx_2 )~ ) \enspace , \\
( S', D' ) = ( ~(\covtx_0 \mid \mid \covtx_1 \mid \mid \covtx_3 )~ , ~(\vtx_0 \mid \vtx_1 \mid \vtx_3 )~ ) \enspace .
\end{align*}
(Note the orientations of $(S, D)$ and $(S^\prime, D^\prime)$ induced by $(R, C)$ and $(R^\prime, C^\prime)$.)
\end{lem}
\begin{proof}
We need to show positivity of the new triangle parameters $\det ( D )$ and $\det ( D' )$ as well as positivity of the new edge parameters $\covtx_1 \cdot \vtx_3$ and $\covtx_3 \cdot \vtx_1$.
The fact that $(R, C)$ is a concrete decorated triangle ensures that $\{ \vtx_0, \vtx_1, \vtx_2 \}$ is a basis. Therefore we may write $\vtx_3 = \alpha \vtx_0 + \beta \vtx_1 + \gamma \vtx_2$. Since $\det(C)>0$ and $\det(C') >0$ we have $ \beta <0$. Therefore, since $\covtx_0 \cdot \vtx_3 >0$ and $\covtx_2 \cdot \vtx_3 > 0 $ we have $\alpha >0$ and $\gamma > 0$. It follows that
\[
\covtx_1 \cdot \vtx_3 = \covtx_1 \cdot ( \alpha \vtx_0 + \beta \vtx_1 + \gamma \vtx_2 ) = \alpha \covtx_1 \cdot \vtx_0 + \gamma \covtx_1 \cdot \vtx_2 > 0 \enspace .
\]
By symmetry we also get $\covtx_3 \cdot \vtx_1 > 0$.
Furthermore, the triangle parameters can be expressed as
\begin{align*}
\det(D) &= \det( \alpha \vtx_0 + \beta \vtx_1 + \gamma \vtx_2  \mid \vtx_1 \mid \vtx_2 ) = \alpha \det(C) > 0  \enspace , \\
\det(D') &= \det(\vtx_0 \mid \vtx_1 \mid \alpha \vtx_0 + \beta \vtx_1 + \gamma \vtx_2) = \gamma \det(C) > 0 \enspace .
\end{align*}
This completes the proof.
\end{proof}

Let $\Lambda$ be a concrete decorated triangulation. Using the notation from the Lemma~\ref{lem:edge_flip}, the process of removing the two adjacent concrete decorated triangles $(R, C), (R', C')$ from $\Lambda$ and inserting the concrete decorated triangles $(S, T), (S', T')$, resulting in a new concrete decorated triangulation $\Lambda'$, is called an \emph{edge flip} on $\Lambda$ 
along the edge $\{ ~(r_0, C_0)~, (r_2, C_2)~ \}$.

\begin{lem} \label{lem:one_side_of_kernel}
Let $\Lambda = \{ (R, C)_i \}_{i \in I}$ be a concrete decorated triangulation.
Suppose $(\covtx, \vtx)$ and $(\covtx', \vtx')$ are distinct concrete decorated vertices of concrete decorated triangles in $\Lambda.$ Then $\covtx \cdot \vtx' > 0.$
\end{lem}

\begin{proof}
Let $(R, C) \in \Lambda$ be a concrete decorated triangle containing the concrete flag $(\covtx,\vtx)$.  Throughout this proof we will say that $\vtx'$ is at distance $0$ from $\vtx$ if $\vtx'$ and $\vtx$ are contained in a common triangle in $\Lambda$. If $\vtx'$ and $\vtx$ are not at distance $0$ then we will say that $\vtx'$ and $\vtx$ are at distance $m>0$ if $m$ is the smallest integer such that $\vtx'$ is contained in a common triangle with a vertex $\vtx_k$ such that $\vtx$ and $\vtx_k$ are at distance $m-1$.
We proceed by induction on the minimal possible distance between $\vtx$ and $\vtx'$ in any concrete decorated triangulation such that $\covtx \cdot \vtx' \leq 0$ . 

If $\vtx'$ and $\vtx$ are at distance $0$ then $\covtx \cdot \vtx' > 0$ since $\vtx$ and $\vtx'$ are contained in a common concrete decorated triangle in $\Lambda$.
In case $\vtx'$ and $\vtx$ are at distance $1$ it follows from Lemma~\ref{lem:edge_flip} that we may perform an edge flip on the common edge of the triangles to obtain a new concrete decorated triangulation in which the vertices $\vtx$ and $\vtx'$ are contained in a concrete decorated triangle. Therefore $\covtx \cdot \vtx' > 0$ as required.
Now suppose by induction that $\vtx'$ is at distance $m$ from $\vtx$. Then $\vtx'$ belongs to an ideal triangle whose vertices are $\vtx'$, $\vtx_k$ and $\vtx_\ell$ such that $\vtx_k$ and $\vtx_\ell$ are at distance $m-1$ from $\vtx_0$. By performing an edge flip on the edge whose vertices are $\vtx_k$ and $\vtx_\ell$, we obtain a new concrete triangulation $\Lambda'$ in which $\vtx'$ is at distance $m-1$ from $\vtx$. By induction it must be the case that $\covtx \cdot \vtx' > 0$.
\end{proof}

\begin{cor} \label{lem:positive_parameters}
	Let $\Lambda$ be a concrete decorated triangulation. The restriction of the quotient map $\pi : \R^3 \setminus \{ 0 \} \rightarrow \RP^2$ to $\hull ( \Lambda )$ is a homeomorphism onto its image. Moreover, the image of $\pi|_{ \hull( \Lambda ) }$ is properly convex, and, in particular, the interior of $\hull ( \Lambda )$ is an open disc.
\end{cor}
\begin{proof}
	Let $(R, C) = (~(\covtx_0 \mid \mid \covtx_1 \mid \mid \covtx_2 )~,~ (\vtx_0 \mid \vtx_1 \mid \vtx_2 ) ~)$   be a concrete decorated triangle in $\Lambda$ as depicted in Figure \ref{fig:positive_parameters}. 
	Lemma~\ref{lem:one_side_of_kernel} shows that if $V$ is any vertex other than $\vtx_0$ of any concrete decorated triangle in $\Lambda$, then $\covtx_0 \cdot V > 0$. 
	Applying this to every covector $\covtx_i$ of $(R , C)$ it follows that $\hull ( \Lambda )$ is contained in the intersection of three distinct half-spaces and therefore $\pi ( \hull ( \Lambda ))$ is strictly contained in an affine patch. 
	Moreover, by Lemma~\ref{lem:next_vertex}, $\pi ( \hull ( \Lambda ))$ is equal to the projection of the convex cone generated by the vertices of $\Lambda$ and will therefore appear as a convex set inside that affine patch, completing the proof.
\end{proof}

A \emph{decorated isomorphism} between two concrete decorated triangulations $ \Lambda $ and $ \Lambda' $ is a bijection $\phi : \Lambda \rightarrow \Lambda'$ induced by some $M \in \SL_3 ( \R ) $ such that $\phi ( (R, C) ) = M \cdot (R, C)$ for all $(R, C) \in \Lambda$. A decorated isomorphism determines a bijection between the concrete decorated vertices of $\Lambda$ and $\Lambda'$. 

\begin{lem} \label{lem:dec_iso}
There is a decorated isomorphism between the concrete decorated triangulations $\Lambda$ and $\Lambda'$  if and only if there exists a bijection $ \phi : \Lambda \rightarrow \Lambda' $ such that for every concrete decorated triangle $ ( R , C ) \in \Lambda $, if $ \phi ( (R, C) ) = ( R' , C' ) $ then 
\begin{itemize}
\item $ \det ( C ) = \det ( C' ) $ and
\item $ R \cdot C = R' \cdot C' $.
\end{itemize}
\end{lem}
\begin{proof}
The forwards direction is an immediate consequence of the definition of a decorated isomorphism together with Lemma~\ref{lem:abstract_triangle_equivalence}. The backwards direction follows from Lemma~\ref{lem:abstract_triangle_equivalence} and the uniqueness part of Lemma~\ref{lem:next_vertex}.
\end{proof}


\subsection{The $\mathcal{A}$-coordinates} \label{subsec:A-coords}

We now describe Fock and Goncharov's $\A$-coordinates~\cite{FoGo-moduli-2006} and give a self-contained proof showing that they parametrise doubly-decorated real projective structures.

Let $\Delta$ be an ideal triangulation of $S= S_{g,n}.$ and $x = (\Omega, \Gamma, \phi, (\vec_dec, \covec_dec)) \in \decTtf(S).$ 
Lift $\Delta$ to an ideal triangulation $\widetilde{\Delta}$ of the fixed universal cover of $\widetilde{S}$. 
We may then push forward $\widetilde{\Delta}$ to a straight ideal triangulation of $\Omega$ via the developing map. 
Finally, we lift the ideal triangulation of $\Omega$ to $\R^3$ such that ideal vertices $p \in \partial \Omega$
are lifted to the corresponding light-cone representatives $V_p$ in the vector decoration $\vec_dec$. 
Together with the covector decoration $\covec_dec$, we claim that this induces a concrete decorated triangulation, denoted by $\Lambda$. We note here 
that $\Lambda$ is $\Gamma$-invariant by construction.

To confirm that $\Lambda$ is a concrete decorated triangulation we first show that the triangle parameters and edge parameters of its constituent decorated triangles are positive real numbers. 
Let $(R, C)$ be a triple of decorated vectors in $\Lambda$ whose vectors lift the vertices of an ideal triangle in $\dev(\widetilde{S})$. It must be the case that the columns of $C$ are linearly independent. Indeed otherwise $\dev$ must have sent the three ideal vertices of an ideal triangle in $\widetilde{S}$ to a single projective line in $\RP^2$. In this case the interior of that triangle must be empty in $\RP^2$ so $\dev$ must not be a homeomorphism between $\widetilde{S}$ and $\Omega$.
When $\Lambda$ is constructed we use the fact that $\Omega$ inherits an orientation from $\widetilde{S}$ to impose on ordering on the vertices of $(R, C)$ such that the associated triangle parameter is positive, i.e. $\det(C) > 0$.

Since $\Omega$ is properly convex and the boundary of $S$ is empty so the arcs of $\widetilde{\Delta}$ are strictly on the interior of $\Omega$, each covector $\covtx_p \in \covec_dec$ in the covector decoration satisfies that $\covtx_p \cdot V_q > 0$ for all parabolic fixed points $q\neq p \in \partial \Omega$.  
Therefore all of the edge parameters $\Lambda$ are positive. 
Thus, $x$ induces a canonical assignment of a triangle parameters to every ideal triangle of $\Lambda$ and an edge parameters to every oriented edge of $\Lambda$. Note that, by construction, the interior of $\pi(\hull(\Lambda))$ equals $\Omega$.

Let $x_i = (\Omega_i, \Gamma_i, \phi_i, (\covec_dec_i, \vec_dec_i) ) \in \decTtf(S)$ for $i=0, 1$. Denote by $\Lambda_i$ the concrete decorated triangulation induced by $x_i$. Suppose that there is a $\hol$-equivariant decorated isomorphism $M: \Lambda_0 \rightarrow \Lambda_1$. That is, $M$ is an isomorphism such that for all concrete decorated triangles $(R, C)$ of $\Lambda_0$ we have
\[
\hol_0(\gamma) \cdot (R, C) = M^{-1} \cdot \hol_1(\gamma) \cdot M \cdot (R, C) \enspace ,
\]
where $\hol_i$ is the holonomy representation determined by $x_i$. Since $\Omega_i = \hull(\Lambda_i)$ we have the following, where $\dev_i$ is the PL developing map determined by $x_i$,
\[
M \cdot \dev_0( z ) = \dev_1 ( z ) \hspace{4mm} \forall \hspace{2mm} z \in \widetilde{S} \enspace .
\]
This shows that $\Lambda_0$ is decorated-isomorphic to $\Lambda_1$ via a hol-equivariant isomorphism if and only if $x_0$ and $x_1$ are equivalent as doubly-decorated convex projective structures.

Consider a doubly-decorated real convex projective structure $x = ( \Omega , \Gamma , \phi, (\covec_dec , \vec_dec) ) \in \decTtf(S)$. 
As explained at the beginning of this subsection, $x$ determines a lift of $\Delta$ to a concrete decorated triangulation $\Lambda \subset \R^3.$

Since $\Lambda$ is $\Gamma$-invariant, it follows from Lemma \ref{lem:abstract_triangle_equivalence} that the triangle (resp. edge) parameters of $\Lambda$ descend to an assignment of positive real numbers to the ideal triangles (resp. oriented edges) of $\Delta \subset S$.
Let $\A_\Delta \cong \R_{>0}^{T+\underline{E}}$ denote the set of all possible assignments of positive real numbers to the set of ideal triangles $T$ and the set of oriented edges 
$\underline{E}$ of $\Delta$ endowed with its natural analytic structure.
We have a well-defined map
\begin{equation} \label{eq:homeo_A_coords}
\psi_{\Delta}\co \decTtf(S) \rightarrow \A_{\Delta} \cong \R^{T+ \underline{E}}_{>0} \enspace ,
\end{equation}
which assigns to every ideal triangle and every oriented edge in $\Delta$ the triangle and edge parameters 
inherited from $\Lambda$. We refer to this as Fock and Goncharov's \emph{$\A$-coordinates} on $\decTtf(S)$ with respect to an ideal triangulation $\Delta$.
If there is no ambiguity as to the ideal triangulation to which we refer, then we write $\A = \A_\Delta$ and $\psi=\psi_{\Delta}$. 
\begin{thm} \label{thm:A-coords}
The map $\psi_{\Delta} \co \decTtf(S) \rightarrow \A_{\Delta} \cong \R^{T+ \underline{E}}_{>0}$ is a homeomorphism. 
\end{thm}
\begin{proof}
We first show that $\psi=\psi_{\Delta}$ is injective. Suppose $x_i := ( \Omega_i , \Gamma_i , \phi_i, (\covec_dec_i , \vec_dec_i) )$ is a doubly-decorated real projective structure on $S$ for $i=0, 1$. Denote the respective developing maps and holonomy representations by $\dev_i$ and $\hol_i$ and denote the respective concrete decorated triangulations in $\R^3$ by $\Lambda_0$ and $\Lambda_1$.

Suppose that $\psi ( x_0 ) = \psi ( x_1 )$. 
This ensures that for every ideal triangle $t \in \widetilde{S}$, the triangle parameter assigned to $\dev_0(t)$ is the same as that assigned to $\dev_1(t)$. 
Similarly, if $e$ is an oriented edge of $\widetilde{S}$, the oriented edge parameter of $\dev_0(e)$ is the same as that of $\dev_1(e)$. 
It follows from Lemma \ref{lem:dec_iso} that there is a decorated isomorphism $M : \Lambda_0 \rightarrow \Lambda_1$. 
The fact that $M$ is a $\hol$-equivariant decorated isomorphism follows from the fact that $\Gamma_0$ and $\Gamma_1$ act by decorated automorphisms on $\Lambda_0$ and $\Lambda_1$ respectively, which is also immediate from Lemma \ref{lem:dec_iso}. 
It follows from the discussion at the end of  \S\ref{subsec:A-coords} that $x_0 = x_1 \in \decTtf(S)$, hence $\psi$ is injective.

We next show that $\psi$ is continuous. Recall that the topology on $\decTtf(S)$ is uniquely determined by the natural topologies on the spaces of developing maps and the vector decorations. 
Moreover, the $\A$-coordinates $\psi(x_i)$ are continuous functions of the vector and covector decorations. 
Therefore a small perturbation of $\dev$ or a small perturbation of $(\covec_dec, \vec_dec)$ results in a small perturbation of the $\A$-coordinates.

We now show that $\psi$ is surjective by constructing a doubly-decorated convex projective structure from a given assignment $\alpha \in \A$ of positive real numbers to the ideal triangles and ideal edges of $S$. Fix an ideal triangle in $\widetilde{\Delta}$ with cyclically ordered vertices $(V_0, V_1, V_2)$ whose orientation agrees with that of $\widetilde{S}$. In this process we also lift $\alpha$ to an assignment of positive real numbers to the ideal triangles and oriented ideal edges of $\widetilde{\Delta}$.

Lemma~\ref{lem:dec_iso} shows that, if $\psi$ is to be surjective then the decorated ideal triangle $(R, C)_{012}$ assigned to $(V_0, V_1, V_2)$ by $\psi^{-1}(\alpha)$ is uniquely determined up to decorated isomorphism by the positive real numbers assigned to $(V_0, V_1, V_2)$ by $\alpha$.

We define a PL homeomorphism $\widetilde{\dev}: \widetilde{S} \rightarrow \R^3 \setminus \{0\}$ in such a manner that $\pi \circ \dev$ will define a PL developing map for $\psi^{-1}(\alpha)$. We will do so by determining an extension of $\widetilde{\dev}$ to the ideal vertices of $\widetilde{\Delta}$. Once we have done so we may extend linearly to the interior to ensure that $\pi \circ \widetilde{\dev}$ is $\hol$-equivariant. Therefore we need only determine $\widetilde{\dev}$ on the ideal vertices of $\widetilde{\Delta}$. In order to ensure we obtain a well-defined doubly-decorated structure, the extension of $\widetilde{\dev}$ to the ideal vertices of $\widetilde{\Delta}$ will identify each ideal vertex $V_i$ of $\widetilde{\Delta}$ with a concrete flag $(R_i, C_i)$ in $\R^3 \setminus \{0\}$.

Lemma~\ref{lem:next_vertex} shows that, if it is possible to construct $\widetilde{\dev}$ then the concrete decorated triangle whose concrete flags are $\widetilde{\dev}(V_i)$ for $i=0,1,2$, is uniquely determined, modulo the action of $\SL_3(\R)$, by the numbers assigned to $\widetilde{\Delta}$ by $\alpha$. Moreover there exists such an equivalence class of concrete decorated triangles for any such assignment of positive real numbers. Therefore fix an arbitrary representative $\widetilde{\dev}(V_i) = (R_i, C_i)$ for $i=0,1,2$ such that the triangle parameters and oriented edge parameters assigned to this concrete decorated triangle are those assigned to the ideal triangle $(V_0, V_1, V_2)$ by $ \alpha$.

Let $V_3$ be the vertex which is not $V_0$, $V_1$ or $V_2$ but which shares an ideal triangle in $\widetilde{\Delta}$ with both $V_0$ and $V_2$. Lemma~\ref{lem:next_vertex} shows that, having fixed $\widetilde{\dev}(V_i)$ for $i=0,1,2$, the concrete flag $\widetilde{\dev}(V_3)$ is uniquely determined by the positive real numbers assigned by $\alpha$ to the ideal triangle $(V_0, V_2, V_3)$ and the oriented edges $(V_0, V_3)$, $(V_3, V_0)$, $(V_2, V_3)$ and $(V_3, V_2)$. Continuing in this manner we determine a concrete flag $\widetilde{\dev}(V) =(r, C) $ on each of the ideal vertices $V$ of $\widetilde{\Delta}$. The set of all concrete flags constitutes a $\hol$-equivariant vector decoration $\vec_dec$ and covector decoration $\covec_dec$ for the set of ideal vertices of $\widetilde{\Delta}$.

In order to show that these satisfy the definition of a double decoration given in  \S\ref{subsec:convex_projective_structures} we need to show that for each covector $r \in \covec_dec$ there is a unique $V \in \vec_dec$ such that $r \cdot V = 0$ and that $r \cdot V' > 0$ for all $V' \in \vec_dec \setminus \{ V \} $. Fix $r \in \covec_dec$. By construction there is at least one $V \in \vec_dec$ such that $r \cdot V = 0$. We denote this vector by $V_r$. The fact that $r \cdot V' > 0$ for all $V' \in \vec_dec \setminus \{ V_r \}$ is exactly the content of Lemma~\ref{lem:one_side_of_kernel}.

We now linearly extend $\widetilde{\dev}$ to the interior of each ideal triangle of $\widetilde{\Delta}$ to obtain a PL developing map 
\[
\dev = \pi \circ \widetilde{\dev} : \widetilde{S} \rightarrow \RP^2 \enspace .
\]
We denote the image of $\dev$ by $\Omega$ and note that $\Omega$ is properly convex by Lemma~\ref{lem:positive_parameters}.

The holonomy representation $\hol : \pi_1(S) \rightarrow \PGL_3(\R)$ is uniquely determined by $\dev$ and our choice of universal cover because it is the push-forward of the action of the deck transformations on $\widetilde{S}$. 
We have a holonomy group $\Gamma = \hol(\pi_1(S) )$ which is discrete and torsion-free because $\hol$ is an isomorphism onto its image by construction.
Indeed if $\hol$ were either not discrete or not faithful then it could not be the case that $\dev$ was a homeomorphism onto its image and $(\covec_dec, \vec_dec)$ could not have arisen from a choice of positive triangle and edge invariants.

The marking $\phi : \sfrac{ \widetilde{S} }{ \pi_1(S) } \rightarrow \sfrac{ \Omega }{ \Gamma }$ is constructed by pushing down the developing map. Therefore we have uniquely determined a doubly-decorated convex projective structure $(\Omega, \Gamma, \phi, (\covec_dec, \vec_dec) )$ on $S$. It is clear by construction that $\psi ( (\Omega, \Gamma, \phi, ( \covec_dec, \vec_dec) ) ) = \alpha.$ So to complete the proof that $\psi$ is surjective, we need to show that the structure is of finite volume.

Denote the developing map and holonomy representation of $(\Omega, \Gamma, \phi, ( \covec_dec, \vec_dec)$ by $\dev$ and $\hol$ respectively.
Recall from  \S\ref{subsec:convex_projective_structures} that $S$ has finite area if and only if for all non-trivial peripheral elements $\gamma \in \pi_1(S)$ we have
\begin{equation}\label{eq:standard_parabolic}
\hol(\gamma) \sim \begin{pmatrix} 1 & 1 & 0 \\ 0 & 1 & 1 \\ 0 & 0 & 1 \end{pmatrix} \enspace .
\end{equation}

Fix a cusp $K$ of $S$ and let $\gamma \in \pi_1(S)$ be a peripheral curve such that $\pi_1 ( K ) = \langle \gamma \rangle$. 
The action of $\gamma$ on $\widetilde{S}$ via deck transformations fixes a unique point $\widetilde{W}$ in $\partial_{\infty} \widetilde{S}$. 
The vector decoration $\vec_dec$ determines a light-cone vector $W \in \vec_dec$ lifting $\dev( \widetilde{W} )$. 
Included in the definition of $\decTtf(S)$ is the condition that $W$ is the unique light cone representative of $\pi(W)$ in $( \covec_dec, \vec_dec ) $ and $( \covec_dec,  \vec_dec)$ is $\Gamma$-invariant. 
In particular $\hol(\gamma) \cdot W = W$, so one of the eigenvalues of $\hol(\gamma)$ is $1$.

On the other hand, since 
$\hol(\gamma)$ is nontrivial and preserves a properly convex domain 
there are only three options for the Jordan normal form of $\hol(\gamma)$, namely
\[
\begin{pmatrix} 1 & 1 & 0 \\ 0 & 1 & 1 \\ 0 & 0 & 1 \end{pmatrix},  \hspace{2mm} \begin{pmatrix} \xi^2 & 0 & 0 \\ 0 & \xi^{-1} & 1 \\ 0 & 0 & \xi^{-1} \end{pmatrix} \hspace{2mm}  \text{or } \begin{pmatrix} \eta & 0 & 0 \\ 0 & \mu & 0 \\ 0 & 0 & (\eta \mu )^{-1}\end{pmatrix} \enspace , \text{ where } \xi, \mu, \eta > 0 \enspace ,
\]
and where  $\xi, \eta \neq 1$.
We may assume that $\hol(\gamma)$ is equal to its Jordan normal form.
In the middle case there are no eigenvectors having eigenvalue $1$ so we may eliminate this possibility. If the peripheral element is diagonalisable, then we must have $\mu = 1$ and we may assume without loss of generality that $\eta > 1$ Suppose this was the case. The $1$-eigenspace of $\hol(\gamma)$ is the set of scalar multiples of the second basis vector $e_2$. Therefore we must have $W = t e_2$ for some $t \in \R^{\times}$ and $\hol(\gamma)$ preserves a covector whose kernel contains $e_2$. Denote that covector by $w = (A, 0, B)$ for some $A, B \in \R$. The fact that $\hol(\gamma)$ preserves $w$ means exactly that
\begin{align*}
w = w \cdot \left( \hol(\gamma) \right)^{-1} = (\eta^{-1} A, 0, \eta B ) \enspace .
\end{align*}
In particular $\eta = 1$ so $\hol(\gamma)$ is trivial. This contradiction shows that the only viable option for the Jordan normal form of $\hol(\gamma)$ is the standard parabolic matrix.

This concludes the proof that $(\Omega, \Gamma, \phi, (\covec_dec, \vec_dec) ) \in \decTtf(S).$

A small perturbation of $\alpha$ will result in a corresponding small pertburation of the double-decoration $( \covec_dec, \vec_dec )$ and a small, or zero, perturbation of the resulting developing map. Therefore it is clear that $\psi^{-1}$ is continuous. Therefore $\psi$ is a homeomorphism onto its image.\end{proof}


\subsection{Change of coordinates}

The choice of an underlying triangulation 
used to construct the $\A$-coordinates may be viewed as  choosing a global coordinate chart for $\decTtf(S)$.
We now investigate the associated transition maps that arise by changing the triangulations.

As the first case, assume that two triangulations $\Delta, \Delta' \subset S$ differ only in an edge flip.
That is, there are $a,b,c,d,e \in \Delta$ such that $(a,b,e)$ and $(c,d,e)$ are two distinct triangles and the triangulation $\Delta'$ is obtained by performing an edge flip along edge $e \in \Delta$.
This flip introduces the edge $f \in \Delta'$ and the two triangles $(b,c,f)$ and $(a,f,d)$.

Fix  $\A$-coordinates $\alpha \in \A_\Delta$ describing the doubly-decorated structure $x=\psi_\Delta^{-1}(\alpha) \in \decTtf(S)$. 
We denote the triangle and edge parameters around edge $e$ that are prescribed by $\alpha$ as depicted in Figure \ref{fig:flip_Acoords}.
For instance, $a^\pm$ are the two edge parameters assigned to the edge $a \in \Delta$ and the triangles $(a,b,e)$ and $(c,d,e)$ have triangle parameters $A$ and $B$, respectively.
Let $\alpha' \in \A_{\Delta'}$ describe the same doubly-decorated structure, i.e. $\alpha'=\psi_{\Delta'}(x)$.
Since $\Delta'$ arises from $\Delta$ by flipping edge $e$, all triangle and edge parameters in $\alpha$ apart from $A,B$ and $e^\pm$ carry over to $\alpha'$.
Adopting the notation from Figure \ref{fig:flip_Acoords} and using the calculations from the proof of Lemma \ref{lem:edge_flip} the remaining triangle parameters of $\alpha'$ are given by
\begin{align}\label{eq:flip_triang}
C = \frac{Ac^+ + Bb^+}{e^-} \qquad \text{and} \qquad D = \frac{Ad^- + Ba^-}{e^+} \enspace .
\end{align}
In turn, these give the two missing edge parameters
\begin{align}\label{eq:flip_edge}
f^+ = \frac{Ca^+ + Db^-}{A} \qquad \text{and} \qquad f^- = \frac{Cd^+ + Dc^-}{B} \enspace .
\end{align}

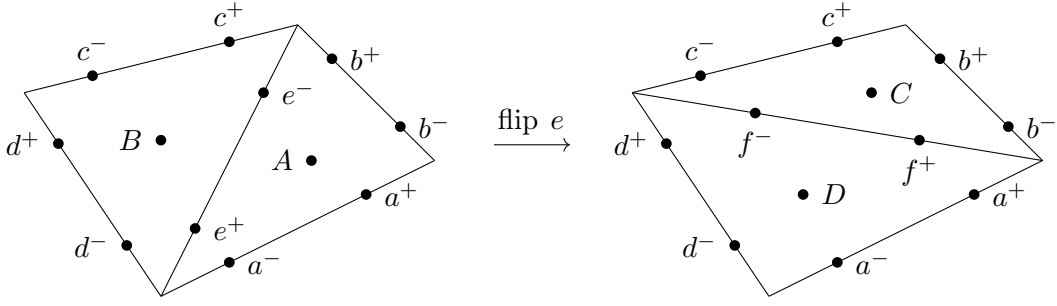
\begin{figure}
	\center
	\begin{tikzpicture}
	\node at (-4,0) {\begin{tikzpicture}[scale=0.9]
		\draw (4, 4) edge  (6, 2); 
		\draw (6, 2) edge (2,0); 
		\draw (2, 0) edge (4, 4); 
		\draw (4, 4) edge (0, 3); 
		\draw (0, 3) edge (2, 0); 
		\node [circle, fill = black, inner sep = 0pt, minimum size = 4pt, label = left : {$A$} ] (A) at (4.2, 2) {};	
		\node [circle, fill = black, inner sep = 0pt, minimum size = 4pt, label = left : {$B$} ] (B) at (2, 2.3) {};
		\node [circle, fill = black, inner sep = 0pt, minimum size = 4pt, label = right : {$e^+$} ] (ep) at (2.5, 1) {};
		\node [circle, fill = black, inner sep = 0pt, minimum size = 4pt, label = right : {$e^-$} ] (em) at (3.5, 3) {};
		\node [circle, fill = black, inner sep = 0pt, minimum size = 4pt, label = right : {$a^-$} ] (am) at (3, 0.5) {};
		\node [circle, fill = black, inner sep = 0pt, minimum size = 4pt, label = right : {$a^+$} ] (ap) at (5, 1.5) {};
		\node [circle, fill = black, inner sep = 0pt, minimum size = 4pt, label = right : {$b^+$} ] (bp) at (4.5, 3.5) {};
		\node [circle, fill = black, inner sep = 0pt, minimum size = 4pt, label = right : {$b^-$} ] (bm) at (5.5, 2.5) {};
		\node [circle, fill = black, inner sep = 0pt, minimum size = 4pt, label = above : {$c^+$} ] (cp) at (3, 3.75) {};
		\node [circle, fill = black, inner sep = 0pt, minimum size = 4pt, label = above : {$c^-$} ] (cm) at (1, 3.25) {};
		\node [circle, fill = black, inner sep = 0pt, minimum size = 4pt, label = left : {$d^+$} ] (dp) at (0.5, 2.25) {};
		\node [circle, fill = black, inner sep = 0pt, minimum size = 4pt, label = left : {$d^-$} ] (dm) at (1.5, 0.75) {};
		\end{tikzpicture}};
	\draw[->] (-0.5,0) -- (0.5,0);
	\node at (0,0.3) {flip $e$};
	\node at (4,0) {\begin{tikzpicture}[scale=0.9]
		\draw (4, 4) edge  (6, 2); 
		\draw (6, 2) edge (2,0); 
		\draw (6, 2) edge (0, 3); 
		\draw (4, 4) edge (0, 3); 
		\draw (0, 3) edge (2, 0); 
		\node [circle, fill = black, inner sep = 0pt, minimum size = 4pt, label = south : {$f^-$} ] at (1.8, 2.7) {};	
		\node [circle, fill = black, inner sep = 0pt, minimum size = 4pt, label = south : {$f^+$} ] at (4.2, 2.3) {};
		\node [circle, fill = black, inner sep = 0pt, minimum size = 4pt, label = right : {$D$} ] at (2.5, 1.5) {};
		\node [circle, fill = black, inner sep = 0pt, minimum size = 4pt, label = right : {$C$} ] at (3.5, 3) {};
		\node [circle, fill = black, inner sep = 0pt, minimum size = 4pt, label = right : {$a^-$} ] (am) at (3, 0.5) {};
		\node [circle, fill = black, inner sep = 0pt, minimum size = 4pt, label = right : {$a^+$} ] (ap) at (5, 1.5) {};
		\node [circle, fill = black, inner sep = 0pt, minimum size = 4pt, label = right : {$b^+$} ] (bp) at (4.5, 3.5) {};
		\node [circle, fill = black, inner sep = 0pt, minimum size = 4pt, label = right : {$b^-$} ] (bm) at (5.5, 2.5) {};
		\node [circle, fill = black, inner sep = 0pt, minimum size = 4pt, label = above : {$c^+$} ] (cp) at (3, 3.75) {};
		\node [circle, fill = black, inner sep = 0pt, minimum size = 4pt, label = above : {$c^-$} ] (cm) at (1, 3.25) {};
		\node [circle, fill = black, inner sep = 0pt, minimum size = 4pt, label = left : {$d^+$} ] (dp) at (0.5, 2.25) {};
		\node [circle, fill = black, inner sep = 0pt, minimum size = 4pt, label = left : {$d^-$} ] (dm) at (1.5, 0.75) {};
		\end{tikzpicture}};
	
	\end{tikzpicture}
	\caption{Flipping an edge $e$ in a triangulation introduces two new edge parameters $f^\pm$ and two new triangle parameters $C,D$ satisfying \eqref{eq:flip_edge} and \eqref{eq:flip_triang}.} \label{fig:flip_Acoords}
\end{figure}

The latter relations naturally define a real analytic isomorphism $R_{\Delta,\Delta'}: \A_\Delta \rightarrow \A_{\Delta'} $ with inverse $R_{\Delta',\Delta}$, yielding the commutative diagram

\begin{equation}
\begin{tikzpicture}
\matrix (m) [matrix of math nodes, row sep=2.3em, column sep=2.3em, nodes={anchor=center}]{
	& \decTtf(S) & \\
	\A_\Delta & & \A_{\Delta' \enspace .}\\
};
\path[-stealth] (m-1-2) edge node[above left]{\scriptsize $\psi_{\Delta}$}(m-2-1);
\path[-stealth] (m-1-2) edge node[above right]{\scriptsize $\psi_{\Delta'}$}(m-2-3);
\path[-stealth] (m-2-1) edge node[above ]{\scriptsize $R_{\Delta,\Delta'}$}(m-2-3);
\end{tikzpicture}\label{eq:cd_chart_trafo}
\end{equation}

Thus, the map $R_{\Delta,\Delta'}$ is the transition map we were aiming for.

Now consider two arbitrary triangulations $\Delta, \Delta' \subset S$. 
Choose a finite sequence of triangulations $(\Delta_0, \ldots , \Delta_m)$ such that $\Delta_0=\Delta$, $\Delta_m=\Delta'$ and two consecutive triangulations $\Delta_i,\Delta_{i+1}$ differ by an edge flip, $i = 0,\ldots,m-1$ and come with the transition map $R_{\Delta_i,\Delta_{i+1}}$.
We define the transition map $R_{\Delta,\Delta'}:\A_\Delta \rightarrow \A_{\Delta'}$ to be the concatenation 
\begin{equation}\label{eq:def_chart_transition}
R_{\Delta,\Delta'} := R_{\Delta_{m-1},\Delta_{m}} \circ \ldots \circ R_{\Delta_1,\Delta_2} \circ R_{\Delta_0,\Delta_{1}} \enspace .
\end{equation}
It follows that $R_{\Delta,\Delta'}$ is a real analytic isomorphism. Clearly, $R_{\Delta,\Delta'}$ admits the same commutative diagram \eqref{eq:cd_chart_trafo}. 

Recall from  \S\ref{subsec:Decorated and doubly decorated} that there is a natural action of $\MCG(S)$ on $\decTtf ( S ) $ by the `change of marking' map. If $\mu \in \MCG(S)$ and $\Delta$ is an ideal cell decomposition of $S$ then we denote by $\mu \circ \Delta$ the ideal cell decomposition whose edges are the image under $\mu$ of the edges of $\Delta$. We use this map and the fact that $R_{\Delta, \Delta'}$ is real analytic to obtain the following corollary.

\begin{cor}\label{cor:analytic structure}
Fix an ideal triangulation $\Delta$ of $S$ and give $\decTtf(S)$ the real analytic structure pulled back from from $\A_{\Delta}.$ Then this real analytic structure is independent of the choice of $\Delta,$ and the action of $\MCG(S)$ on $\decTtf(S)$ is analytic.
\end{cor}
\begin{proof}
The natural analytic structure on $\A_{\Delta}$ is pulled back to $\decTtf(S)$ via the homeomorphism $\Psi_{\Delta}$. The fact that $R_{\Delta, \Delta'}$ is a homogeneous rational homeomorphism ensures that the analytic structure thus obtained is independent of the choice of $\Delta$.
Let $\mu \in \MCG(S)$. We now show that $\mu^* : \decTtf(S) \rightarrow \decTtf(S)$ is real analytic with respect to the structure obtained above. Let $\Delta' := \mu \circ \Delta$. We have
\begin{align*}
R_{ \Delta , \Delta' } \circ \Psi_{\Delta} ( \Omega , \Gamma , \phi , ( \covec_dec, \vec_dec ) ) &= \Psi_{\Delta'} ( \Omega , \Gamma , \mu \circ \phi , ( \covec_dec, \vec_dec ) ) \\
&= \Psi_{\Delta'} \circ \mu^*  ( \Omega , \Gamma , \phi , ( \covec_dec, \vec_dec ) ) \\
\Rightarrow \Psi_{\Delta'}^{-1} \circ R_{\Delta, \Delta'} \circ \Psi_{\Delta}( \Omega, \Gamma, \phi, (\covec_dec, \vec_dec) ) &= \mu^* ( \Omega, \Gamma, \phi, (\covec_dec, \vec_dec) ) \enspace .
\end{align*}
The maps $\Psi_{\Delta}$ and $\Psi_{\Delta'}$ are bianalytic by construction and the map $R_{\Delta, \Delta'}$ is analytic. Therefore $\Psi_{\Delta'}^{-1} \circ R_{\Delta, \Delta'} \circ \Psi_{\Delta} = \mu^*$ is a real bianalytic map with respect to the analytic structure on $\decTtf(S)$ pulled back from $\A_{\Delta}$.
\end{proof}


\subsection{Matrix representations}

Fock and Goncharov \cite{FoGo-moduli-2007} parametrize the moduli space of (undecorated) convex projective structures on $S$ using the so called $\X$-coordinates. 
We will denote the moduli space of $\X$-coordinates by $\X_{\Delta}$.
See also \cite{CTT-moduli-2018} for a comprehensive construction of $\X$-coordinates.
In a manner analogous 
to the construction of $\A$-coordinates, the key idea is to assign positive real numbers to the triangles and oriented edges of $\Delta$, called \emph{triple ratios} and \emph{quadruple ratios}, respectively.

Let $\alpha \in \A_\Delta$. 
Then $\alpha$ determines a convex projective structure whose $\X$-coordinates are given as follows.
Consider triangle and edge parameters around an edge $e \in \Delta$ as in Figure \ref{fig:A_to_X_coords}~(left). 
Then the triple ratios $t_A$ and $t_B$ assigned to the two triangles adjacent to $e$ and the two quadruple ratios $q_{e^+}$ and $q_{e^-}$ along $e$ are given by

\begin{equation}
\label{eq:AtoXcoords}
t_{A} = \frac{a^{-}b^{-}e^{-}}{a^{+}b^{+}e^{+}} \enspace , \qquad t_{B} = \frac{c{^+}d^{+}e^{+}}{c^{-}d^{-}e^{-}}  \enspace ,
\qquad q_{e^+} = \frac{Ad^{-}}{Ba^{-}} \enspace , \qquad q_{e^-} = \frac{Bb^{+}} {Ac^{+}} \enspace ,	
\end{equation}
where the triple and quadruple ratios are as in Figure \ref{fig:A_to_X_coords} (right).

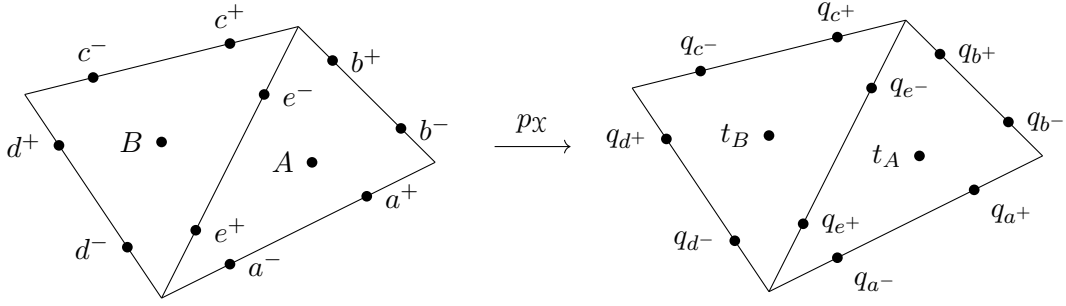
\begin{figure}
	\center
	\begin{tikzpicture}
	\node at (-4,0) {\begin{tikzpicture}[scale=0.9]
		\draw (4, 4) edge  (6, 2); 
		\draw (6, 2) edge (2,0); 
		\draw (2, 0) edge (4, 4); 
		\draw (4, 4) edge (0, 3); 
		\draw (0, 3) edge (2, 0); 
		\node [circle, fill = black, inner sep = 0pt, minimum size = 4pt, label = left : {$A$} ] (A) at (4.2, 2) {};	
		\node [circle, fill = black, inner sep = 0pt, minimum size = 4pt, label = left : {$B$} ] (B) at (2, 2.3) {};
		\node [circle, fill = black, inner sep = 0pt, minimum size = 4pt, label = right : {$e^+$} ] (ep) at (2.5, 1) {};
		\node [circle, fill = black, inner sep = 0pt, minimum size = 4pt, label = right : {$e^-$} ] (em) at (3.5, 3) {};
		\node [circle, fill = black, inner sep = 0pt, minimum size = 4pt, label = right : {$a^-$} ] (am) at (3, 0.5) {};
		\node [circle, fill = black, inner sep = 0pt, minimum size = 4pt, label = right : {$a^+$} ] (ap) at (5, 1.5) {};
		\node [circle, fill = black, inner sep = 0pt, minimum size = 4pt, label = right : {$b^+$} ] (bp) at (4.5, 3.5) {};
		\node [circle, fill = black, inner sep = 0pt, minimum size = 4pt, label = right : {$b^-$} ] (bm) at (5.5, 2.5) {};
		\node [circle, fill = black, inner sep = 0pt, minimum size = 4pt, label = above : {$c^+$} ] (cp) at (3, 3.75) {};
		\node [circle, fill = black, inner sep = 0pt, minimum size = 4pt, label = above : {$c^-$} ] (cm) at (1, 3.25) {};
		\node [circle, fill = black, inner sep = 0pt, minimum size = 4pt, label = left : {$d^+$} ] (dp) at (0.5, 2.25) {};
		\node [circle, fill = black, inner sep = 0pt, minimum size = 4pt, label = left : {$d^-$} ] (dm) at (1.5, 0.75) {};
		\end{tikzpicture}};
		\draw[->] (-0.5,0) -- (0.5,0);
		\node at (0,0.3) {$p_{\mathcal{X}}$};
	\node at (4,0) {\begin{tikzpicture}[scale=0.9]
	\draw (4, 4) edge  (6, 2); 
	\draw (6, 2) edge (2,0); 
	\draw (2, 0) edge (4, 4); 
	\draw (4, 4) edge (0, 3); 
	\draw (0, 3) edge (2, 0); 
	\node [circle, fill = black, inner sep = 0pt, minimum size = 4pt, label = left : {$t_A$} ] (A) at (4.2, 2) {};	
	\node [circle, fill = black, inner sep = 0pt, minimum size = 4pt, label = left : {$t_B$} ] (B) at (2, 2.3) {};
	\node [circle, fill = black, inner sep = 0pt, minimum size = 4pt, label = right : {$q_{e^+}$} ] (ep) at (2.5, 1) {};
	\node [circle, fill = black, inner sep = 0pt, minimum size = 4pt, label = right : {$q_{e^-}$} ] (em) at (3.5, 3) {};
	\node [circle, fill = black, inner sep = 0pt, minimum size = 4pt, label = below right : {$q_{a^-}$} ] (am) at (3, 0.5) {};
	\node [circle, fill = black, inner sep = 0pt, minimum size = 4pt, label = below right : {$q_{a^+}$} ] (ap) at (5, 1.5) {};
	\node [circle, fill = black, inner sep = 0pt, minimum size = 4pt, label = right : {$q_{b^+}$} ] (bp) at (4.5, 3.5) {};
	\node [circle, fill = black, inner sep = 0pt, minimum size = 4pt, label = right : {$q_{b^-}$} ] (bm) at (5.5, 2.5) {};
	\node [circle, fill = black, inner sep = 0pt, minimum size = 4pt, label = above : {$q_{c^+}$} ] (cp) at (3, 3.75) {};
	\node [circle, fill = black, inner sep = 0pt, minimum size = 4pt, label = above : {$q_{c^-}$} ] (cm) at (1, 3.25) {};
	\node [circle, fill = black, inner sep = 0pt, minimum size = 4pt, label = left : {$q_{d^+}$} ] (dp) at (0.5, 2.25) {};
	\node [circle, fill = black, inner sep = 0pt, minimum size = 4pt, label = left : {$q_{d^-}$} ] (dm) at (1.5, 0.75) {};		
		\end{tikzpicture}};
	
\end{tikzpicture}
\caption{Projecting $\A$-coordinates (left) to $\X$-coordinates (right) by forgetting the vector and covector decorations.} 
\label{fig:A_to_X_coords}
\end{figure}

The map $p_\X: \A \rightarrow \mathcal{X}$ defined by the equations \eqref{eq:AtoXcoords} is the projection that simply forgets about the decorations of vectors and covectors.
Indeed, it is not hard to verify that each of the terms in \eqref{eq:AtoXcoords} is invariant under redecorating both vectors and covectors.
Points in the image of $p_\X$ have fibers homeomorphic to $\R_{>0}^{2n}$, the space of admissible vector and covector decorations.

In \cite{CTT-moduli-2018} the authors describe how one can obtain a representation $\rho_x: \pi_1(S) \rightarrow \Gamma < \PGL_3(\Omega)$ associated to the convex projective structure $x\in \X_\Delta$ determined by an assignment of $\X$-coordinates. 
Here, we won't be concerned with the full description of this process, but rather give the fundamental building blocks to obtain an explicit description of $\rho_x$ in terms of the coordinates of $x$. 
Consider the triangulation $\widehat{\Delta}$ that is obtained from $\Delta$ by joining the vertices of each triangle to its barycenter.
The \emph{monodromy graph} $\mathcal{G}$ is defined as the dual spine of $\widehat{\Delta}$.
By construction $\mathcal{G}$ has three nodes for each triangle in $\Delta$ and the edges of $\mathcal{G}$ come in two types, dualizing either an edge of $\Delta$ or an edge of $\widehat{\Delta}\setminus\Delta$.
Consequently, each edge of the monodromy graph is associated to either a triangle or and edge of $\Delta$.
The developing map $\dev$ determined by $x$ 
lifts both $\Delta$ and $\mathcal{G}$ to the convex domain $\Omega$, and we denote these lifts by $\widetilde{\Delta}$ and $\widetilde{\mathcal{G}}$, repectively.

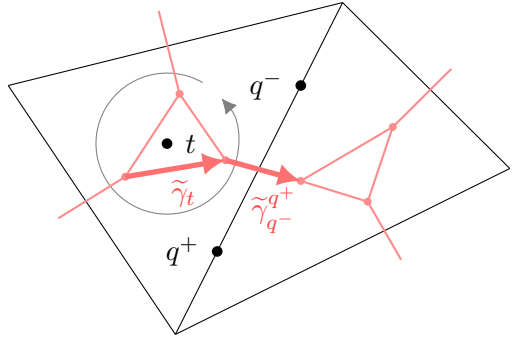
\begin{figure}
	\centering
	\begin{tikzpicture}[scale=1.1]
	\draw (4, 4) edge  (6, 2); 
	\draw (6, 2) edge (2,0); 
	\draw (2, 0) edge (4, 4); 
	\draw (4, 4) edge (0, 3); 
	\draw (0, 3) edge (2, 0); 
	\node [circle, fill = black, inner sep = 0pt, minimum size = 4pt, label = right : {$t$} ] (B) at (1.9, 2.3) {};
	\node [circle, fill = black, inner sep = 0pt, minimum size = 4pt, label = left : {$q^+$} ] (ep) at (2.5, 1) {};
	\node [circle, fill = black, inner sep = 0pt, minimum size = 4pt, label = left : {$q^-$} ] (em) at (3.5, 3) {};
	\draw [-{Latex[length=2mm, width=2mm]}, gray] ([shift = (60:0.85cm)]1.9,2.3) arc (60:400:0.85cm);
	\node [circle, fill = red1!60, inner sep = 0pt, minimum size = 3pt] at (2.05,2.9) {};
	\node [circle, fill = red1!60, inner sep = 0pt, minimum size = 3pt] at (1.4,1.9) {};
	\node [circle, fill = red1!60, inner sep = 0pt, minimum size = 3pt] at (2.6,2.1) {};
	\node [circle, fill = red1!60, inner sep = 0pt, minimum size = 3pt] at (3.5,1.85) {};
	\node [circle, fill = red1!60, inner sep = 0pt, minimum size = 3pt] at (4.3,1.6) {};
	\node [circle, fill = red1!60, inner sep = 0pt, minimum size = 3pt] at (4.6,2.5) {};
	
	\draw[color=red1!60, line width=0.8pt] (2.05,2.9) -- (1.4,1.9) -- (2.6,2.1) -- cycle;
	\draw[color=red1!60, line width=0.8pt] (2.05,2.9) -- (1.8,3.9);
	\draw[color=red1!60, line width=0.8pt] (1.4,1.9)  -- (0.6,1.4);
	\draw[color=red1!60, line width=0.8pt] (2.6,2.1)  -- (3.5,1.85);
	\draw[color=red1!60, line width=0.8pt] (3.5,1.85) -- (4.3,1.6) -- (4.6,2.5) -- cycle;
	\draw[color=red1!60, line width=0.8pt] (4.3,1.6) -- (4.7,0.9);
	\draw[color=red1!60, line width=0.8pt] (4.6,2.5) -- (5.3,3.2);
	\draw[color=red1!75, line width=2pt, -{Latex[length=3mm, width=3mm]}]  (1.4,1.9) -- (2.6,2.1);
	\draw[color=red1!75, line width=2pt, -{Latex[length=3mm, width=3mm]}]  (2.6,2.1)  -- (3.5,1.85);
	\node at (2.1,1.7) {\textcolor{red1}{$\widetilde{\gamma}_t$}};
	\node at (3.18,1.5) {\textcolor{red1}{$\widetilde{\gamma}^{q^+}_{q^-}$}};
	\end{tikzpicture}
	\caption{Local sketch of the developed monodromy graph $\widetilde{\mathcal{G}}$ (red).}
	\label{fig:monodromy}
\end{figure}
 
Let $\widetilde{\gamma}_t \in \underbar{E}(\widetilde{\mathcal{G}})$ be associated to a triangle of $\widetilde{\Delta}$ whose triple ratio is $t\in \R_{>0}$ and such that the orientation of $\widetilde{\gamma}_t$ coincides with the orientation of the triangle. We associate to $\widetilde{\gamma}_t$ the matrix
\begin{equation}
T(\widetilde{\gamma}_t)=T(t) = \frac{1}{\sqrt[3]{t}}\begin{pmatrix} 0 & 0 & 1 \\ 0 & -1 & -1 \\ t & t+1 & 1 \end{pmatrix} \enspace .
\end{equation}
The oppositely oriented edge $\widetilde{\gamma}_t^{-1}$ is associated to the inverse matrix $T(\widetilde{\gamma}_t^{-1})=T(t)^{-1}$.
Now let $\widetilde{\gamma}^{q^+}_{q^-} \in \underbar{E}(\widetilde{\mathcal{G}})$ be associated (dual) to an edge of $\Delta$ and oriented in such a way that the quadruple ratio $q^+$ (respectively $q^-$) appears on the right (left) of $\widetilde{\gamma}^{q^+}_{q^-}$. Then the matrix
\begin{equation}
E(\widetilde{\gamma}^{q^+}_{q^-})=E(q^+,q^-) = \sqrt[3]{\frac{q^+}{q^-}}\begin{pmatrix} 0 & 0 & q^- \\ 0 & -1 & 0 \\ \frac{1}{q^+} & 0 & 1 \end{pmatrix}
\end{equation}
is associated to $\widetilde{\gamma}^{q^+}_{q^-}$. 
 
Now let $\gamma \in \pi_1(S)$.
Furthermore, denote by $\underbar{E}(\widetilde{\mathcal{G}})$ the set of oriented edges of $\widetilde{\mathcal{G}}$.
Up to a homotopy equivalence we may think of $\gamma$ as being a composition of oriented edges in $\mathcal{G}$. 
Hence, we may consider a lift of $\gamma$ to a sequence of oriented edges $\widetilde{\gamma} = \widetilde{\gamma}_{t_1}^{\varepsilon_1}  \cdot \widetilde{\gamma}^{q_1^+}_{q_1^-} \cdots \widetilde{\gamma}_{t_m}^{\varepsilon_m}  \cdot \widetilde{\gamma}^{q_m^+}_{q_m^-} \subset \underbar{E}(\widetilde{\mathcal{G}})$ that alternates between edges associated to triangles and edges of $\Delta$. 
Here the value of $\varepsilon_i \in \{\pm 1\}$ depends on the relative orientation of the oriented edge $\widetilde{\gamma}_{t_i}$ to the triangle $t_i$ as explained above.
The key observation (\cite[Theorem 5.2.4]{CTT-moduli-2018}) is that the conjugacy class of the monodromy matrix $\rho_x(\gamma)$ is given by
\begin{equation}\label{eq:monodromy}
\left[ \rho_x(\gamma) \right] = \left[ T(t_1)^{\varepsilon_1} \cdot E(q_1^+,q_1^-)~ \cdots~ T(t_m)^{\varepsilon_m} \cdot E(q_m^+,q_m^-) \right] \enspace .
\end{equation}

In order to obtain the monodromy matrices in terms of a given set of $\A$-coordinates, one first uses the projection $p_\X$ defined by \eqref{eq:AtoXcoords} to extract the underlying $\X$-coordinates and then applies the above techniques to calculate the monodromy matrices as in \eqref{eq:monodromy}. 

Note that by definition $\A$-coordinates only define convex projective structures that have finite area. 
This is not true for $\X$-coordinates in general.
The conditions on $\X$-coordinates to define a finite area structure are as follows.
Let $i \in \left[ n \right]$ be a puncture of $S$ and let $(q_{ik})_{k=1}^{m_i}$ denote the cyclically ordered quadruple ratios assigned to the edges oriented away from $i$.
Similarly, denote by $(q_{ki})_{k=1}^{m_i}$ the quadruple ratios  
assigned to edges which are oriented towards $i$, and let $t_k$ be the triple ratio of the triangle that is oriented with $q_{ij}$. Figure \ref{fig:around_puncture} depicts the case $m_i=5$.
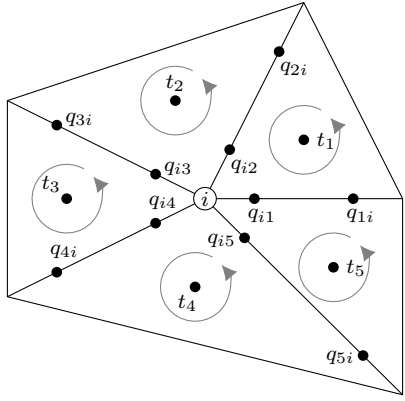
\begin{figure}
	\centering
	\begin{tikzpicture}[scale=1.3]
\draw (0,0) --  (2,0); 
\draw (0,0) --  (1,2);
\draw (0,0) --  (-2,1);  
\draw (0,0) --  (-2,-1); 
\draw (0,0) --  (2,-2); 
\draw (2,0) -- (1,2) -- (-2,1) -- (-2,-1) -- (2,-2) -- cycle;
\begin{scriptsize}
\node [circle, draw = black, fill = white, inner sep = 3pt, minimum size = 3pt, label = center: {$i$} ] at (0,0) {};
\node [circle, fill = black, inner sep = 0pt, minimum size = 4pt, label = { [shift ={(0.1,-0.5)}] $q_{i1}$} ] at (0.5,0) {};
\node [circle, fill = black, inner sep = 0pt, minimum size = 4pt, label = { [shift ={(0.1,-0.5)}] $q_{1i}$} ] at (1.5,0) {};
\node [circle, fill = black, inner sep = 0pt, minimum size = 4pt, label = { [shift ={(0.2,-0.5)}] $q_{i2}$} ] at (0.25, 0.5) {};
\node [circle, fill = black, inner sep = 0pt, minimum size = 4pt, label = { [shift ={(0.2,-0.5)}] $q_{2i}$} ] at (0.75, 1.5) {};
\node [circle, fill = black, inner sep = 0pt, minimum size = 4pt, label = { [shift ={(0.3,-0.2)}] $q_{i3}$} ] at (-0.5, 0.25) {};
\node [circle, fill = black, inner sep = 0pt, minimum size = 4pt, label = { [shift ={(0.3,-0.2)}] $q_{3i}$} ] at (-1.5, 0.75) {};
\node [circle, fill = black, inner sep = 0pt, minimum size = 4pt, label = { [shift ={(0.1,-0.0)}] $q_{i4}$} ] at (-0.5, -0.25) {};
\node [circle, fill = black, inner sep = 0pt, minimum size = 4pt, label = { [shift ={(0.1,-0.0)}] $q_{4i}$} ] at (-1.5, -0.75) {};
\node [circle, fill = black, inner sep = 0pt, minimum size = 4pt, label = { [shift ={(-0.3,-0.3)}] $q_{i5}$} ] at (0.4, -0.4) {};
\node [circle, fill = black, inner sep = 0pt, minimum size = 4pt, label = { [shift ={(-0.3,-0.3)}] $q_{5i}$} ] at (1.6, -1.6) {};
\node [circle, fill = black, inner sep = 0pt, minimum size = 4pt, label = { [shift ={(0.3,-0.3)}] $t_{1}$} ] at (1, 0.6) {};
\draw [-{Latex[length=2mm, width=2mm]}, gray] ([shift = (60:0.35cm)]1, 0.6) arc (60:400:0.35cm);
\node [circle, fill = black, inner sep = 0pt, minimum size = 4pt, label = { [shift ={(0,-0.05)}] $t_{2}$} ] at (-0.3, 1) {};
\draw [-{Latex[length=2mm, width=2mm]}, gray] ([shift = (60:0.35cm)]-0.3, 1) arc (60:400:0.35cm);
\node [circle, fill = black, inner sep = 0pt, minimum size = 4pt, label = { [shift ={(-0.2,-0.1)}] $t_{3}$} ] at (-1.4, 0) {};
\draw [-{Latex[length=2mm, width=2mm]}, gray] ([shift = (60:0.35cm)]-1.4, 0) arc (60:400:0.35cm);
\node [circle, fill = black, inner sep = 0pt, minimum size = 4pt, label = { [shift ={(-0.1,-0.5)}] $t_{4}$} ] at (-0.1, -.9) {};
\draw [-{Latex[length=2mm, width=2mm]}, gray] ([shift = (60:0.35cm)]-0.1,-0.9) arc (60:400:0.35cm);
\node [circle, fill = black, inner sep = 0pt, minimum size = 4pt, label = { [shift ={(0.3,-0.3)}] $t_{5}$} ] at (1.3, -0.7) {};
\draw [-{Latex[length=2mm, width=2mm]}, gray] ([shift = (60:0.35cm)]1.3,-0.7) arc (60:400:0.35cm);
\end{scriptsize}
	\end{tikzpicture}
	\caption{Local picture of $\X$-coordinates around a puncture $i$ on a triangulation with $m_i=5$ triangles adjacent to $i$.}
	\label{fig:around_puncture}
\end{figure}
With this notation at hand, the finite area $\X$-coordinates are exactly those that satisfy
\begin{equation} \label{eq:X_finite_area}
	 \prod \limits_{k=1}^{m_i} q_{ik} = 1 \qquad \text{and} \qquad  \prod \limits_{k=1}^{m_i} q_{ki} t_k = 1  \enspace ,
\end{equation}
for all punctures $i \in \left[n \right]$.
To see this one computes the monodromy matrices $\rho_x(\gamma)$ as in \eqref{eq:monodromy} for peripheral elements $\gamma \in \pi_1(S)$ and recalls that these have to be conjugate to the standard parabolic \eqref{eq:standard_parabolic} for finite area structures (see \cite{CTT-moduli-2018}).
Again, using the definition of $p_\X$ it is straightforward to check that \eqref{eq:X_finite_area} is automatically satisfied in the image of $p_\X$.

As a result, if we denote by $\X_\Delta^f$ the space of $\X$-coordinates on $\Delta$ that define a finite area structure, we see that $\A_\Delta$ is a trivial $(\R^{2n}_{>0})$-fiber bundle over $\X_\Delta^f$ with $p_\X$ as canonical projection. 


\subsection{The space of singly decorated structures}

We conclude this section by describing the effect of scaling vectors or covectors individually. We use the notation of the previous section and refer again to Figure~\ref{fig:around_puncture}. Scaling the vector at vertex $v_i$ by $\lambda>0$ multiplies the edge parameters $q_{ki}$ by $\lambda$ and the triangle parameters of all triangles incident with $v_i$ by $\lambda.$ Here, the triangle parameters are counted with multiplicity, so if $m$ corners of the triangle with parameter $t$ are at $v_i,$ then the parameter is multiplied by $\lambda^m.$ In particular, scaling \emph{all} vector decorations by $\lambda$, we may arrange that all triangle parameters sum to one.

Scaling the covector at $v_i$ by $\lambda$ scales the edge parameters $q_{ik}$ by $\lambda$ and does not affect the triangle parameters. In particular, one may use the covector scaling to arrange that all parameters near $v_1$ sum to one. The space of \emph{(independent) decorated marked structures} can therefore be identified with the subspace 
\[
\A_\Delta^\dag = \{ \alpha \in \A_\Delta \mid \sum_{j=1}^{|T|} t_j=1, \ \sum_{k=1}^{m_i} q_{ik} = 1 \text{ for all } i\in [n]\} 
\cong (\bigtimes_{i\in [n]} \Delta^{m_i -1} ) \times \Delta^{T-1},
\]
i.e.\thinspace it is a product of open simplices, one for each vertex and of dimension one less than its degree, times a simplex of dimension one less than the number of triangles.

\section{Outitude}
\label{sec:outitude}


We briefly recall Cooper and Long's \cite{CoLo-Epstein-Penner-2015} generalisation of Epstein and Penner's \cite{EpPe-Euclidean-1988} convex hull construction in \S\ref{subsec:convex_hull}. The convex hull construction can be achieved via an edge flipping algorithm, and we show in \S\ref{subsec:Outitude is intrinsic} that the decision of whether or not to flip an edge can be made based on an \emph{outitude} computed from the $\A$--coordinates. This is used in \S\ref{subsec:putative cells} to describe the putative cells in the decomposition of moduli space as semi-algebraic sets. The algorithm to compute the canonical cell decomposition of a doubly decorated projective surface is given in \S\ref{subsec:flip algo}. 


\subsection{Cooper and Long's convex hull construction} 
\label{subsec:convex_hull}

Let $(\Omega, \Gamma, \phi) \in \Ttf(S)$ be a properly convex projective structure of finite area, and let $\vec_dec$ be a vector decoration of this structure.
A key observation of Cooper and Long is that the vector decoration $\vec_dec$ is a discrete subset of $\mathcal{L}^+ \cup \{0\}$.
Let $\mathcal{C}$ be the convex hull of $\vec_dec$. 
Although $\mathcal{C}$ has infinitely many vertices, its facets are finite polygons.
The projection of the proper faces of $\mathcal{C}$ to $\Omega$ is a $\Gamma$-invariant cell decomposition of $\Omega$. 
This descends to a cell decomposition of $\Omega / \Gamma$ because of $\Gamma$-invariance. 
Via the marking $\phi: S \mapsto \sfrac{\Omega}{\Gamma}$ this induces an ideal cell decomposition of $S$. 

With this at hand, each doubly-decorated convex projective structure $x = ( \Omega , \Gamma , \phi, (\vec_dec , \covec_dec) ) \in \decTtf(S)$ induces an ideal cell decomposition of $S$ via the
convex hull construction. 
We denote this ideal cell decomposition by $\Delta(x)$ and refer to it as its \emph{canonical cell decomposition}.
Note that the convex hull construction does not depend on a covector decoration $\covec_dec$. 
For each $x \in \decTtf(S)$ this yields a trivial $(n+1)$--parameter family of doubly-decorated structures that induce the same ideal decomposition of $S$. Here, $n$ dimensions come from varying the orbits of the covectors independently, and one dimension comes from varying all vector decorations by the same scalar.

Analogous to Penner \cite{Penner-Decorated-1987}, define for any ideal cell decomposition $\Delta \subset S$ the sets
\begin{align}
\Cell( \Delta ) &= \{ x \in \decTtf(S) \mid  \Delta(x) = \Delta \} \enspace , \label{eq:def_open_cell}\\
\clCell( \Delta ) &= \{ x \in \decTtf(S) \mid  \Delta(x) \subseteq  \Delta \} \enspace .\label{eq:def_closed_cell}
\end{align} 

Our ultimate goal is to show that the decomposition of $\decTtf(S)$ into these sets, that is
\begin{equation} \label{eq:cell_dec}
	\{ \Cell(\Delta) ~|~ \Delta \text{ ideal cell decomposition of } S \} \enspace , 
\end{equation}
is a $\MCG(S)$-invariant cell decomposition of $\decTtf(S)$, and that $\Cell( \Delta )\neq \emptyset$ for each ideal cell decomposition $\Delta \subset S$.

Recall that the mapping class group $\MCG(S)$ acts on both the surface $S$ and on $\decTtf(S)$.
The convex hull construction is natural for these actions in the sense that for $\mu \in \MCG(S)$ and $x \in \decTtf(S)$ we have $\Delta( \mu \cdot x ) = \mu \cdot \Delta(x)$ by definition.
It follows that 
\[
\mu \cdot \Cell(\Delta) = \Cell(\mu \cdot \Delta)
\]
showing that the decomposition defined in \eqref{eq:cell_dec} is $\MCG(S)$-invariant, indeed.

What is left is to show that each $\Cell(\Delta)$ is actually a cell. This is what we will be concerned with in \S\ref{subsec:main_results}. The remainder of this section is concerned with developing the main tool, \emph{outitude}, for this task, and highlighting some of its properties.


\subsection{Outitude is intrinsic}
\label{subsec:Outitude is intrinsic}

Suppose $( R, C )$ and $( R', C' )$ are two concrete decorated triangles sharing a side $e$ in a concrete decorated triangulation. Denote the matrices $C, C', R$ and $R'$ as follows.
\[
\begin{array}{lccp{.5cm}lcc}
C &=& ( \vtx_0 \mid \vtx_1 \mid \vtx_2 )  & & R  & = & ( \covtx_0 \mid \mid \covtx_1 \mid \mid \covtx_2 ) \\
C' &=& ( \vtx_0 \mid \vtx_2 \mid \vtx_3 ) & & R' &=& ( \covtx_0 \mid \mid \covtx_2 \mid \mid \covtx_3 )
\end{array}
\]
Consider the (possibly degenerate) tetrahedron $T=\conv(\vtx_0,\vtx_1,\vtx_2,\vtx_3)$. 
The matrices $C, C', D = ( \vtx_1 \mid \vtx_2 \mid \vtx_3 )$ and $D' = ( C_0  \mid C_1 \mid C_3)$ represent the four facets of $T$ in the sense that their column vectors are the vertices of the corresponding facets. 
Let $e^+ := \covtx_0 \cdot \vtx_2$ and $e^- := \covtx_2 \cdot \vtx_0$ denote the edge parameters assigned to the two oriented edges on $e$. The value
\begin{equation} \label{eq:outitude_def1}
	\Out ( e ) := e^+ e^- \left( \det(D) + \det(D') - \det(C) - \det(C') \right)
\end{equation}
is called the \emph{outitude} along the edge $e$.

\begin{lem}
Let $(R, C)$ and $(R', C')$ be two concrete decorated triangles that share a side as above.
The tetrahedron $T$ they determine does not contain the origin.

\end{lem}
\begin{proof}
Consider the closed half space $H_0^+ = \{ X \in \R^3 \mid \covtx_0 \cdot X \geq 0 \}$. By definition of concrete decorated triangles we have
\[
\covtx_0 \cdot \vtx_1 > 0 \enspace , \quad \covtx_0 \cdot \vtx_2 > 0 \enspace , \quad \covtx_0 \cdot \vtx_3 > 0 \quad \text{and}  \quad \covtx_0 \cdot \vtx_0 = 0 \enspace .
\]
Thus the tetrahedron $T=\conv(\vtx_0, \vtx_1, \vtx_2, \vtx_3)$ is contained in $H_0^+$ and only its vertex $\vtx_0$ is contained in the boundary $H_0=\bd H_0^+$. Since $\vtx_0 \neq 0$ it follows that the origin is not contained in the tetrahedron.
\end{proof}

Let $K \subset \R^3$ be a closed convex subset not containing the origin. The \emph{bottom} of $p \in K$ is the point $\lambda \cdot p$ where
\[
\lambda := \min \{ \lambda \in \R \mid \lambda \cdot p \in K \} \enspace .
\]
We say that $p$ is \emph{outer} if it is its own bottom. The \emph{outer hull} of $K$ is its set of outer points.

\begin{lem}\label{lem:convex}
Let $C, C', D, D'$ and $T$ be as above and denote by $e'$ the side shared by $D$ and $D'$. Then exactly one of the following statements holds.
\begin{enumerate}[label=(\roman*)]
	\item The outer hull of $T$ is $C \cup_e C'$ and $\Out(e) > 0$.
	\item The outer hull of $T$ is $D \cup_e' D'$ and $\Out(e) < 0$. 
	\item The outer hull of $T$ is $C \cup_e C' = D \cup_{e'} D'$ and $\Out(e) = 0$. In this case $T$ is degenerate.
\end{enumerate} 
\end{lem}
\begin{proof}
Let us first prove the claim on the outer hull of $T$. If $T$ is degenerate then it is contained in a plane which does not pass through the origin and the outer hull of $T$ is clearly $C \cup_e C' = D \cup_{e'} D'$. 

Now assume that this is not the case. The facets of $T$ are those determined by $C, C', D$ and $D'$ so clearly the outer hull of $T$ is contained in the union of these sets.
We have $\det(C)>0$ and $\det(C')>0$ by definition. 
Furthermore, Lemma \ref{lem:edge_flip} ensures that $\det(D)>0$ and $\det(D')>0$ as well. 
We will show that rays through $T$ intersect $C \cup_e C'$ exactly once and $D \cup_{e'} D'$ exactly once because of the positive determinant condition.

To see this, let $\ell \in \pos \{\vtx_0,\vtx_1,\vtx_2,\vtx_3\}$ be a ray from the origin which passes through $T$. 
Since $\det(C)>0$ and $\det(C')>0$ we know that the vertices $\vtx_1$ and $\vtx_3$ lie on different sides of the hyperplane $H_{02} = \{ X \in \R^3 \mid \det(\vtx_0 \mid \vtx_2 \mid X) = 0 \}$.
The ray $\ell$ either is contained in the hyperplane $H_{02}$ or lies on the same side as either $\vtx_1$ or $\vtx_3$.
Thus $\ell$ intersects $C \cup_e C'$ either in their common edge $e$ or it only intersects one of the triangles $C$ or $C'$.
In any case, there is it most one such intersection by convexity. 
Similarly, positivity $\det(D)$ and  $\det(D')$ imply that the ray $\ell$ meets $D \cup_{e'} D'$ exactly once.
 
Now suppose there was one ray which passes through $C \cup_e C'$ strictly before it passes through $D \cup_{e'} D'$ and another ray which passes through $D \cup_{e'} D'$ strictly before it passes through $C \cup_e C'$. These rays pass through the interior of $T$. By the convexity of $T$ there is a ray from the origin passing through the interior of $T$ which intersects $C \cup_e C'$ and $D \cup_{e'} D'$ at the same time. This contradicts the assumption that $T$ is nondegenerate. It follows that either $C \cup_e C'$ or $D \cup_{e'} D'$ is the outer hull.

Finally, let us prove that the sign of the outitude along $e$ is as is as claimed. Since the value $e^+ e^-$ is strictly positive by definition, we may instead study the sign of $\frac{1}{e^+ e^-} \Out(e)$ under the relevant conditions.
Now the nice feature of the value $$\frac{1}{e^+ e^-} \Out(e) =  \left( \det(D) + \det(D') - \det(C) - \det(C') \right)$$ is that it equals the Euclidean volume of the tetrahedron $T$ if and only if $C \cup_e C^\prime$ is the outer hull. Clearly, in this case the outitute along $e$ is nonnegative and vanishes if and only if $T$ is degenerate. 
Similarly, the Euclidean volume of $T$ is given by $\frac{-1}{e^+ e^-} \Out(e)$ if and only if $D \cup D^\prime$ is the outer hull, finishing the proof. 
\end{proof}

Let $x \in \decTtf(S)$  be a doubly-decorated real convex projective structure on $S$ and $\Delta \subset S$ an ideal triangulation.
As discussed at the end of Section \ref{sec:decorated_higher_teichmueller_space}, $x$ induces a lift of $\Delta$ to a concrete decorated triangulation $\widetilde{\Delta}$, unique up to decorated isomorphism.
This enables us to define the \emph{outitude} along an edge $e \in \Delta$, denoted by $\Out_x(e)$, as the outitude of one of its lifts $\widetilde{e} \in \widetilde{\Delta}$. Note that the latter is well-defined by Lemma~\ref{lem:abstract_triangle_equivalence}.

\begin{lem} \label{lem:outitude_intrinsic}
	Having fixed a doubly-decorated real projective structure on $S$, the sign of the outitude along a given edge is an intrinsic property of this structure.
\end{lem}

\begin{proof}
	Fix a doubly-decorated real projective structure on $S$ and fix an (un-oriented) edge $e$ of $S$. Let $\Delta$ be a triangulation of $S$ that contains the edge $e$. 
	Again, lift $\Delta$ to a concrete decorated triangulation in $\R^3$. The outitude $\Out(e)$ defined in \eqref{eq:outitude_def1} is a function of the volume of a tetrahedron whose vertices are those of the ideal triangles containing a chosen lift of $e$. 
	The edge parameters assigned to that lifted edge are determined by the concrete decorated triangulation. 
	However any two such lifts differ by the action of an element of the holonomy group, which is a subgroup of $\SL_3(\R)$ and hence preserves volume and edge parameters. Therefore the outitude condition is independent of the choice of lift. Similarly, any two concrete decorated triangulations realising the doubly-decorated real projective structure on $S$ differ by the action of $\SL_3(\R)$ so $\Out(e)$ is independent of the choice of such a realisation.
\end{proof}

We say that a triangulation $\Delta$ of $S$ is \emph{canonical} for $x$ if the outitude along every edge in $e\in \Delta$ is nonnegative.

\begin{cor} \label{cor:canonical}
Let $x \in \decTtf(S)$ and let $\Delta$ be a triangulation of $S$. The following are equivalent.
\begin{enumerate}[label=(\roman*)]
	\item $\Delta$ is canonical for $x$.
	\item The cell decomposition of $S$ that is formed by those edges of $\Delta$ along which the outitude is strictly positive is the ideal cell decomposition determined by Cooper and Long's \cite{CoLo-Epstein-Penner-2015} convex hull construction for $x$.
\end{enumerate}
\end{cor}
\begin{proof}
Suppose that $\Delta$ is canonical for $x=(\Omega, \Gamma, \phi, (\vec_dec , \covec_dec))$. 
Let $\Delta'$ denote the ideal cell decomposition of $S$ determined by Cooper and Long's convex hull construction for $x$. 
Recall that $\Delta'$ is obtained by taking the convex hull $\mathcal{C}$ of the vectors in $(\vec_dec , \covec_dec)$ and then projecting the edges of $\mathcal{C}$ to the $\Omega / \Gamma$.
We identify $\Delta$ with a lift of $\Delta$ to a concrete decorated triangulation in $\R^3$. We have seen in Lemma \ref{lem:convex} that edges with positive outitude correspond to edges in the outer hull $\Delta$. These outer edges determine edges of $\Delta'$ by the convex hull construction.
Furthermore, those edges of $\Delta$ with vanishing outitude do not appear as edges of the convex hull, and therefore they are not contained in $\Delta'$.

Conversely suppose that the edges of $\Delta$ along which the outitude is strictly positive produce the ideal cell decomposition $\Delta' \subseteq \Delta$ determined by the convex hull construction. 
Then the lift of $\Delta$ must bound a convex set whose edges are exactly lifts of edges of $\Delta'$. 
Therefore the outitudes along edges of $\Delta$ must be nonnegative again by Lemma \ref{lem:convex}.
\end{proof}

Fix a triangulation $\Delta$ of $S$ and let $x \in \decTtf(S)$.
Denote by $\alpha = \psi(x) \in \mathcal{A}$ the corresponding $\mathcal{A}$-coordinates, where $\psi$ is the bijection from Theorem \ref{thm:A-coords}.
In this case we denote $\Out_\alpha(e) := \Out_x(e)$ interchangeably.

\begin{lem} \label{lem:outitude_A-coords}
Let $\Delta \subset S$ be an ideal triangulation and let $\alpha \in \mathcal{A}$ be a doubly-decorated convex projective structure described via $\mathcal{A}$-coordinates on $\Delta$.
For an edge $e \in \Delta$ we denote the triangle and edge coordinates of $\alpha$ around $e$ as depicted in Figure \ref{fig:outitude_A-coords}.
Then in terms of $\alpha$ the outitude along $e$ is given by
\begin{equation}\label{eq:A_Out}
\Out_\alpha(e) = A (e^+c^+ + e^-d^- - e^+ e^- ) + B ( e^+b^+ + e^-a^- - e^+ e^-) \enspace .
\end{equation}
\end{lem}
\begin{proof}
Following the procedure described at the start of \S\ref{subsec:A-coords} we lift $\Delta$ to an concrete decorated triangulation \textcolor{red}{$\Lambda$} in $\R^3$ according to the doubly-decorated structure $\alpha \in \mathcal{A}$.
Let $(R, C)$ and $(R', C')$ be two concrete decorated triangles in \textcolor{red}{$\Lambda$} that share a lift of $e$.
We denote the four matrices that contribute to $\Out_\alpha(e)$ by $C = ( \vtx_0 \mid \vtx_1 \mid \vtx_2 )$ and $C' = ( \vtx_0 \mid \vtx_2 \mid \vtx_3 )$ as well as $D = ( \vtx_0 \mid \vtx_1 \mid \vtx_3 )$ and $D' = ( \vtx_3 \mid \vtx_1 \mid \vtx_2 )$.
Without loss of generality we may assume $C = \lambda \Id_3$ where $\lambda^3 = \det(C) = A$. Let $\covtx_i$ be the covector whose kernel contains the ideal vertex $\vtx_i$. Then we may write the covectors $\covtx_0$ and $\covtx_2$ as 
\[
\covtx_0 = \begin{pmatrix} 0 & y_0 & z_0 \end{pmatrix} \quad \text{and} \quad \covtx_2 = \begin{pmatrix} x_2 & y_2 & 0 \end{pmatrix} \enspace .
\]
Let $C_3 = \begin{pmatrix} X_3 & Y_3 & Z_3 \end{pmatrix}^T$. Then we obtain
\begin{align*}
\Out_\alpha(e) &= e^+ e^- \left(\det(D) + \det(D') -\det(C) - \det(C') \right) \\
&= e^+ e^- \lambda^2 (X_3 + Y_3 + Z_3 - \lambda) \\
&= \lambda^4 z_0 x_2 ( X_3 + Y_3 + Z_3 - \lambda ) \enspace .
\end{align*}
Conversely we have
\begin{align*}
& A (e^+c^+ + e^-d^- - e^+ e^- ) + B ( e^+b^+ + e^-a^- - e^+ e^-)  \\
&= \lambda^3( \lambda z_0 ( x_2 X_3 + y_2 Y_3 ) + \lambda x_2 (y_0 Y_3 + z_0 Z_3) - \lambda^2 z_0 x_2 ) - \lambda^2Y_3 ( \lambda^2 z_0 y_2 + \lambda^2 x_2y_0 - \lambda^2z_0x_2  ) \\
&= \lambda^4 \big( \left( z_0x_2X_3 + z_0y_2Y_3 + x_2y_0Y_3 + z_0Z_3x_2 - \lambda z_0 x_2 ) - Y_3 ( z_0 y_2 + x_2 y_0 - z_0 x_2 \right) \big) \\
&= \lambda^4 z_0 x_2( X_3 + Z_3 + Y_3 - \lambda) \enspace .
\end{align*}
This completes the lemma.
\end{proof}

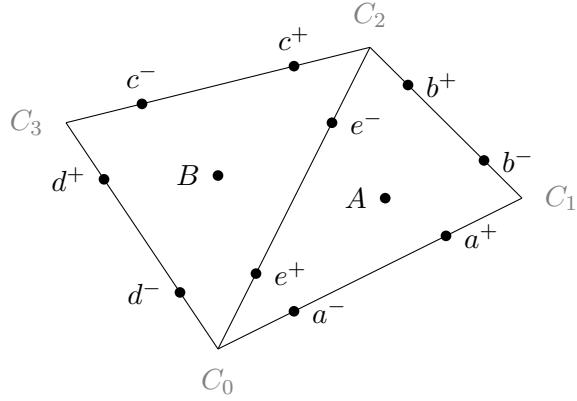
\begin{figure}
\center
\begin{tikzpicture}
	\draw (4, 4) edge  (6, 2); 
	\draw (6, 2) edge (2,0); 
	\draw (2, 0) edge (4, 4); 
	\draw (4, 4) edge (0, 3); 
	\draw (0, 3) edge (2, 0); 

	\node [label = below : {\textcolor{gray}{$\vtx_0$}} ] (a) at (2, 0) {};
	\node [label = right : {\textcolor{gray}{$\vtx_1$}} ] (a) at (6, 2) {};
	\node [label = above : {\textcolor{gray}{$\vtx_2$}} ] (a) at (4, 4) {};
	\node [label = left : {\textcolor{gray}{$\vtx_3$}} ] (a) at (0, 3) {};
	\node [circle, fill = black, inner sep = 0pt, minimum size = 4pt, label = left : {$A$} ] (A) at (4.2, 2) {};	
	\node [circle, fill = black, inner sep = 0pt, minimum size = 4pt, label = left : {$B$} ] (B) at (2, 2.3) {};
	\node [circle, fill = black, inner sep = 0pt, minimum size = 4pt, label = right : {$e^+$} ] (ep) at (2.5, 1) {};
	\node [circle, fill = black, inner sep = 0pt, minimum size = 4pt, label = right : {$e^-$} ] (em) at (3.5, 3) {};
	\node [circle, fill = black, inner sep = 0pt, minimum size = 4pt, label = right : {$a^-$} ] (am) at (3, 0.5) {};
	\node [circle, fill = black, inner sep = 0pt, minimum size = 4pt, label = right : {$a^+$} ] (ap) at (5, 1.5) {};
	\node [circle, fill = black, inner sep = 0pt, minimum size = 4pt, label = right : {$b^+$} ] (bp) at (4.5, 3.5) {};
	\node [circle, fill = black, inner sep = 0pt, minimum size = 4pt, label = right : {$b^-$} ] (bm) at (5.5, 2.5) {};
	\node [circle, fill = black, inner sep = 0pt, minimum size = 4pt, label = above : {$c^+$} ] (cp) at (3, 3.75) {};
	\node [circle, fill = black, inner sep = 0pt, minimum size = 4pt, label = above : {$c^-$} ] (cm) at (1, 3.25) {};
	\node [circle, fill = black, inner sep = 0pt, minimum size = 4pt, label = left : {$d^+$} ] (dp) at (0.5, 2.25) {};
	\node [circle, fill = black, inner sep = 0pt, minimum size = 4pt, label = left : {$d^-$} ] (dm) at (1.5, 0.75) {};
\end{tikzpicture}
\caption{The edge and triangle parameters on two adjacent triangles. $\Out(e)$ is determined by the two triangle parameters $A=\det( \vtx_0 \mid \vtx_1 \mid \vtx_2 )$ and $B=\det( \vtx_0 \mid \vtx_2 \mid \vtx_3 )$ and the six edges parameters proximal to $\vtx_0$ and $\vtx_2$, namely $e^+, e^-,a^-,b^+,c^+$ and $d^-$.} \label{fig:outitude_A-coords}
\end{figure}


\subsection{Description of the putative cells in moduli space}
\label{subsec:putative cells}

We arrive at the following $\mathcal{A}$-coordinate description of the sets defined in \eqref{eq:def_open_cell} and \eqref{eq:def_closed_cell}.   
\begin{cor} \label{cor:cell_decomposition}
Let $\Delta$ be an ideal cell decomposition of $S$ and let $\Delta'$ be an ideal triangulation refining $\Delta$. 
Let $\{ b_j  \}_{ j \in J }$ denote the edges of $\Delta$ and let $\{ a_k \}_{ k \in K }$ denote the set of edges in $\Delta' \setminus \Delta$.
Then the image of the sets defined in \eqref{eq:def_open_cell} and \eqref{eq:def_closed_cell} under the bijection $\psi = \psi_{\Delta'}$ from Theorem \ref{thm:A-coords} is given by the semi-algebraic sets
\begin{align}
\psi( \mathring{ \mathcal{C}}( \Delta ) ) &= \left\{ \alpha \in \mathcal{A}_{\Delta'} \mid \Out_\alpha( b_j ) > 0 \hspace{2mm} \forall \hspace{2mm} j \in J \text{ and } \Out_\alpha( a_k ) = 0 \hspace{2mm} \forall \hspace{2mm} k \in K \right\} \enspace , \\
\psi( \mathcal{C}( \Delta ) ) &= \left\{ \alpha \in \mathcal{A}_{\Delta'} \mid \Out_\alpha( b_j ) \geq 0 \hspace{2mm} \forall \hspace{2mm} j \in J \text{ and } \Out_\alpha( a_k ) = 0 \hspace{2mm} \forall \hspace{2mm} k \in K \right\} \enspace .
\end{align}
\end{cor}


\subsection{An algorithm to compute the canonical cell decomposition of a surface}
\label{subsec:flip algo}

A particularly nice feature of the $\mathcal{A}$-coordinates is that they provide us with a nice algorithm to determine a canonical triangulation for a given (doubly-)decorated convex projective structure.
The key idea of this algorithm has been established by Tillmann and Wong \cite{TiWo-algorithm-2016}. 
It can be explained as follows. 
Assume we are given a convex projective structure of finite area on a surface $S$ together with a vector decoration $\vec_dec$.
To check if a triangulation $\Delta \subset S$ is canonical, consider a lift to a concrete triangulation $\widetilde{\Delta}$ with vertex set $\vec_dec$.
Now $\Delta$ is canonical if each edge $e \in \Delta$ satisfies the local convexity condition, meaning that the union of the hulls of the two concrete triangles adjacent to a lift of $e$ 
is the bottom of the (possibly degenerate) tetrahedron formed by their vertices.
If an edge violates the local convexity condition, we may flip it to introduce a new edge that satisfies the local convexity condition.
Tillmann and Wong showed that flipping edges that violate the local convexity condition terminates and ultimately produces a canonical triangulation.

However, the latter algorithm involves a lot of additional computations.
For example, one needs to compute a holonomy representation of the chosen convex projective structure in order to determine a sufficiently large subset of the decoration $\vec_dec$.
The advantage of using $\mathcal{A}$-coordinates is that all the necessary data is intrinsic. This provides the algorithm of Tillmann and Wong with a more efficient data structure.

Let us describe the algorithm using $\mathcal{A}$-coordinates.
Fix an ideal triangulation $\Delta \subset S$ along with the corresponding $\mathcal{A}$-coordinates $\mathcal{A}_{\Delta}$.
Let $\alpha \in \mathcal{A}_{\Delta}$ be a doubly-decorated convex projective structure.
We wish to describe an ideal triangulation $\Delta(\alpha)$ that is canonical for $\alpha$.
Each edge $e\in \Delta$ comes with an associated outitude value $\Out_\alpha(e)$ described in \eqref{eq:A_Out}.
Note that by Lemma \ref{lem:convex} $\Out_\alpha(e)$ is nonnegative if and only if $e$ satisfies the local convexity condition.
Thus, we may apply the algorithm of Tillmann and Wong by successively flipping edges whose outitute is negative.
This produces a sequence of triangulations
\[ \left( \Delta = \Delta_0, \Delta_1, \ldots , \Delta_{m-1}, \Delta_m = \Delta(\alpha) \right) \]
where each $\Delta_{i+1}$ can be obtained from $\Delta_i$ by performing an edge flip and the final triangulation is canonical for $\alpha \in \mathcal{A}_\Delta$.

\begin{exm}\label{ex:flipping}
The following example is based on an implementation of the flip-algorithm for $\A$-coordinates in \texttt{polymake} \cite{polymake}.
Consider the once-punctured torus $S_{1,1}$. 
We choose a triangulation $\Delta_0$ of $S_{1,1}$ as in the top left of Figure \ref{fig:example_flipping} (with the usual edge identifications) together with the depicted edge and triangle parameters \[
\alpha_0 = (A,B,a^+,a^-,b^+,b^-,c^+,c^-) =(\tfrac{107}{12},\tfrac{95}{18},1,1,\tfrac{17}{6},\tfrac{25}{12},\tfrac{1145}{72},\tfrac{1289}{72}) \in \mathcal{A}_{\Delta_0} \cong \R_{>0}^8 \enspace .
\]
This defines a decorated convex projective structure $x = \psi_{\Delta_0}^{-1}(\alpha_0) \in \decTtf(S_{1,1})$ on the once punctured torus.
The three outitude values are 
\[
\Out_{x}(a) = 265.629 \qquad \Out_{x}(b) = 548.357 \qquad  \Out_{x}(c) = -3234.55 \enspace .
\]
Since the outitude along edge $c \in \Delta_0$ (colored red) is negative, we know that $\Delta_0$ is not canonical for $x$.
However, the algorithm tells us that we should flip this edge to find a canonical triangulation.
Flipping edge $c$ introduces a new edge $d$ and results in a new triangulation $\Delta_1$.
Using the coordinate change equations \eqref{eq:flip_triang} we see that the $\mathcal{A}_{\Delta_1}$-coordinates of $x$ are given by
$\alpha_1 = (\tfrac{3}{2},\tfrac{4}{3},1,1,\tfrac{17}{6},\tfrac{25}{12},1,\tfrac{1}{2}) \in \mathcal{A}_{\Delta_1}.$ 
The outitude values for the triangulation $\Delta_1$ are 
 \[
 \Out_{x}(a) = 6.36111 \qquad \Out_x(b) = -4.91898 \qquad  \Out_x(d) = 5.68056 \enspace 
 \]
and as expected $\Out_{x}(d)$ is positive.
Unfortunately, now the outitude along edge $b \in \Delta_1$ (colored blue) is negative.
We remove $b$ to introduce its flipped counterpart $e$ and obtain yet another triangulation $\Delta_2$.
The $\mathcal{A}_{\Delta_2}$-coordinates of $x$ are given by $\alpha_2 = (1,1,1,1,1,\tfrac{3}{2},1,\tfrac{1}{2}) \in \mathcal{A}_{\Delta_1}.$  
Finally, the outitude values in $\Delta_2$ are 
\[
\Out_{x}(a) = 2 \qquad \Out_x(b) = 1.25 \qquad  \Out_x(d) = 2.25 \enspace 
\]
and since these are all positive the triangulation $\Delta_2$ is canonical for the decorated convex projective structure $x \in \decTtf(S_{1,1})$ we 
chose when we fixed a set of $\A$-coordinates on $\Delta_0$.
\begin{figure}
	\centering
	\definecolor{myblue}{rgb}{0.19, 0.55, 0.91}
	\definecolor{myred}{rgb}{1.0, 0.21, 0.37}
	\begin{tikzpicture}
	\node at (-5.5,0) {\begin{tikzpicture}
	\begin{scriptsize}
	\node [circle, fill = red1, inner sep = 0pt, minimum size = 4pt, label = above : {$\tfrac{107}{12}$} ] (A) at (1,2) {};
	\node [circle, fill = red1, inner sep = 0pt, minimum size = 4pt, label = above : {$\tfrac{95}{18}$} ] (B) at (2,1) {};	
	\draw (0,0) -- (0,3) -- (3,3) -- (3,0) -- cycle;
	\draw[myred, line width=0.7pt] (0,0) -- (3,3);
	\node [circle, fill = blue1, inner sep = 0pt, minimum size = 4pt, label = below : {$1$}] at ( 1, 0 ) {};
	\node [circle, fill = blue1, inner sep = 0pt, minimum size = 4pt, label = below : {$1$}] at ( 2, 0 ) {};
	\node [circle, fill = blue1, inner sep = 0pt, minimum size = 4pt, label = left : {$\tfrac{17}{6}$}] at ( 0, 1 ) {}; 
	\node [circle, fill = blue1, inner sep = 0pt, minimum size = 4pt, label = left : {$\tfrac{25}{12}$}] at ( 0, 2 ) {}; 
	\node [circle, fill = blue1, inner sep = 0pt, minimum size = 4pt, label = above : {$1$}] at ( 1, 3 ) {};
	\node [circle, fill = blue1, inner sep = 0pt, minimum size = 4pt, label = above : {$1$}] at ( 2, 3 ) {};
	\node [circle, fill = blue1, inner sep = 0pt, minimum size = 4pt, label = right : {$\tfrac{17}{6}$}] at ( 3, 1 ) {}; 
	\node [circle, fill = blue1, inner sep = 0pt, minimum size = 4pt, label = right : {$\tfrac{25}{12}$}] at ( 3, 2 ) {}; 
	\node [circle, fill = blue1, inner sep = 0pt, minimum size = 4pt, label = above : {$\tfrac{1145}{72}$}] at ( 2, 2 ) {};
	\node [circle, fill = blue1, inner sep = 0pt, minimum size = 4pt, label = above : {$\tfrac{1289}{72}$}] at ( 1, 1 ) {};	
	\end{scriptsize}	
	\end{tikzpicture}};
	\node at (0,0) {\begin{tikzpicture}
	\begin{scriptsize}
	\node [circle, fill = red1, inner sep = 0pt, minimum size = 4pt, label = above : {$\tfrac{3}{2}$} ] (A) at (1,2) {};
	\node [circle, fill = red1, inner sep = 0pt, minimum size = 4pt, label = above : {$\tfrac{4}{3}$} ] (B) at (2,1) {};	
	\draw (0,0) -- (0,3) -- (3,3) -- (3,0) -- cycle;
	\draw[myblue, line width=0.7pt] (0,0) -- (3,3);
	\node [circle, fill = blue1, inner sep = 0pt, minimum size = 4pt, label = below : {$1$}] at ( 1, 0 ) {};
	\node [circle, fill = blue1, inner sep = 0pt, minimum size = 4pt, label = below : {$1$}] at ( 2, 0 ) {};
	\node [circle, fill = blue1, inner sep = 0pt, minimum size = 4pt, label = left : {$1$}] at ( 0, 1 ) {}; 
	\node [circle, fill = blue1, inner sep = 0pt, minimum size = 4pt, label = left : {$\tfrac{1}{2}$}] at ( 0, 2 ) {}; 
	\node [circle, fill = blue1, inner sep = 0pt, minimum size = 4pt, label = above : {$1$}] at ( 1, 3 ) {};
	\node [circle, fill = blue1, inner sep = 0pt, minimum size = 4pt, label = above : {$1$}] at ( 2, 3 ) {};
	\node [circle, fill = blue1, inner sep = 0pt, minimum size = 4pt, label = right : {$1$}] at ( 3, 1 ) {}; 
	\node [circle, fill = blue1, inner sep = 0pt, minimum size = 4pt, label = right : {$\tfrac{1}{2}$}] at ( 3, 2 ) {}; 
	\node [circle, fill = blue1, inner sep = 0pt, minimum size = 4pt, label = above : {$\tfrac{25}{12}$}] at ( 2, 2 ) {};
	\node [circle, fill = blue1, inner sep = 0pt, minimum size = 4pt, label = above : {$\tfrac{17}{6}$}] at ( 1, 1 ) {};	
	\end{scriptsize}	
	\end{tikzpicture}};
	\node at (5.5,0) {\begin{tikzpicture}
	\begin{scriptsize}
	\node [circle, fill = red1, inner sep = 0pt, minimum size = 4pt, label = above : {$1$} ] (A) at (1,2) {};
	\node [circle, fill = red1, inner sep = 0pt, minimum size = 4pt, label = above : {$1$} ] (B) at (2,1) {};	
	\draw (0,0) -- (0,3) -- (3,3) -- (3,0) -- cycle;
	\draw (0,0) -- (3,3);
	\node [circle, fill = blue1, inner sep = 0pt, minimum size = 4pt, label = below : {$1$}] at ( 1, 0 ) {};
	\node [circle, fill = blue1, inner sep = 0pt, minimum size = 4pt, label = below : {$1$}] at ( 2, 0 ) {};
	\node [circle, fill = blue1, inner sep = 0pt, minimum size = 4pt, label = left : {$\tfrac{3}{2}$}] at ( 0, 1 ) {}; 
	\node [circle, fill = blue1, inner sep = 0pt, minimum size = 4pt, label = left : {$1$}] at ( 0, 2 ) {}; 
	\node [circle, fill = blue1, inner sep = 0pt, minimum size = 4pt, label = above : {$1$}] at ( 1, 3 ) {};
	\node [circle, fill = blue1, inner sep = 0pt, minimum size = 4pt, label = above : {$1$}] at ( 2, 3 ) {};
	\node [circle, fill = blue1, inner sep = 0pt, minimum size = 4pt, label = right : {$\tfrac{3}{2}$}] at ( 3, 1 ) {}; 
	\node [circle, fill = blue1, inner sep = 0pt, minimum size = 4pt, label = right : {$1$}] at ( 3, 2 ) {}; 
	\node [circle, fill = blue1, inner sep = 0pt, minimum size = 4pt, label = above : {$\tfrac{1}{2}$}] at ( 2, 2 ) {};
	\node [circle, fill = blue1, inner sep = 0pt, minimum size = 4pt, label = above : {$1$}] at ( 1, 1 ) {};	
	\end{scriptsize}	
	\end{tikzpicture}};
	\end{tikzpicture}
	\includegraphics[width=0.9\linewidth]{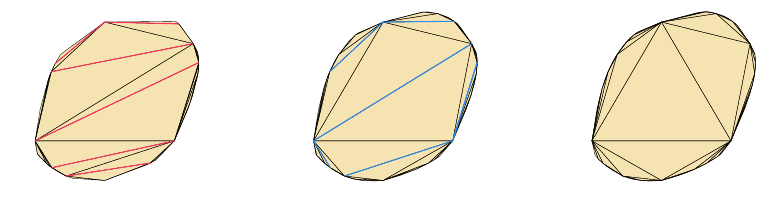}
	\caption{The three triangulations of the flip algorithm from Example \ref{ex:flipping}. The top row depicts three different $\mathcal{A}$-coordinate charts for the convex projective structure $x \in \decTtf(S_{1,1})$. The bottom row shows a finite part of the three triangulations developed into the corresponding convex set $\Omega$.} \label{fig:example_flipping}
\end{figure}
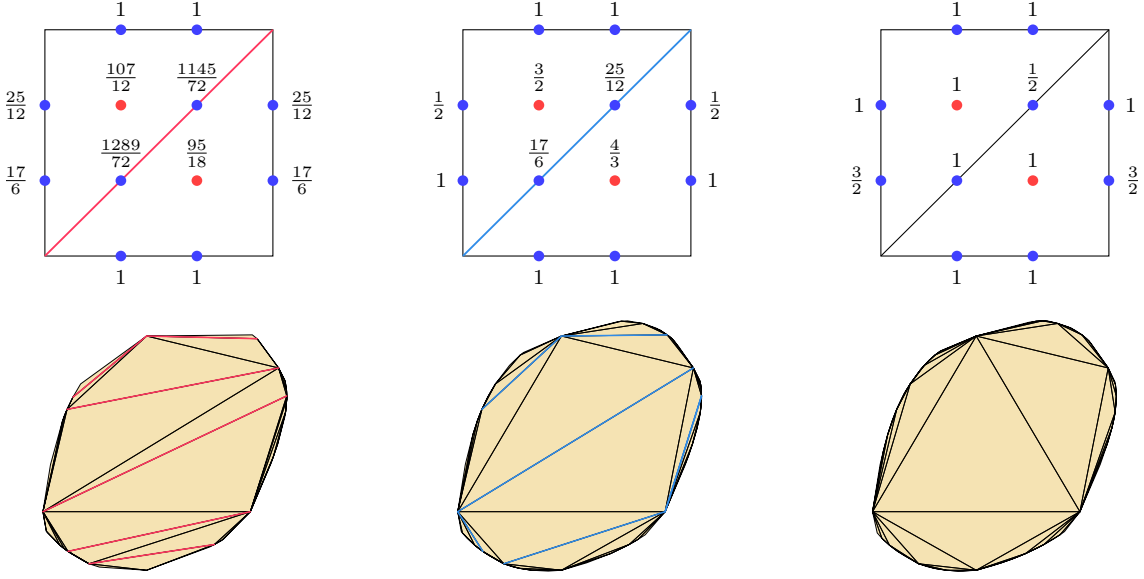
\end{exm}


\section{The cell decomposition of moduli space}
\label{sec:main_proof}
\label{subsec:main_results}

Let $\Delta \subset S$ be an 
 ideal cell decomposition, refined by an ideal triangulation $\Delta^\prime \subset S$.
This section is devoted to showing that 
$\mathring{\mathcal{C}}(\Delta)$ is a cell of real dimension $\mid \underline{E} \mid  + \mid T \mid -  \mid K \mid $ where $ \mid K \mid $ is the number of 
unoriented edges of $\Delta^\prime \setminus \Delta$ and we recall that $\underline{E}$ denotes the set of oriented edges and $T$ the set of triangles of 
$\Delta^\prime$.
If there is no ambiguity as to the ideal triangulation $\Delta^\prime$ then we will identify $\mathring{\mathcal{C}}(\Delta)$ with its image under $\psi_{\Delta^\prime}$ inside $\A_{\Delta^\prime}$ without further comment.

We first study the case in which $\Delta$ is an ideal triangulation of $S$.
In what follows we will denote the corresponding $\A$-coordinates by $(\trCoords , \eCoords) \in \A \cong \R_{>0}^{ T + \underline{E} }$.
That is, we denote the triangle parameters by $\trCoords \in \R_{>0}^T$ and the edge parameters by $\eCoords \in \R_{>0}^{\underline{E} }$.

\begin{lem}\label{lem:main1}
Let $\Delta \subset S$ be an ideal triangulation. 
Fix an assignment $\trCoords \in \R^T_{>0}$ of positive real numbers to the ideal triangles of $\Delta$. Then the semi-algebraic set
\[
X_{\trCoords} := \left\{ \eCoords \in \R_{>0}^{\underline{E} } ~\mid~ \Out_{( \trCoords , \eCoords )}(e) > 0 \hspace{2mm} \text{ for all } \hspace{2mm}  e \in \Delta \right\} \subseteq \R_{>0}^{ \underline{E} }
\]
is homeomorphic to $\R^{ \underline{E} }_{>0}$.
\end{lem}
\begin{proof}
Let $ B_{\mathbbm{1}}(0.1)$ denote the set of points in $\R^{ \underline{E} }_{>0}$ within a Euclidean distance of $0.1$ from the all ones vector $\mathbbm{1} := ( 1, 1, \dots, 1) \in \R^{ \underline{E} }_{>0}$. 
Recall that $\Out_{( \trCoords , \eCoords )}(e)>0$ if and only if an inequality of the following form is true, cf. \eqref{eq:A_Out},
\begin{equation}\label{eq:A_Out_pf}
A (e^+c^+ + e^-d^- - e^+ e^- ) + B ( e^+b^+ + e^-a^- - e^+ e^-)  > 0 \enspace
\end{equation}
for
 $a, b, c, d \in \Delta$ each contained in the closure of a common ideal triangle with $e$ as depicted in Figure~\ref{fig:outitude_A-coords}. 
Regardless of the choice of $\trCoords \in \R^T_{>0}$ we have $B_{\mathbbm{1}}(0.1) \subset X_{\trCoords}$.  
We will construct a homeomorphism $B_{\mathbbm{1}}(0.1) \cong X_{\trCoords}$.

For all edge parameters $\eCoords \in \R_{>0}^{ \underline{E} }$, we consider a linear deformation to the center $\mathbbm{1}$ of $B_{\mathbbm{1}}(0.1)$ given by
\begin{align*}
\phi _{\eCoords:} [0, 1] &\rightarrow \R_{>0}^{ \underline{E} }  \enspace ,\\
t  &\mapsto t \eCoords + (1-t)\mathbbm{1} \enspace .
\end{align*}
The coordinate function of $\phi_{\eCoords}$ corresponding to an edge parameter $e^\pm$ will be denoted  by $\phi_{e^\pm}$.
Suppose that at  $(\trCoords , \eCoords)$ we have positive outitude along edge $e$, i.e. inequality \eqref{eq:A_Out_pf} holds.
We now show that this implies that at the deformed points $(\trCoords , \phi_{\trCoords}(t) )$ the outitude along $e$ is positive as well, that is 
\begin{align*}
 ~&A \left( \phi_{ e^+}(t) \phi_{c^+}(t) +   \phi_{e^-}(t) \phi_{d^-}(t) - \phi_{e^+}(t) \phi_{e^-}(t) \right) \\
 + ~&B \left( \phi_{e^+}(t) \phi_{b^+}(t) + \phi_{e^-}(t) \phi_{a^-}(t) - \phi_{e^+}(t) \phi_{e^-}(t) \right) > 0 \enspace ,
\end{align*}
for all $t \in [0,1]$.
We replace the functions $\phi_{e^{\pm}}(t)$ by their definitions in the above inequality and collect terms in $t^2$, $t(1-t)$ and $(1-t)^2$ to obtain
\begin{align*}
 ~&A \mbox{\Large (} ( t e^+ + 1 - t  )(t c^+ + 1 - t )  + ( t e^- + 1 - t)( t d^- + 1 - t ) \\
~& \qquad- ( t e^+ + 1 - t )  ( t e^- + 1 - t ) \mbox{\Large )} \\
+ ~&B \mbox{\Large (} (  t e^+ + 1 - t )(  t b^+ +  1 - t  ) + (  t e^- + 1-t) ( t a^- + 1 - t ) \\
 ~&\qquad - ( t e^+ + 1-t) ( t e^- + 1 - t ) \mbox{\Large )} \\
= \qquad &t^2 A \mbox{\Large (} (e^+ c^+ + e^- d^- -  e^+ e^-  )  + B ( e^+ b^+ + e^- a^- - e^+ e^-) \mbox{\Large )} \\
 + ~&t(1-t) \mbox{\Large (} A ( c^+ + d^- ) + B ( a^+  + b^- ) \mbox{\Large )} \\
 + ~&(1-t)^2 ( A + B ) \enspace .
\end{align*}
The coefficient of $t^2$ is exactly $\Out_{(\trCoords,\eCoords)}(e)$ and is therefore positive by assumption. The other terms are nonnegative since $t\in[0,1]$ and because every triangle and edge parameter is positive. 

This shows that $X_{\trCoords} \subseteq \R_{>0}^{ \underline{E} }$ is a 
star-shaped set with center $\mathbbm{1}$. 
It is open because it is defined by a finite family of algebraic inequalities and is therefore the finite intersection of open sets.
As such it is homeomorphic to any 
$| \underbar{E} |$-dimensional open ball. 
In particular, since $X_{\trCoords}$ is semi-algebraic, it is homeomorphic to $B_{\mathbbm{1}}(0.1)$ by the canonical homeomorphism along rays originating from $\mathbbm{1}$.
\end{proof}

We now show that the deformation to $\mathbbm{1}$ described in Lemma~\ref{lem:main1} induces a product structure on $\mathring{\mathcal{C}}(\Delta)$.

\begin{thm}\label{thm:main}
Let $\Delta$ be an ideal triangulation of $S$. Then the semi-algebraic set
\[
\mathring{\mathcal{C}}(\Delta) = \left\{ (\trCoords,\eCoords) \in \A \mid \Out_{(\trCoords,\eCoords)}(e) > 0 \hspace{2mm} \text{for all} \hspace{2mm}  e \in \Delta \right\}
\]
is homeomorphic to $\R^{T+ \underline{E} }_{>0}$.
\end{thm}
\begin{proof}
Lemma \ref{lem:main1} shows that for all $ \trCoords \in \R^T_{>0}$ there exists a homeomorphism 
\[
f_{\trCoords} : X_{\trCoords} \rightarrow \{ \trCoords \} \times B_{\mathbbm{1}}(0.1) \subset \R^{T+ \underline{E} }_{>0} \enspace .
\]
Since the components $\partial \clCell (\Delta)$ are defined by algebraic functions the we may glue the $f_{\trCoords}$ together to obtain the homeomorphism
\begin{align*}
f: \mathring{\mathcal{C}}(\Delta) &\rightarrow \R^T_{>0} \times B_{\mathbbm{1}}(0.1) \enspace , \\
(\trCoords , \eCoords ) & \mapsto (\trCoords , \phi_{\trCoords} ( \eCoords ) ) \enspace .
\end{align*}
Clearly $\R^T_{>0} \times B_{\mathbbm{1}}(0.1) \cong \R^{T + \underline{E} }$, completing the proof.
\end{proof}

What is left is to study the sets $\mathring{ \mathcal{C}}(\Delta)$ in the case where $\Delta$ is an ideal cell decomposition of $S$
but not an ideal triangulation.
In order to prove cellularity of $\mathring{ \mathcal{C}}(\Delta)$ we may consider any triangulation $\Delta^\prime$ refining $\Delta$, i.e. $\Delta \subset \Delta^\prime$.
The homeomorphic image of $\mathring{ \mathcal{C}}(\Delta)$ in the corresponding $\A_{\Delta^\prime}$-coordinates is given by the semi-algebraic set 
\begin{equation}\label{eq:algebraic_bd_cell}
\psi_{ \Delta^\prime }( \mathring{ \mathcal{C} } ( \Delta ) ) = \left\{ \alpha \in \A_{\Delta^\prime} ~|~ \Out_\alpha ( e ) > 0 \hspace{2mm} \forall \hspace{2mm} e \in \Delta ~ , ~ \Out_\alpha(e)=0 \hspace{2mm} \forall \hspace{2mm} e \in \Delta^\prime \setminus \Delta \right\} \enspace .
\end{equation}

Note that cellularity of $\mathring{ \mathcal{C}}(\Delta)$ follows from cellularity of \eqref{eq:algebraic_bd_cell} independent of the ideal triangulation $\Delta'$ that is chosen to refine $\Delta$.
Indeed, for two triangulations $\Delta^\prime,\Delta^{\prime \prime}$ refining $\Delta$ the transition homeomorphism $R_{\Delta^\prime,\Delta^{\prime \prime }}$ defined in \eqref{eq:def_chart_transition} restricts to a homeomorphism from $\psi_{\Delta^\prime }(\mathring{ \mathcal{C}}( \Delta ) ) $ to $\psi_{\Delta^{\prime \prime}}(\mathring{ \mathcal{C} }( \Delta ) )$.   
In the following we therefore restrict to a special class of triangulations that refine $\Delta$, namely those where all polygons of $S$ in the complement of $\Delta$ have a standard triangulation. 

\begin{lem}\label{lem:surjective}
Let $\Delta$ be an ideal cell decomposition of $S$ with ideal edges $\{ b_j \}_{j \in J}$. 
Complete $\Delta$ to an ideal triangulation $\Delta^\prime$ of $S$ such that all polygons in $S$ complementary to $\Delta$ have a standard triangulation. 
Denote the edges of $\Delta^\prime \setminus \Delta$ by $\{ a_k \}_{k \in K}$. 
Fix an assignment $\trCoords \in \R_{>0}^T$ of positive real numbers to the ideal triangles of $\Delta^\prime$. 
Denote by $\eCoordsb$ and $\eCoords$ assignments of edge parameters to $\{ b_j \}_{j \in J}$ and $\{ a_k \}_{k \in K}$ respectively so that $(\trCoords, \eCoords, \eCoordsb )$ is understood to be an element of $\A=\A_{\Delta^\prime}$. 
Then the set 
\begin{equation} \label{eq:algebraic_bd_cell_slice}
X_{\trCoords} = \left\{ (\trCoords, \eCoords, \eCoordsb )  \in \A \mid \Out_\alpha(b_j) > 0 \hspace{2mm} \forall \hspace{2mm} j \in J \hspace{2mm} \text{and} \hspace{2mm} \Out_\alpha(a_k) = 0 \hspace{2mm} \forall \hspace{2mm} k \in K \right\}
\end{equation}
is nonempty.
\end{lem}

\begin{proof}
We will explicitly construct assignments $\eCoordsb$ and $\eCoords$ of positive real numbers to the oriented edges of $S$ so that $(\trCoords, \eCoords, \eCoordsb ) \in \mathring{\mathcal{C}}(\Delta)$. Set $b^+_j = b^-_j = 1$ for all $j \in J$. Fix a polygon $P$ in the complement of the edges of $\Delta$ with $n$ ideal vertices. 

First we will show that there is a choice of edge parameters on the edges inside $P$ which ensures that $\Out(a_k) = 0 $ for all edges $a_k$ strictly contained in $P$. The definition of outitude ensures that the parameters governing the outitude of the edge $a_k$ in $P$ are those assigned to the ideal triangles and oriented edges inside $P$ and on the oriented edges on the boundary of $P$. Therefore we may consider each polygon independently and this part of the proof will suffice to show that $\Out (a_k) = 0$ for all $k \in K$.

We will complete the lemma by showing that the construction in the previous paragraph may be performed in such a manner that $a^\pm_k \geq 1$ for all $k \in K$. Given that $ b^{\pm}_j = 1$ for all $j \in J$ this ensures that $\Out ( b_j ) > 0 $ for all $j \in J$ regardless of the choice of $\trCoords$.

The ideal triangulation of $P$ is standard by assumption. Denote by $V_0$ the ideal vertex of $P$ which abuts every ideal edge strictly contained in $P$. Let $\{ a_m \} $ for $m \in \{ 0, 1, \dots, n-4 \}$ denote the subset of $\{a_k \}_{k \in K}$ comprising edges which are strictly contained in $P$. Order these edges  as depicted in the right-hand image in Figure \ref{fig:Standard_Triangulation2}. Similarly the triangle parameters in $P$ will be denoted $A_m$ for $m \in \{0, 1, \dots, n-3 \}$ with ordering as in Figure \ref{fig:Standard_Triangulation2}. Let $a^+_m$ denote the edge parameter 
oriented away from $V_0$. The condition that $\Out ( a_m ) = 0$ for all $m \in \{ 0, 1, \dots, n-4 \}$ is equivalent to the following equations.
\begin{align*}
a^-_0 &= \frac{ a^+_0 ( A_0 + A_1 ) }{ A_0 (  a^+_0 - a^+_1 ) + A_1 ( a^+_0 - 1 ) } \enspace , \\
a^-_m &= \frac{ a^+_m ( A_m + A_{ m + 1 } ) }{ A_m (  a^+_m - a^+_{ m + 1 } ) + A_{ m + 1 } ( a^+_m -  a^+_{ m - 1 } ) } \quad \text{for } 1 \leq m \leq n-3 \enspace ,  \\
a^-_{n-4} &= \frac{ a^+_{ n - 4 } ( A_{ n - 4 } + A_{ n - 3 } ) }{ A_{ n - 4 } (  a^+_{ n - 4 } - 1 ) + A_{ n - 3 } ( a^+_{ n - 4 } - a^+_{ n - 5 } ) } \enspace .
\end{align*}
It remains to ensure that, subject to these conditions, $a^-_m$ is well-defined and positive. This is equivalent to the following 
inequalities being satisfied.
\begin{align*}
& A_0 ( a^+_0 - a^+_1 ) + A_1 ( a^+_0 - 1 ) > 0  \enspace , \\
& A_m( a^+_m - a^+_{ m + 1 } ) + A_{m+1}( a^+_m - a^+_{ m - 1 } ) > 0 \text{ for } 1 \leq m \leq n-3 \enspace , \\ 
& A_{n-4}( a^+_{ n - 4 } - 1 ) + A_{n-3}( a^+_{ n - 4 } - a^+_{ n - 5 } ) > 0 \enspace .
\end{align*}
Therefore we define $\mathcal{M} := \max \{ A_0, A_1, \dots A_{n-3} \}$,  $\mathcal{N} := \min \{ A_0 , A_1 , \dots A_{ n - 3 } \}$ and $\epsilon := \frac{\mathcal{N}}{\mathcal{N} + \mathcal{M} }$. Now we set
\[
a^+_m := 1 +  \sum_{x=0}^m \epsilon^x \enspace,
\]
{so that $a^+_m > 1 $ for all $m \in \{0, \dots, n-4\}$ by definition. Now we verify that $a^-_m > 1$ for all $m \in \{0, \dots, n-4\}$. We begin with $m=0$. We would like to show that
\begin{align*}
a^-_0 = \frac{2 (A_0 + A_1) }{ - A_0 \epsilon + A_1 } > 1 \enspace .
\end{align*}
If $A_1 - A_0 \epsilon > 0$ then this is clearly true because the statement is equivalent to $2(A_0 + A_1) > A_1 - A_0 \epsilon $ which is immediate since $A_0 \epsilon $, $ A_0 $ and $A_1$ are strictly positive. Therefore it remains to show that
\begin{align*}
A_1 - A_0 \frac{\mathcal{N} }{ \mathcal{N}+ \mathcal{M} } > 0  \\
\iff A_1 (\mathcal{N} + \mathcal{M} ) > A_0 \mathcal{N} \enspace .
\end{align*}
However this is clearly true because the term $A_1 \mathcal{N}$ is strictly positive and we have $A_1 \geq \mathcal{N}$ and $\mathcal{M} \geq A_0$. Therefore we have $a^-_0 > 1$. Now we show that $a^-_m > 1$ for $1 \leq m \leq n-3$. This simplifies to the following,
\begin{align*}
\frac{ ( 1 + \sum_{x=0}^m \epsilon^x) ( A_m + A_{ m + 1 } ) }{ -  A_m  \epsilon^{m+1} + A_{ m + 1 } \epsilon^m  } > 1 \enspace .
\end{align*}
Once again we note that if the denominator is positive then the inequality follows easily as it is reduced to the following,
\begin{align*}
 ( 1 + \sum_{x=0}^m \epsilon^x) ( A_m + A_{ m + 1 } ) & > - A_m \epsilon^{m+1} + A_{ m + 1 } \epsilon^m \\
 \iff  ( 1 + \sum_{x=0}^m \epsilon^x) ( A_m + A_{ m + 1 } ) + A_m \epsilon^{m+1} &> A_{ m + 1 } \epsilon^m \enspace .
\end{align*}
The latter is verified by noting that all terms in the expression are strictly positive and the coefficient of $A_{m+1}$ on the left is larger than that on the right. It remains to show that $ -  A_m  \epsilon^{m+1} + A_{ m + 1 } \epsilon^m  > 0$. As $\epsilon > 0$ this is reduced to $A_{m+1} > A_m \epsilon$, which we may restate as
\[
A_{m+1}(\mathcal{N} + \mathcal{M} ) > A_m \mathcal{N} \enspace .
\]
This statement is clear because all terms are strictly positive and $A_{m+1} \geq \mathcal{N}$ and $\mathcal{M} \geq A_m$. We have shown that $a^-_{m} > 1$ for $1 \leq m \leq n-3$. It remains to consider the case $m = n-4$. We must verify the inequality
\begin{align*}
\frac{ \left( 1 + \sum_{x=0}^{n-4} \epsilon^x \right) ( A_{ n - 4 } + A_{ n - 3 } ) }{ A_{ n - 4 } \sum_{x=0}^{n-4} \epsilon^x + A_{ n - 3 } \epsilon^{n-4} } & > 1 \\
\iff \left( 1 + \sum_{x=0}^{n-4} \epsilon^x \right) \left( A_{ n - 4 } + A_{ n - 3 } \right)  & >  A_{ n - 4 } \sum_{x=0}^{n-4} \epsilon^x + A_{ n - 3 } \epsilon^{n-4} \\
\iff A_{n-4} + \left( 1 + \sum_{x=0}^{n-3} \epsilon^x \right) A_{n-3} & > 0 \enspace .
\end{align*}
The final inequality is clear given that all terms are strictly positive. We have now verified that $a^\pm_m > 1$ for all $m \in \{0, \dots, n-4 \}$}. Performing this procedure for all polygons in $S$ 
we determine that $a^{\pm}_k > 1$ for all $k \in K$. The definition of $a^-_k$ 
is chosen so that under these conditions $\Out(a_k) = 0$ for all $k \in K$. 

Now we will verify that $\Out(b_j) > 0$ for all $j \in J$. Having imposed the condition that $b^{\pm}_j = 1$ for all $j \in J$ the definition of $\Out(b_j)$ simplyfies to
\[
\Out(b_j) = B_0 ( x + y - 1 ) + B_1 ( z + w - 1 ) \enspace ,
\]
where $B_0$ and $B_1$ are the parameters assigned to the triangles of $\Delta^\prime$ whose boundary contains $b_j$ and $x, y, z, w$ are edge parameters in the set $\{ b^{\pm}_j \}_{j \in J} \cup \{ a^{\pm}_k \}_{k \in K}$. We have seen that $a^{\pm}_k > 1$ for all $k \in K$ and $b^{\pm}_j = 1$ for all $j \in J$. Therefore $x, y, z, w \geq 1$ and $\Out ( b_j ) > 0$ for all $j \in J$. We have explicitly constructed an element of $X_{\trCoords}$ as required, completing the proof.
\end{proof}

{Theorem~\ref{thm:main} and Lemma~\ref{lem:surjective} are instructive in how we will proceed to show that $\Cell (\Delta)$ is a topological cell for all cell decompositions $\Delta$ of $S$. Let $X_{\trCoords}$ be as defined in \eqref{eq:algebraic_bd_cell_slice}. Using the notation established in the Lemma~\ref{lem:surjective}, we will prove that $X_{\trCoords}$ is homeomorphic to an open ball by retracting it onto a relatively open neighbourhood in $X_{\trCoords}$ about $\mathbbm{1} \in \partial X_{\trCoords} $. The main challenge will be ensuring that each fiber of the retraction stays in $X_{\trCoords}$. For this purpose we require the following technical lemma. }
\begin{lem}\label{lem:linear_bound}
	Let $B_0, B_1, a^+, b, c, d,  e \in \R_{>0}$. Define the function $a^-: (0, 1] \rightarrow \R$ by
	\[
	t \mapsto a^-(t) := \frac{ ( t a^+ +  1-t ) \big( B_0 ( t e + 1-t)  + B_1( t d + 1-t ) \big)  }{ t \big( B_0 \left( a^+ - b \right) + B_1 \left( a^+ - c \right) \big) } \enspace .
	\]
	
	Assume that $a^- (1) $ is positive and finite. Then the following inequality is true:
	\begin{equation}\label{eq:a0m}
	a^-(t) \geq t \frac{a^+ ( B_0 e + B_1 d ) }{ B_0 ( a^+  - b ) + B_1 ( a^+ - c ) } + 1-t \quad \forall \hspace{2mm} t \in (0, 1] \enspace .
	\end{equation}
\end{lem}
\begin{proof}
	First we show that $a^- ( t ) \in \R_{>0}$. Clearly the numerator in the definition of $a^- ( t ) $ is positive and well-defined for all $t \in (0, 1]$. It remains to show that the denominator $t \left( B_0(a^+ - b) + B_1(a^+ - c ) \right)$ is strictly postive for all $t \in  (0,1]$. However the denominator is a linear function in $t$ which is positive at $t=1$ and zero at $t=0$ therefore it is positive for all $t \in (0, 1]$.
	
	Now we prove the main claim of the Lemma.  We rewrite $a^-(t)$ as 
	\begin{align*}
	a^-(t) = & \quad t \cdot \frac{a^+ ( B_0 e + B_1 d ) }{ B_0( a^+ - b ) + B_1 ( a^+ - c ) } + (1-t) \cdot \left( \frac{ B_0 a^+ + B_1 a^+ }{ B_0 ( a^+ - b ) + B_1 ( a^+ - c ) } \right)  \\
	& + \frac{1-t}{t} \cdot \left( \frac{ B_0( t e + 1-t ) + B_1 ( t d + 1-t ) }{ B_0 (a^+ - b ) + B_1 ( a^+ - c )  } \right) \enspace .
	\end{align*}
	{We have already seen that the denominator of each term on the right hand side is positive. Therefore it is clear that
	\begin{align*}
	\frac{ a^+ ( B_0 + B_1 ) }{ B_0( a^+ - b ) + B_1 ( a^+ - c ) } > 1 \enspace .
	\end{align*}
	Moreover, the facts that $t \in (0, 1]$ and that the terms $B_0, e$ and $d$ are strictly positive ensure that 
	\begin{align*}
	\frac{ (1-t) \left( B_0 ( t e + 1-t ) + B_1 ( t d + 1-t ) \right) }{ t ( B_0 ( a^+ - b ) + B_1 ( a^+ - c ) ) } > 0 \enspace .
	\end{align*}
	Combining these facts we obtain
	\begin{align*}
	a^-(t) > & \quad t \cdot \frac{a^+ ( B_0 e + B_1 d ) }{ B_0( a^+ - b ) + B_1 ( a^+ - c ) } + 1-t  + \frac{1-t}{t} \cdot \left( \frac{ B_0( t e + 1-t ) + B_1 ( t d + 1-t ) }{ B_0 (a^+ - b ) + B_1 ( a^+ - c )  } \right) \quad \\
	> & \quad  t \cdot \frac{a^+ ( B_0 e + B_1 d ) }{ B_0( a^+ - b ) + B_1 ( a^+ - c ) } + 1-t \enspace ,
	\end{align*}
	as required.}
\end{proof}

{
Now, in the manner of Theorem~\ref{thm:main}, before showing that $\mathring{\mathcal{C}}(\Delta)$ is a toplogical cell we will show that it is a fibration over a topological cell such that all fibers are sets of the form $X_{\trCoords}$ defined analogously to the definition in the statement of Lemma~\ref{lem:main1}. We will see that each fiber $X_{\trCoords}$ is a cell of given dimension, regardless of the choice of $\trCoords$. Having showed that $\mathring{\mathcal{C}}(\Delta)$ is a fibration over a contractible set whose fibers are homeomorphic we will complete the result in Theorem~\ref{thm:main2}.}

\begin{lem}\label{lem:main2}
Let $\Delta$ be an ideal cell decomposition of $S$ with ideal edges $\{ b_j \}_{j \in J}$. 
Complete $\Delta$ to an ideal triangulation $\Delta'$ of $S$ such that all polygons in $S$ have a standard triangulation. Denote the edges of $\Delta^\prime \setminus \Delta$ by $\{ a_k \}_{k \in K}$. 
Fix an assignment $\trCoords \in \R^T_{>0}$ of positive real numbers to each ideal triangle of $\Delta^\prime$. 
If $\eCoordsb$ and $\eCoords$ denote assignments of edge parameters to $\{ b_j \}_{j \in J}$ and $\{ a_k \}_{k \in K} $, respectively, then we will understand $(A, \eCoords, \eCoordsb )$ to be an element of $\A$. The set
\[
X_{\trCoords} := \left\{ \left(  \trCoords, \eCoords, \eCoordsb \right) \in \mathcal{A} \mid \Out(b_j) > 0 \hspace{2mm} \forall \hspace{2mm} j \in J \hspace{2mm} \text{and} \hspace{2mm} \Out(a_k) = 0 \hspace{2mm} \forall \hspace{2mm} k \in K \right\}
\]
is homeomorphic to {$\R^{ \mid \underline{E} \mid  -  \mid K \mid  }_{>0}$}.
\end{lem}
\begin{proof}
We proceed in a similar manner to Theorem \ref{thm:main}. 
As in Lemma~\ref{lem:surjective} (cf. Figure~\ref{fig:Standard_Triangulation2}). we denote by $ b^{\pm}_j$ for $j \in J$ the two edge parameters assigned to the ideal edge $b_j$. We denote by $a^{\pm}_k$ for $k \in K$ the two edge parameters assigned to the ideal edge $a_k$. 

Fix a $n$-gon $P$ in the complement of $\Delta$ and let $V_0$ denote the ideal vertex of $P$ which abuts every ideal edge of $\Delta'$ strictly contained in $P$. 
We denote the set of ideal edges strictly contained in $P$ by $\{ a_m \}$ where $m = 0, 1, \dots, n-4$. 
We denote by $a^+_m$ the edge parameter on $a_m$ 
assigned to the edge oriented away from $V_0$ and $a^-_m$ the edge parameter on $a_m$  
assigned to the edge oriented towards $V_0$. 
Similarly we denote by $b^+_m$ the parameter assigned to the oriented edge from $V_m$ to $V_{m+1}$ and $b^-_m$ the parameter assigned to the oriented edge from $V_{m+1}$ to $V_m$ where indices are read modulo $n$. 
Finally we denote by $A_m$ the triangle parameter assigned to the ideal triangle of $\Delta'$ strictly contained in $P$ whose ideal vertices are $V_0$, $V_m$ and $V_{m+1}$ for $m = 0, 1, \dots, n-3$. 
See Figure \ref{fig:Standard_Triangulation2} for a depiction of the case $n=8$.

\begin{figure}
\centering
\begin{tikzpicture}
\begin{scriptsize}
	\draw (2, 0) edge[red1, thick] (6, 2);
	\draw (2, 0) edge[red1, thick] (6, 4);
	\draw (2, 0) edge[red1, thick] (4, 6);
	\draw (2, 0) edge[red1, thick] (2, 6);
	\draw (2, 0) edge[red1, thick] (0, 4);
	
	\node [circle, fill = red1, inner sep = 0pt, minimum size = 4pt, label = above right : {$A_0$} ] (A0) at (3.6, 0.3) {};
	\node [circle, fill = red1, inner sep = 0pt, minimum size = 4pt, label = above right : {$A_1$} ] (A1) at (4.75, 2) {};
	\node [circle, fill = red1, inner sep = 0pt, minimum size = 4pt, label = above right : {$A_2$} ] (A2) at (4, 3.5) {};
	\node [circle, fill = red1, inner sep = 0pt, minimum size = 4pt, label = above : {$A_3$} ] (A3) at (2.75, 4.5) {};
	\node [circle, fill = red1, inner sep = 0pt, minimum size = 4pt, label = above left: {$A_4$} ] (A4) at (1.25, 3.5) {};
	\node [circle, fill = red1, inner sep = 0pt, minimum size = 4pt, label = above left : {$A_5$} ] (A5) at (0.9, 1.6) {};
	
	\draw (2, 0) -- (4, 0)
	node [circle, fill = black,  inner sep = 0pt, minimum size = 4pt,  label = below: {$V_0$}, pos = 0] () {}
	node [circle, fill = black,  inner sep = 0pt, minimum size = 4pt,  label = below right: {$V_1$}, pos = 1] () {}
	node [circle, fill = blue1,  inner sep = 0pt, minimum size = 4pt,  label = below: {$b^+_0$}, pos = 0.3] () {}
	node [circle, fill = blue1,  inner sep = 0pt, minimum size = 4pt,  label = below: {$b^-_0$}, pos = 0.7] () {};
		
	\draw (4, 0) -- (6, 2)
	node [circle, fill = black,  inner sep = 0pt, minimum size = 4pt,  label = right: {$V_2$}, pos = 1] () {}
	node [circle, fill = blue1,  inner sep = 0pt, minimum size = 4pt,  label = below right: {$b^+_1$}, pos = 0.3] () {}
	node [circle, fill = blue1,  inner sep = 0pt, minimum size = 4pt,  label = below right: {$b^-_1$}, pos = 0.7] () {};
	
	\draw (6, 2) -- (6, 4)
	node [circle, fill = black,  inner sep = 0pt, minimum size = 4pt,  label = right: {$V_3$}, pos = 1] () {}
	node [circle, fill = blue1,  inner sep = 0pt, minimum size = 4pt,  label = right: {$b^+_2$}, pos = 0.3] () {}
	node [circle, fill = blue1,  inner sep = 0pt, minimum size = 4pt,  label = right: {$b^-_2$}, pos = 0.7] () {};

	\draw (6, 4) -- (4, 6)
	node [circle, fill = black,  inner sep = 0pt, minimum size = 4pt,  label = above: {$V_4$}, pos = 1] () {}
	node [circle, fill = blue1,  inner sep = 0pt, minimum size = 4pt,  label = above right: {$b^+_3$}, pos = 0.3] () {}
	node [circle, fill = blue1,  inner sep = 0pt, minimum size = 4pt,  label = above right: {$b^-_3$}, pos = 0.7] () {};
		
	\draw (4, 6) -- (2, 6)
	node [circle, fill = black,  inner sep = 0pt, minimum size = 4pt,  label = above: {$V_5$}, pos = 1] () {}
	node [circle, fill = blue1,  inner sep = 0pt, minimum size = 4pt,  label = above: {$b^+_4$}, pos = 0.3] () {}
	node [circle, fill = blue1,  inner sep = 0pt, minimum size = 4pt,  label = above: {$b^-_4$}, pos = 0.7] () {};
	
	\draw (2, 6) -- (0, 4)
	node [circle, fill = black,  inner sep = 0pt, minimum size = 4pt,  label = above: {$V_6$}, pos = 1] () {}
	node [circle, fill = blue1,  inner sep = 0pt, minimum size = 4pt,  label = above: {$b^+_5$}, pos = 0.3] () {}
	node [circle, fill = blue1,  inner sep = 0pt, minimum size = 4pt,  label = above: {$b^-_5$}, pos = 0.7] () {};
	
	\draw (0, 4) -- (0, 2)
	node [circle, fill = black,  inner sep = 0pt, minimum size = 4pt,  label = left: {$V_7$}, pos = 1] () {}
	node [circle, fill = blue1,  inner sep = 0pt, minimum size = 4pt,  label = left: {$b^+_6$}, pos = 0.3] () {}
	node [circle, fill = blue1,  inner sep = 0pt, minimum size = 4pt,  label = left: {$b^-_6$}, pos = 0.7] () {};
	
	\draw (0, 2) -- (2, 0)
	node [circle, fill = blue1,  inner sep = 0pt, minimum size = 4pt,  label = below left: {$b^+_7$}, pos = 0.3] () {}
	node [circle, fill = blue1,  inner sep = 0pt, minimum size = 4pt,  label = below left: {$b^-_7$}, pos = 0.7] () {};


	\draw[red1, thick] (10, 0) -- (14, 2)
	node [circle, fill = blue1,  inner sep = 0pt, minimum size = 4pt, label = { [black] right: $a^+_0$}, pos = 0.4] () {}
	node [circle, fill = blue1,  inner sep = 0pt, minimum size = 4pt,  label = { [black] above: $a^-_0$}, pos = 0.8] () {};
	
	\draw[red1, thick] (10, 0) -- (14, 4)
	node [circle, fill = blue1,  inner sep = 0pt, minimum size = 4pt, label = { [black] right: $a^+_1$}, pos = 0.4] () {}
	node [circle, fill = blue1,  inner sep = 0pt, minimum size = 4pt,  label = { [black] above: $a^-_1$}, pos = 0.8] () {};

	\draw[red1, thick] (10, 0) -- (12, 6)
	node [circle, fill = blue1,  inner sep = 0pt, minimum size = 4pt, label = { [black] right: $a^+_2$}, pos = 0.4] () {}
	node [circle, fill = blue1,  inner sep = 0pt, minimum size = 4pt,  label = { [black] right: $a^-_2$}, pos = 0.8] () {};
	
	\draw[red1, thick] (10, 0) -- (10, 6)
	node [circle, fill = blue1,  inner sep = 0pt, minimum size = 4pt, label = { [black] left: $a^+_3$}, pos = 0.4] () {}
	node [circle, fill = blue1,  inner sep = 0pt, minimum size = 4pt,  label = { [black] left: $a^-_3$}, pos = 0.8] () {};

	\draw[red1, thick] (10, 0) -- (8, 4)
	node [circle, fill = blue1,  inner sep = 0pt, minimum size = 4pt, label = { [black] left: $a^+_4$}, pos = 0.4] () {}
	node [circle, fill = blue1,  inner sep = 0pt, minimum size = 4pt,  label = { [black] right: $a^-_4$}, pos = 0.8] () {};
	
	\draw (10, 0) -- (12, 0)
	node [circle, fill = black,  inner sep = 0pt, minimum size = 4pt,  label = below: {$V_0$}, pos = 0] () {};
	
	\draw (12, 0) -- (14, 2)
	node [circle, fill = black,  inner sep = 0pt, minimum size = 4pt,  label = below: {$V_1$}, pos = 0] () {};
	
	\draw (14, 2) -- (14, 4)
	node [circle, fill = black,  inner sep = 0pt, minimum size = 4pt,  label = right: {$V_2$}, pos = 0] () {};
	
	\draw (14, 4) -- (12, 6)
	node [circle, fill = black,  inner sep = 0pt, minimum size = 4pt,  label = right: {$V_3$}, pos = 0] () {};
	
	\draw (12, 6) -- (10, 6)
	node [circle, fill = black,  inner sep = 0pt, minimum size = 4pt,  label = above: {$V_4$}, pos = 0] () {};
	
	\draw (10, 6) -- (8, 4)
	node [circle, fill = black,  inner sep = 0pt, minimum size = 4pt,  label = above: {$V_5$}, pos = 0] () {};
	
	\draw (8, 4) -- (8, 2)
	node [circle, fill = black,  inner sep = 0pt, minimum size = 4pt,  label = left: {$V_6$}, pos = 0] () {};
	
	\draw (8, 2) -- (10, 0)
	node [circle, fill = black,  inner sep = 0pt, minimum size = 4pt,  label = left: {$V_7$}, pos = 0] () {};
\end{scriptsize}	
\end{tikzpicture}
		\caption{{Labelling of triangle and oriented edge parameters of a standard triangulation in a polygon containing $6$ ideal triangles. Some of the ideal vertices $V_i$ and edges $b_j$ may be identified in $S$.} }
\label{fig:Standard_Triangulation2}
\end{figure}
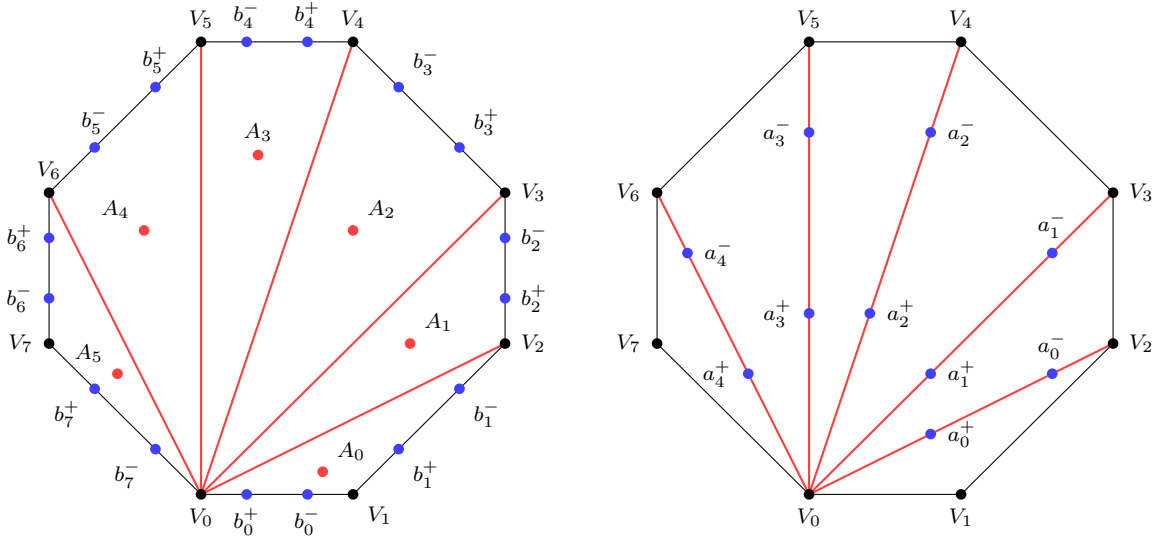

In order to show that $X_{\trCoords}$ is homeomorphic to $\R^{ \mid \underline{E} \mid - \mid K \mid }$ we will explicitly construct a homeomorphism from $X_{\trCoords}$ to a convex open subset of the ball $B_{\mathbbm{1} }(0.1) \subset \R^{ \mid \underline{E} \mid - \mid K \mid }$ of Euclidean radius $0.1$ and center $\mathbbm{1} := (1, 1, \dots, 1)$. 
An essential ingredient in this homeomorphism will be the linear deformation of some coordinates  to $1$ by the function
\begin{align*}
\phi : \R_{>0} \times (0, 1] & \rightarrow \R_{>0} \enspace , \\
 ( x , t)  & \mapsto t x + 1-t \enspace .
\end{align*}
We note here that the domain of $\phi$ is $\R_{>0} \times (0, 1]$ as opposed to $\R_{>0} \times [0,1]$ because, as we will see, the point $\mathbbm{1}$ is an element of $ \partial X_{\trCoords}$ rather than $X_{\trCoords}$ itself. Therefore there exists an open subset $U \subset B_{\mathbbm{1} }(0.1)$ onto which we can perform an isotopy of $X_{\trCoords}$.

{We hope to glue the maps $\phi$ together to construct an isotopy $\Phi: X_{\trCoords} \times (0, 1] \rightarrow U \subset B_{\mathbbm1}(0.1)$.} The difficulty here {which was not present in the proof of Theorem~\ref{thm:main}} is that we cannot apply this linear deformation to all edge parameters simultaneously if we want to guarantee 
that $\Phi( x, t ) \in X_{\trCoords}$ for all $t \in (0, 1]$.
It is true however that $\Out(a_k) = 0$ for all $\overline{a} \in X_{\trCoords}$. We may therefore impose these equalities to express the interior edge parameters $\eCoords^-$ which are 
assigned to edges oriented towards $V_0$ in terms of the parameters $\eCoords^+$ which are 
assigned to edges oriented away from $V_0$ and the edge parameters $\overline{b}^\pm_j$ of $\Delta$.
That is, while deforming all the edge parameters $\eCoordsb$ and $\eCoords^+$ linearly via $\phi$ we will deform the $n-4$ values of $\eCoords^-$ inside the $n$-gon $P$ via functions $\phi_m^P : \mathcal{A} \times (0, 1] \rightarrow \R_{>0}$ given by
\begin{align*}
a_0^-(t) = \phi^P_0 ( \trCoords , \eCoords^+, \eCoordsb , t ) & =  \frac{ \phi ( a^+_0,  t )  \left( A_0 \phi ( b^+_2,  t ) + A_1 \phi ( b^-_1,  t ) \right ) }{ A_0 \left( \phi ( a^+_0 ,  t ) - \phi ( a^+_1 ,  t ) \right) + A_1 \left( \phi ( a^+_0 ,  t ) - \phi ( b^+_0 , t ) \right) } \enspace , \\
a_m^-(t) = \phi^P_m ( \trCoords, \eCoords^+, \eCoordsb , t ) & = \frac{ \phi( a^+_m , t ) \left( A_m \phi( b^+_{m+2} , t ) + A_{m+1} \phi( b^-_{m+1} , t ) \right) }{ A_m \left( \phi ( a^+_m , t ) -  \phi( a_{m+1}^+ ,  t ) \right) + A_{m+1} \left( \phi( a^+_m , t ) -  \phi( a_{m-1}^+ , t ) \right) } \enspace , \\
a_{n-4}^- = \phi^P_{n-4} ( \trCoords, \eCoords^+ , \eCoordsb , t ) & = \frac{ \phi( a^+_{n-4},  t ) \left( A_{ n - 4 } \phi ( b^+_{ n - 2 } ,  t ) + A_{ n - 3 } \phi ( b^-_{ n - 3 } ,  t ) \right)  }{ A_{ n - 4 } \left( \phi ( a^+_{ n - 4 } ,  t ) - \phi( b^-_{ n - 1 } ,  t )  \right) + A_{ n - 3 } \left( \phi( a^+_{ n - 4 } , t ) - \phi( a^+_{ n - 5 } , t ) \right) } \enspace ,
\end{align*} 
where $m = 1, \dots, n-5$ in the case where $P$ is an $n$-gon with $n>4$. 
If $n=4$ we deform $a_0^-$ by the function
\begin{align*}
a_0^- = \phi^P_0 \left( \trCoords, \eCoords^+, \eCoordsb ,  t \right)  & = \frac{ \phi ( a^+_0 ,  t ) \left( A_0 \phi( b^+_2 , t ) + A_1 \phi( b^-_1 , t ) \right) }{ A_0 \left( \phi( a^+_0 , t ) - \phi (  b^-_3 , t )  \right) + A_1 \left( \phi ( a^+_0 , t ) - \phi ( b^+_0 , t ) \right) } \enspace .
\end{align*}
{The functions $\phi^P_m$ are defined to ensure that $\Out( a_m ) = 0$ for all $t \in (0, 1]$.} Let us consider the projection
\begin{align*}
{\pi_{\trCoords} } : X_{\trCoords} \rightarrow {\R^{ \mid \underline{E} \mid - \mid K \mid }_{>0} } \\
 ( \trCoords, \eCoords, \eCoordsb ) \mapsto ( \eCoords^+, \eCoordsb ) \enspace .
\end{align*}
This map is invertible precisely because we consider only the case in which the triangle parameters are given by $\trCoords$ and we can recover the edge parameters $\eCoords^-$ inside a polygon $P$ by solving the equation $\Out(a_m) = 0$ for $a^-_m$. { Since the function by which we recover $\overline{a}^-$ is continuous, it follows immediately that $\pi_{\trCoords}^{-1}$ is continuous. We note that $\pi_{\trCoords}$ is continuous because it is a projection.} Thus $\pi_{\trCoords}$ is a homeomorphism onto its image and it suffices to show that $\pi_{\trCoords} (X_{\trCoords}) $ is homeomorphic to $\R^{ \mid \underline{E} \mid -  \mid K \mid }_{>0}$.
We will do this in the following steps:
\begin{enumerate}
	\item First we show that if $( \eCoords^+, \eCoordsb ) \in \pi_{\trCoords} ( X_{\trCoords} ) $ then each point on the interval from $ ( \eCoords^+, \eCoordsb )$ to $ \mathbbm{1} \in { \R^{  \mid \underline{E} \mid  -  \mid K \mid }_{>0} }$ other than $\mathbbm{1}$ itself is also contained in $\pi_{\trCoords}(X_{\trCoords})$.
	\item {We} use the latter property to {show} that $\pi_{\trCoords}(X_{\trCoords})$ is homeomorphic to an open convex set, and as such {is} homeomorphic to {$\R^{ \mid \underline{E} \mid  - \mid K \mid }_{>0} $}.
\end{enumerate}
\emph{Proof of claim} $(1)$: Let $\mathbbm{1} \in {\R^{ \mid \underline{E} \mid  - \mid K \mid }_{>0} }$ denote the vector $(1, 1, \dots, 1)$. 
Fix a point $(\eCoords^+, \eCoordsb) \in \pi_{\trCoords}(X_{\trCoords}) $ where we denote 
$\eCoordsb = \left( b^+_0, b^-_0, b^+_1, \dots, b^+_{|J| - 1}, b^-_{|J| - 1} \right) \in \R^{ 2|J| }_{>0}$ and $\eCoords^+ = \left( a^+_0, a^+_1, \dots, a^+_{ | K | - 1 } \right) \in \R^{ |K| }_{>0}$. 
{ The values taken } along the linear deformation {in $\R^{ \mid \underline{E} \mid - \mid K \mid}_{>0}$ from $(\overline{a}^+, \overline{b} )$} towards $\mathbbm{1}$ {are}
\begin{align*}
\eCoordsb(t) & := \left( \phi ( b^+_0, t), \phi( b^-_0, t) , \phi( b^+_1, t), \dots, \phi( b^+_{|J| - 1}, t ) , \phi( b^-_{ J| - 1}, t) \right) \in \R^{ 2|J| }_{>0} \enspace , \\
\eCoords^+ (t) &:= \left(  \phi ( a^+_0, t) ,\phi( a^+_1, t) , \dots, \phi( a^+_{ | K | - 1 }, t) \right) \in \R^{ | K | }_{>0} \enspace .
\end{align*}
We must show that $( \eCoords^+ ( t ) , \eCoordsb ( t ) ) \in \pi_{\trCoords} ( X_{\trCoords} ) $ for all $ t \in (0, 1]$. 
In order for $( \eCoords^+ ( t ) , \eCoordsb ( t ) ) \in \pi_{\trCoords} ( X_{\trCoords} ) $ we must be able to recover an assignment of edge parameters 
{ $a^-$ such that $(\trCoords, \overline{a}, \overline{b} ) \in X_{\trCoords} $ }.  
If such an assignment exists it is uniquely determined by the condition that $\Out( a_m ) = 0$. 
This is because the standard triangulation is chosen in such a manner that, { having fixed $\trCoords$, $\eCoords^+ ( t ) $ and $\eCoordsb ( t ) $, the function} $\Out ( a_m ) $ is 
a linear function of $a^-_m$.  
We need only to ensure that the value for $a^-_m$ thus obtained is finite and positive.

Fix $t_0 \in (0, 1]$. By construction the value of $a^-_m$ determined by $ ( \eCoords^+(t_0), \eCoordsb (t_0) ) $ is 
 $a_m^-=\phi^P_m ( \trCoords, \eCoords^+, \eCoordsb, t_0 ) $. 
First suppose $n > 4$. By 
 definition of $\phi^P_m$ the claim that
  $a_m^-$  is positive for all $m$  is equivalent to  satisfying each of the inequalities
\begin{align*}
		A_0 \left( \phi ( a^+_0 ,  t_0 ) - \phi ( a^+_1 ,  t_0 ) \right) + A_1 \left( \phi ( a^+_0 ,  t_0 ) - \phi ( b^+_0 , t_0 ) \right) > 0 & \enspace , \\
		A_m \left( \phi ( a^+_m , t_0 ) -  \phi( a_{ m + 1 }^+ ,  t_0 ) \right) + A_{ m + 1 } \left( \phi( a^+_m , t_0 ) -  \phi( a_{ m - 1 }^+ , t_0 ) \right) > 0 & \enspace , \hspace{2mm} 1 \leq m \leq n-5 \enspace , \\
		A_{ n - 4 } \left( \phi ( a^+_{ n - 4 },  t_0 ) - \phi( b^-_{ n - 1 } ,  t_0 )  \right) + A_{ n - 3 } \left( \phi( a^+_{ n - 4 } , t_0 ) - \phi( a^+_{ n - 5 } , t_0 ) \right) > 0 & \enspace .
\end{align*}
Using the definition of $\phi ( x , t_0 ) $ these are equivalent to
\begin{align*}
		t_0 \left( A_0 \left( a^+_0 - a^+_1 \right) + A_1 \left( a^+_0  -  b^+_0 \right) \right) > 0 & \enspace ,\\
		t_0 \left( A_m \left( a^+_m -  a_{ m + 1 }^+ \right) + A_{ m + 1 } \left( a^+_m -   a^+_{ m - 1 } \right) \right) > 0 & \enspace , \hspace{2mm} 1 \leq m \leq n-5 \enspace , \\
		t_0 \left( A_{ n - 4 } \left( a^+_{ n - 4 } -  b^-_{ n - 1 } \right) + A_{ n - 3 } \left( a^+_{ n - 4 } - a^+_{ n - 5 } \right) \right) > 0 & \enspace .
\end{align*}
Modulo the positive factor $t_0$ the latter inequalities are necessary conditions for $(\eCoords^+, \eCoordsb ) \in \pi_{\trCoords} ( X_{\trCoords} )$ which is true by assumption. The case 
in which $n=4$ is completely analogous. 

We have shown that if $( \eCoords^+, \eCoordsb ) \in \pi_{\trCoords} ( X_{\trCoords} )$ then each point $( \eCoords^+ ( t ) , \eCoordsb(t) ) \in \R^{ \mid \underline{E} \mid  -  \mid K \mid }_{>0} $ on the line from $( \eCoords^+, \eCoordsb ) $ to $\mathbbm{1}$ admits an assignment of 
{ an edge parameter} $a^-_m (t) $ to each of the oriented edges 
{$a_m$} of $P$ for $m = 0, 1 \dots n-4$ which ensure that $\Out( a_m ) = 0$ for all $m = 0, 1, \dots n-4$ and for all $t \in (0, 1]$. 
We can make this choice independently for each polygon $P$ so that we have a well-defined candidate vector $ ( \trCoords ,  \eCoords (t) , \eCoordsb (t) ) \in \mathcal{A}$ such that $\pi_{\trCoords}( (\trCoords, \eCoords (t) , \eCoordsb (t) ) ) =  (\eCoords^+ (t) , \eCoordsb (t) ) $. 
It remains to be proved that $( \trCoords, \eCoords(t) , \eCoordsb(t) ) \in X_{\trCoords}$. 
In particular, we must show that $( \trCoords, \eCoords (t) , \eCoordsb(t) ) $ satisfies $\Out(b_j) > 0$ for all $j \in J$.

Fix $j \in J$ and let $B_0$ and $B_1$ be the triangle parameters in $\trCoords$ assigned to ideal triangles of $\Delta^\prime$ which are adjacent to the edge $b_j$. 
Fix $t_0 \in (0, 1]$. 
We seek to 
{verify} the inequality
\begin{align}
\Out ( b_j ) = & B_0 \left( \phi( b^+_j, t_0 ) \psi_0( x_0 , t_0 )+ \phi( b^-_j, t_0 ) \psi_1(x_1, t_0 )  - \phi( b^+_j, t_0 ) \phi ( b^-_j, t_0 ) \right) \label{eq:outitude_deformation1} \\
& + B_1 \left( \phi( b^+_j, t_0 ) \psi_2 ( x_2, t_0 ) + \phi(b^-_j, t_0 ) \psi_3(x_3, t_0 ) - \phi( b^+_j , t_0 ) \phi( b^-_j , t_0 ) \right) > 0 \enspace . \label{eq:outitude_deformation2}
\end{align}
where $x_i \in \{ a^{\pm}_k \}_{k \in K} \cup \{ b^{\pm}_j \}_{j \in J}$ and $\psi_i$ is defined as 
\begin{align*}
\psi_i (x_i, t_0 ) &= \phi^P_m ( x_i, t_0 ) \enspace \text{ if } x_i = a^-_m \text{ for some polygon } P  \text{ and} \enspace , \\
\psi_i (x_i, t_0 ) &= \phi ( x_i, t_0 ) \enspace \text{ otherwise. }
\end{align*}
Lemma \ref{lem:linear_bound} shows that $\phi^P_k (x, t_0 ) \geq \phi(x, t_0 )$ for all $k \in K$. { Now we may replace the terms $\psi_i$ with $\phi_i$ or $\phi_i^P$ and the inequality on lines~\eqref{eq:outitude_deformation1}-\eqref{eq:outitude_deformation2} is now of the form} 
\begin{align*}
\Out ( b_j ) = B_0 & \left( \phi( b^+_j, t_0 ) \phi ( x_0 , t_0 )+ \phi( b^-_j, t_0 ) \phi (x_1, t_0 )  - \phi( b^+_j, t_0 ) \phi ( b^-_j, t_0 ) \right) \\
& + B_1 \left( \phi( b^+_j, t_0 ) \phi (x_2, t_0 ) + \phi(b^-_j, t_0 ) \phi (x_3, t_0 ) - \phi( b^+_j , t_0 ) \phi( b^-_j , t_0 ) \right) > 0 \enspace .
\end{align*}
Ensuring that this inequality is true is the content of Lemma \ref{lem:main1} given the assumption that $( \trCoords, \eCoords, \eCoordsb ) \in X_{\trCoords} $. 
We have shown that if $( \eCoords^+, \eCoordsb ) \in \pi_{\trCoords} ( X_{\trCoords} )$ then $( \eCoords^+ ( t ) , \eCoordsb (t) ) \in \pi_{\trCoords} ( X_{\trCoords} )$ for all $ t \in (0, 1]$. 

\emph{Proof of claim} $(2)$: 
Let $r$ be a ray or maximal open interval in $\pi_{\trCoords} ( X_{\trCoords} )$ whose closure contains $\mathbbm{1}$. Let $B_{\mathbbm{1}}(0.1)$ denote the ball in {$\R^{ \mid \underline{E} \mid -  \mid K \mid }_{>0}$} of Euclidean radius $0.1$ and center $\mathbbm{1}$. We have 
seen {in part (1) } that $\pi_{\trCoords} ( X_{\trCoords} )$ is stratified by 
{rays of this form and that } 
$r$ contains a unique open interval $\iota_r$ with one endpoint at $\mathbbm{1}$ and the other endpoint in $\partial B_{\mathbbm{1}}(0.1)$. 
One may construct  an isotopy 
\begin{align*}
\xi_r : r \cup \{ \mathbbm{1} \} \times [0, 1] \rightarrow \overline{r} \cup \{ \mathbbm{1} \}
\end{align*}
such that $\xi_r (r, 1) = \iota_r$ and $\xi_r(\mathbbm{1}, t) = \mathbbm{1}$ for all $t \in [0, 1]$.
 The set of all rays of maximal open intervals in $\pi_{\trCoords}(X_{\trCoords})$ whose closure contains $\mathbbm{1}$ is pairwise disjoint and 
$\pi_{\trCoords}(X_{\trCoords})$ is, by definition, a semi-algebraic set. Therefore these isotopies may be glued together to  
produce an isotopy
\begin{align*}
\xi : \{ \pi_{\trCoords}(X_{\trCoords} ) \cup \{ \mathbbm{1} \} \} \times [ 0, 1] \rightarrow \pi_{\trCoords}(X_{\trCoords} ) \cup \{ \mathbbm{1} \} \enspace ,
\end{align*} 
such that $\xi \mid_{r \cup \{ \infty \} } = \xi_r$. This shows that $\pi_{\trCoords}(X_{\trCoords} )$ is homeomorphic to $\xi( \pi_{\trCoords}(X_{\trCoords}) , 1 )=:U_{\Delta}$.

We 
conclude this lemma by showing that $U_{\Delta}$ is an open convex subset of $ B_{ \mathbbm{1} }(0.1)$.
Let $ ( \eCoords^+, \eCoordsb ) \in B_{\mathbbm{1}}(0.1)$. We wish to devise a set of { necessary and sufficient} conditions which ensure that $( \eCoords^+, \eCoordsb ) \in U_{\Delta}$. If $(\eCoords^+ , \eCoordsb ) \in \pi_{\trCoords} ( X_{\trCoords} )$ and $P$ is a polygon with four vertices we define $a^-_0$ as
\begin{align} \label{def:ai0}
a^-_0 := \frac{ a^+_0 \left( A_0 b^+_2 + A_1 b^-_1 \right) }{ A_0 \left( a^+_0 - b^-_3  \right) + A_1 \left( a^+_0 - b^+_0 \right) } \enspace .
\end{align}
For each polygon $P$ with more than four ideal vertices we define $a^-_m$ for $m = 0, 1, \dots, n-4$ as
\begin{align}
a^-_0 &:= \frac{ a^+_0  \left( A_0 b^+_2 + A_1 b^-_1 \right) }{ A_0 \left( a^+_0 - a^+_1 \right) + A_1 \left( a^+_0 - b^+_0 \right) } \enspace , \\
a^-_m &:= \frac{ a^+_m \left( A_m b^+_{ m + 2 }+ A_{ m + 1 } b^-_{ m + 1 } \right) }{ A_m \left( a^+_m -  a_{ m + 1 }^+  \right) + A_{ m + 1 } \left( a^+_m - a_{ m - 1 }^+ \right) } \hspace{2mm}, \enspace 1 \leq m \leq n-5 \enspace , \\ 
\label{def:ain} a^-_{ n - 4 } &:= \frac{a^+_{ n - 4 } \left( A_{ n - 4 } b^+_{ n - 2 } + A_{ n - 3 } b^-_{ n - 3 } \right)  }{ A_{ n - 4 } \left(  a^+_{ n - 4 } - b^-_{ n - 1 } \right) + A_{ n - 3 } \left( a^+_{ n - 4 } -  a^+_{ n - 5 } \right) } \enspace .
\end{align}
{Having fixed $\trCoords$, $\eCoords^+$ and $\eCoordsb$, these are the unique values which ensure that $\Out(a_m) = 0$ for all $0 \leq m \leq n-4$.} Using these definitions, further define $\eCoords := ( a^+_0, a^-_0, \dots, a^+_{ |K| - 1 }, a^-_{ |K| - 1 } ) \in \R^{ \mid \underline{E} \mid  - \mid K \mid }_{>0}$. 
To ensure that each $a^-_k$ 
is well-defined and positive we require that for a polygon $P$ with four ideal vertices 
the following inequality holds
\begin{align} \label{eq:ineq1}
 A_0 \left( a^+_0 - b^-_3  \right) + A_1 \left( a^+_0 - b^+_0 \right) > 0 \enspace.
\end{align}
If $P$ is a polygon with more than $n > 4$ vertices we require
\begin{align}
 A_0 \left( a^+_0 - a^+_1 \right) + A_1 \left( a^+_0 - b^+_0 \right) & > 0 \enspace , \\
 A_m \left( a^+_m - a_{ m + 1 }^+ \right) + A_{ m + 1 } \left( a^+_m - a_{ m - 1 }^+ \right) & > 0 \enspace , \enspace 1 \leq m \leq n-5 \enspace , \\
\label{eq:ineq2} A_{ n - 4 } \left( a^+_{ n - 4 } - b^-_{ n - 1 } \right) + A_{ n - 3 } \left( a^+_{ n - 4 } - a^+_{ n - 5 } \right) & > 0 \enspace .
\end{align}
Denote by $Y_{\trCoords} \subset {\R^{ \mid \underline{E} \mid - \mid K \mid }_{>0} }$ the subset satisfying conditions (\ref{eq:ineq1}) - (\ref{eq:ineq2}).
Recall that we have fixed $\trCoords$ so $Y_{\trCoords}$ is 
 the intersection of finitely many open half spaces and is therefore the interior of a convex polyhedron.
 By definition of $U_{\Delta}$ and $Y_{\trCoords}$ we have $U_{\Delta} \subseteq B_{\mathbbm{1}}(0.1) \cap Y_{\trCoords}$. We will show that $U_{\Delta} = B_{\mathbbm{1}}(0.1) \cap Y_{\trCoords}$.

Let $( \eCoords^+, \eCoordsb) \in B_{\mathbbm{1} } (0.1) \cap Y_{\trCoords}$. We have shown above that there is a unique candidate in $\mathcal{A}$ for the value of $\pi_{\trCoords}^{-1} (\eCoords^+, \eCoordsb)$ which is obtained by assigning to each $a^-_k$ the 
value determined by equations \eqref{def:ai0} - \eqref{def:ain}. Denote this candidate by $({\trCoords}, \eCoords, \eCoordsb )$. 
{ The values $a^-_k$ are chosen so that $(\trCoords, \eCoords, \eCoordsb ) $ } satisfies 
$\Out(a_k) = 0$ for all $k \in K$. 
To see that $(\eCoords^+, \eCoordsb) \in U_{\Delta}$ it remains only to show that $({\trCoords}, \eCoords, \eCoordsb ) $
also satisfies $\Out(b_j) > 0$ for all $j \in J$. First note that $0.9 < b^{\pm}_j, a^+_k < 1.1$ for all $j \in J$ and $k \in K$ { since $(\eCoords^+, \eCoordsb) \in B_{\mathbbm{1}}(0.1)$}. 
It follows from definitions (\ref{def:ai0}) - (\ref{def:ain}) that 
\[
a^-_k > \min \left\{ a^+_0, a^+_1, \dots, a^+_{|K|-1}, b^+_0, b^-_0, \dots, b^+_{|J|-1}, b^-_{|J|-1} \right\} > 0.9 \enspace , \quad \forall \; k \in K \enspace .
\]
 For example, in the case $n=4$ and $k=0$ we note that the denominator in the definition of $a^-_0$ is positive because $(\eCoords^+, \eCoordsb)$ satisfies Inequality~\eqref{eq:ineq1} by assumption. Therefore we have
\begin{align*}
a^-_0 = \frac{a_0^+ \left( A_0 b^+_2 + A_1 b^-_1 \right)}{ A_0(a^+_0 - b^-_3) + A_1 ( a^+_0 - b^+_0) } > \frac{A_0 b^+_2 + A_1 b^-_1}{ A_0 + A_1 } \geq \min\{ b^+_2, b^-_1 \} \enspace .
\end{align*}
The other cases are analogous. Therefore $(\trCoords, \eCoords, \eCoordsb )$ satisfies $\Out(b_j) > 0$ for all $j \in J$. 
We have shown that every point in $B_{\mathbbm{1}}(0.1) \cap Y_{\trCoords}$ is also in $\pi_{\trCoords} (X_{\trCoords})$. Therefore $B_{\mathbbm{1}}(0.1) \cap Y_A \subseteq \pi_{\trCoords} (X_{\trCoords}) \cap B_{\mathbbm{1}}(0.1) = U_{\Delta}$. 
In particular $U_{\Delta} = B_{\mathbbm{1}}(0.1) \cap Y_{\trCoords}$ as required.

We have not yet verified that $U_\Delta$ is nonempty. Recall that we have shown in Lemma \ref{lem:surjective} that $X_{\trCoords} \neq \emptyset$. 
Hence also $\pi_{\trCoords}(X_{\trCoords})$ is nonempty and for every point $( \eCoords^+, \eCoordsb ) \in \pi_{\trCoords}( X_{\trCoords} )$, the interval $ ( \eCoords^+ ( t ) , \eCoordsb ( t ) ) $ enters $B_{\mathbbm{1}}(0.1) \cap Y_{\trCoords}$ as $t \rightarrow 0$. 
Therefore $U_{\Delta}$ is a nonempty finite intersection of convex {open} subsets of {$\R^{ \mid \underline{E} \mid  - \mid  K \mid  }_{>0}$}. 
In particular $U_{\Delta}$ is itself a nonempty convex {open} subset of {$\R^{ \mid \underline{E} \mid - \mid K \mid }_{>0}$} and $U_{\Delta'} \cong {\R^{ \mid \underline{E} \mid  - \mid K \mid }_{>0} } $. In summary we have verified the following series of homeomorphisms,
\[
{\R^{ \mid \underline{E} \mid - \mid K \mid }_{>0} } \cong U_{\Delta} \cong \pi_{\trCoords} (X_{\trCoords} ) \cong X_{\trCoords} \enspace .
\]
This completes the required result.
\end{proof}

{We use Lemma~\ref{lem:main2} in the proof of Theorem~\ref{thm:main2} in a manner analogous to that in which Lemma~\ref{lem:main1} is used to prove Theorem~\ref{thm:main}. Having showed that $X_{\trCoords} \cong \R^{ \mid \underline{E} \mid - \mid K \mid }_{>0}$ for all $\trCoords \in \R^T_{>0}$ we will show that the projection
\begin{align}
\pi: \mathring{\mathcal{C}} ( \Delta) &\rightarrow \R^T_{>0}  \label{eq:projection_main1} \\
(\trCoords, \eCoords, \eCoordsb) & \mapsto \trCoords \enspace , \label{eq:projection_main2}
\end{align}
is a surjective globally trivial fibration. Having proved in Lemma~\ref{lem:main2} that the fibers are all homeomorphic to $\R^{ \mid \underline{E} \mid - \mid K \mid }_{>0}$ the result is immediate. }

\begin{thm}\label{thm:main2}
Let $\Delta$ be an ideal cell decomposition of $S$.
We have a homeomorphism $\Cell (\Delta ) \cong {\R^{ \mid T \mid  + \mid \underline{E} \mid - \mid K \mid }_{>0} }$.
\end{thm}
\begin{proof}
{As usual we denote by $\Delta^\prime$ a standard ideal triangulation of $S$ refining $\Delta$. We denote by $\{b_j \}_{j \in J}$ the edges of $\Delta$ and $\{a_k\}_{k \in K}$ the edges of $\Delta^\prime \setminus \Delta$.} Let ${\trCoords}$, $\eCoords$ and $\eCoordsb$ be assignments of positive real numbers to the ideal triangles of $\Delta^\prime$, and to the oriented edges of $\Delta^\prime \setminus \Delta$ and of $\Delta$ respectively so that we understand $({\trCoords}, a, b)$ 
 to be an element of $\A$. 
Let $\pi$ denote the projection defined by \eqref{eq:projection_main1}-\eqref{eq:projection_main2}. We have seen in Lemma \ref{lem:main2} each of the fibers of $\pi$ is homeomorphic to {$\R^{ \mid \underline{E} \mid -  \mid K \mid }_{>0}$ and in particular each is homeomorphic to all of the others}. Since $\partial \clCell (\Delta)$ is defined by a finite family of polynomial equations, it follows that $\pi$ is a locally trivial fibration. However we have seen in Lemma \ref{lem:surjective} that $\pi$ is surjective. Since $\R^T_{>0}$ is contractible it follows that $\pi$ is a globally trivial fibration. Therefore $\Cell (\Delta) \cong {\R^{\mid T \mid  + \mid \underline{E} \mid  - \mid K \mid }_{>0} } $ as required.
\end{proof}

\begin{exm}
	A doubly-decorated convex projective structure on a once-punctured torus $S_{1,1}$ is chosen by fixing a triangulation $\Delta \subset S_{1,1}$ together with $\mathcal{A}$-coordinates, in this case two triangle parameters and six edge parameters. 
	The latter are denoted by $(A,B,a^+,a^-,b^+,b^-,c^+,c^-) \in \R_{>0}^8$ as depicted in Figure~\ref{fig:ex_torus}. 
	We give three examples of points inside $\Cell (\Delta)$, that is structures for which $\Delta$ is the unique canonical triangulation, namely the structures whose $\mathcal{A}$-coordinates are
	\[
	(1,1,\tfrac{7}{4},\tfrac{1}{4},1,4,\tfrac{5}{4},\tfrac{5}{4}) \enspace , \qquad (1,2,3,4,1,2,3,4) \enspace , \qquad (\tfrac{1}{3},\tfrac{1}{2},1,\tfrac{4}{6},\tfrac{2}{5},2,\tfrac{1}{5},\tfrac{1}{10}) \enspace .
	\] 
	Each of these structures gets deformed by the smooth deformations used in the proof of Lemma~\ref{lem:main1} that deforms all edge parameters linearly to $1$, while the triangle parameters remain constant. 
	The representative of each projective structure is chosen by fixing three ideal vertices in $\R^3$.
	\begin{figure}[ht]
		\centering
		\begin{tikzpicture}
		\node [circle, fill = red1, inner sep = 0pt, minimum size = 4pt, label = above : {$A$} ] (A) at (1,2) {};
		\node [circle, fill = red1, inner sep = 0pt, minimum size = 4pt, label = below : {$B$} ] (B) at (2,1) {};	
		\draw (0,0) -- (0,3) -- (3,3) -- (3,0) -- cycle;
		\draw (0,0) -- (3,3);
		\node [circle, fill = blue1, inner sep = 0pt, minimum size = 4pt, label = below : {$a^+$}] at ( 1, 0 ) {};
		\node [circle, fill = blue1, inner sep = 0pt, minimum size = 4pt, label = below : {$a^-$}] at ( 2, 0 ) {};
		\node [circle, fill = blue1, inner sep = 0pt, minimum size = 4pt, label = left : {$b^+$}] at ( 0, 1 ) {}; 
		\node [circle, fill = blue1, inner sep = 0pt, minimum size = 4pt, label = left : {$b^-$}] at ( 0, 2 ) {}; 
		\node [circle, fill = blue1, inner sep = 0pt, minimum size = 4pt, label = above : {$a^+$}] at ( 1, 3 ) {};
		\node [circle, fill = blue1, inner sep = 0pt, minimum size = 4pt, label = above : {$a^-$}] at ( 2, 3 ) {};
		\node [circle, fill = blue1, inner sep = 0pt, minimum size = 4pt, label = right : {$b^+$}] at ( 3, 1 ) {}; 
		\node [circle, fill = blue1, inner sep = 0pt, minimum size = 4pt, label = right : {$b^-$}] at ( 3, 2 ) {}; 
		\node [circle, fill = blue1, inner sep = 0pt, minimum size = 4pt, label = above : {$c^+$}] at ( 2, 2 ) {};
		\node [circle, fill = blue1, inner sep = 0pt, minimum size = 4pt, label = below : {$c^-$}] at ( 1, 1 ) {};		
		\end{tikzpicture}
		\includegraphics[width=0.95\linewidth]{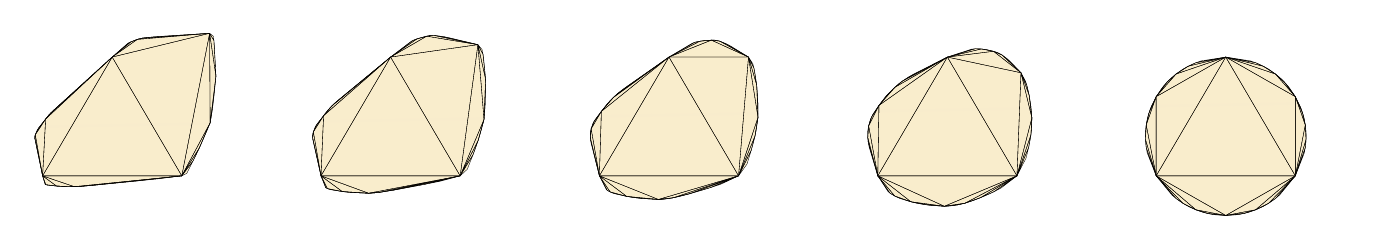}
		\includegraphics[width=0.95\linewidth]{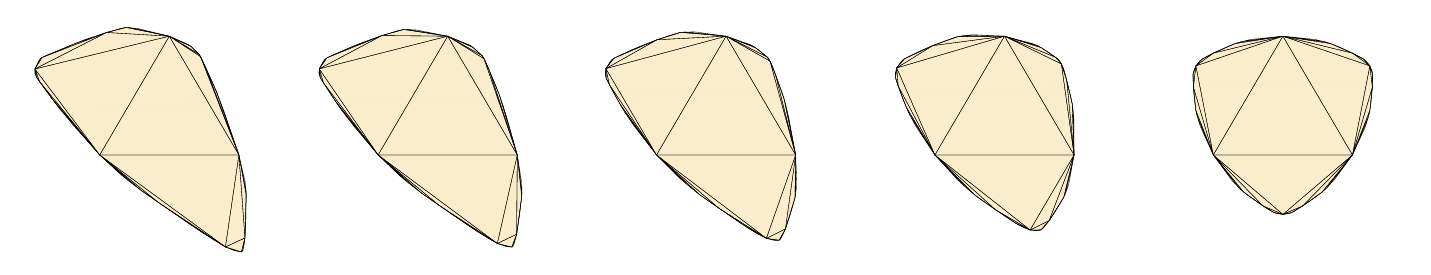}			
		\includegraphics[width=0.95\linewidth]{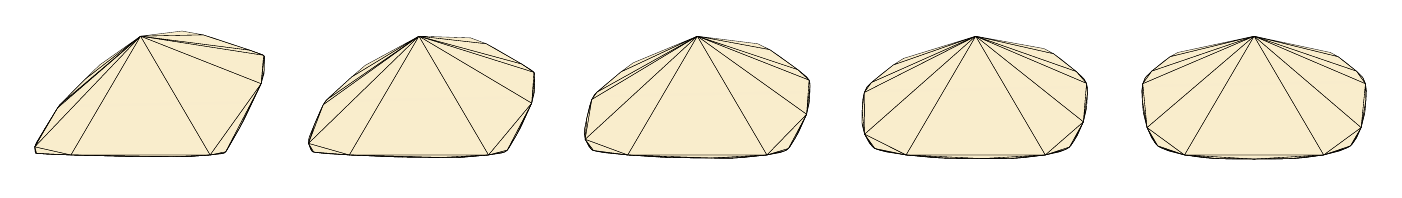}
		\caption{A triangulated torus with $\mathcal{A}$-coordinates $(A,B,a^+,a^-,b^+,b^-,c^+,c^-) \in \R_{>0}^8$ (top). The first row depicts a linear deformation from $(1,1,\tfrac{7}{4},\tfrac{1}{4},1,4,\tfrac{5}{4},\tfrac{5}{4})$ (left) to $\mathbbm{1} \in \R_{>0}^8$ (right). In the second row we linearly deform from $(1,2,3,4,1,2,3,4)$ to $(1,2,1,1,1,1,1,1)$. The last row shows the deformation from $(\tfrac{1}{3},\tfrac{1}{2},1,\tfrac{4}{6},\tfrac{2}{5},2,\tfrac{1}{5},\tfrac{1}{10})$ to $(\tfrac{1}{3},\tfrac{1}{2},1,1,1,1,1,1)$. For all convex projective structures that appear the chosen triangulation is uniquely canonical.}
		\label{fig:ex_torus}
	\end{figure}
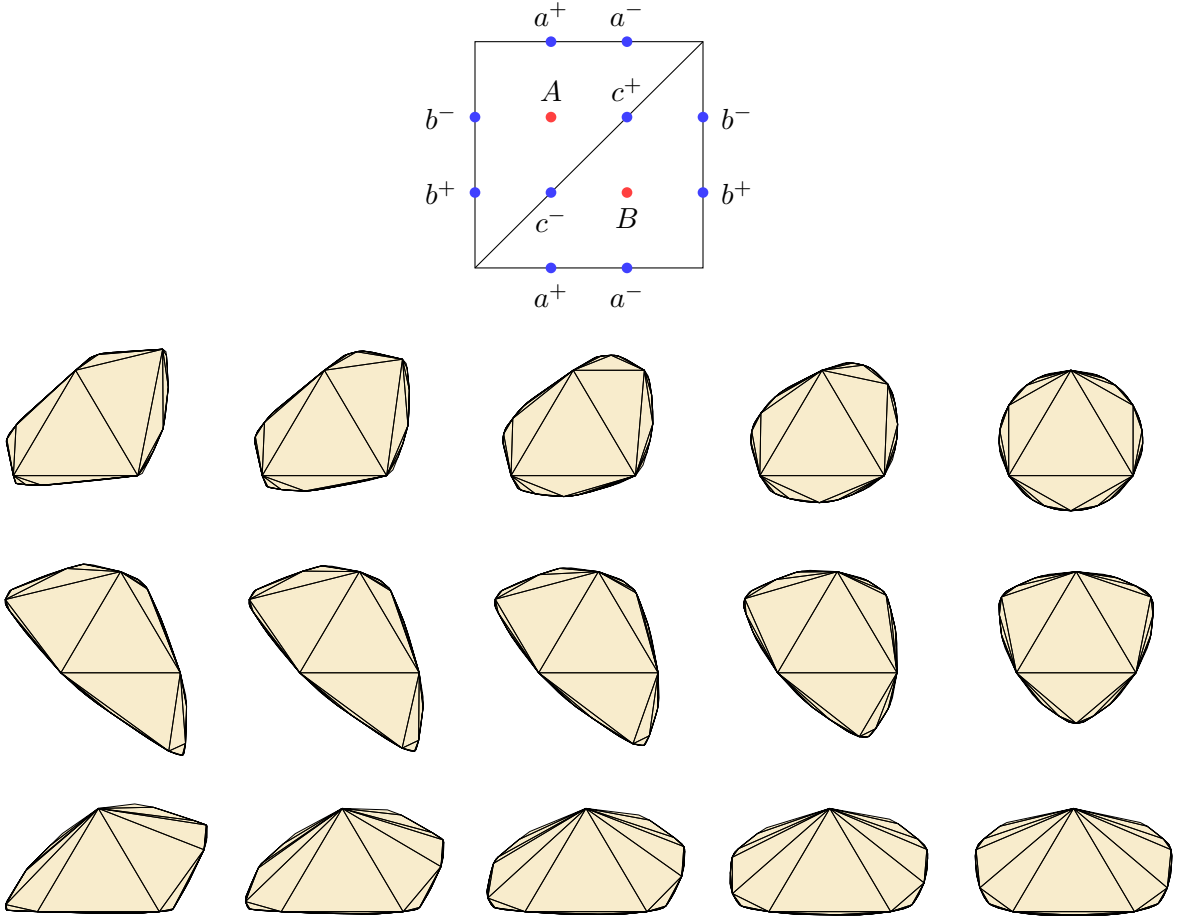
\end{exm}


\section{Duality}
\label{sec:duality}

This section describes a natural projective duality of doubly decorated structures and its effect on $\A$--coordinates (\S\ref{subsec:dual-proj-st}). The self-dual structures are shown to be hyperbolic structures with a preferred dual decoration, leading to an identification of Penner's Decorated Teichm\"uller Space in $\A$--coordinates (\S\ref{subsec:penner_cell_dec}). Examples are given to show that our cell decomposition of moduli space is generally not invariant under this duality, even if the surfaces is once-cusped (\S\ref{subsec:Outitude and duality}).


\subsection{Duality in $\A$--coordinates}
\label{subsec:dual-proj-st}

Every convex projective structure has a dual structure, constructed using the self-duality of the projective plane. This induces an automorphism 
$\sigma\co \Tt(S)\to  \Tt(S)$. We briefly sketch the details and refer to \cite{Vinberg-theory-1963, Cooper-marked-2010} for a complete treatment. We then construct a natural duality map $\sigma^\ddagger\co\decTtf(S)\to \decTtf(S)$ and describe it as an automorphism of $\A_\Delta.$

Let $\Omega$ be a properly convex domain of $\RP^2$. Its \emph{dual domain} is the set 
$$
\Omega^* = \{ P \in \RP^2 \ | \ P^\perp \cap \overline{\Omega} = \emptyset \}.
$$
Then $\Omega^*$ is also a properly convex domain and $\Omega = (\Omega^*)^*$. Supporting lines to $\Omega$ correspond to points on the frontier of $\Omega^*$, and vice-versa. There is a real analytic diffeomorphism $\Phi: \Omega \rightarrow \Omega^*$, called the \emph{dual map}, which has the following elementary interpretation. For $P \in \Omega$, the line $P^\perp$ is disjoint from $\overline{ \Omega^*}$, thus the affine patch $\A_P = \RP^2 \setminus P^\perp$ strictly contains $\Omega^*$. Then $\Phi(P)$ is the centre of mass of $\Omega^*$ in $\A_P$. It turns out that $\Phi$ is a projective transformation if and only if $\Omega$ is a conic.

For $A \in \PGL(\Omega)$ and $P \in \Omega$, $\Phi(A(P)) = (A^t)^{-1} (\Phi(P)).$ Hence the \emph{dual group} of a subgroup $\Gamma \leq \PGL(3,\R)$ is defined to be the group
\[
\Gamma^* = \{ (A^t)^{-1} \ | \ A \in \Gamma \}.
\]
Clearly $\Gamma \cong \Gamma^*$. Since duality inverts the eigenvalues of the holonomy, $\Omega^* / \Gamma^*$ also has finite volume.

So $\Phi$ induces a real analytic diffeomorphism
\[
	\overline{\Phi} \co\Omega / \Gamma \rightarrow \Omega^* / \Gamma^*,
\]
and $(\Omega^*,\Gamma^*,\phi^*)$ is an element of $\Tt(S)$, where $\phi^* := \overline{ \Phi} \circ \phi$. We therefore have a map 
$\sigma\co\Tt(S)\to \Tt(S)$ which takes a projective structure to its dual structure. By construction, $\sigma$ is invertible and has order two. 

We extend $\sigma$ to an involution $\sigma^\ddagger \co \decTtf(S)\to  \decTtf(S)$ by defining
\[
\sigma^\ddagger( \Omega , \Gamma , \phi , ( \covec_dec , \vec_dec ) ) = (\Omega^*,\Gamma^*,\phi^*, ( \vec_dec , \covec_dec ) ) \enspace ,
\]
that is $\sigma^\ddagger$ swaps the role of the vector and covector decorations.

This map has the following very simple description in terms of $\A$--coordinates, and we will see in \S\ref{subsec:penner_cell_dec} that its fixed points correspond to decorated hyperbolic structures with a particularly nice choice of covector decoration.

To understand the the duality map in $\A$--coordinates in suffices to study the change of triangle parameter and edge parameters of a concrete decorated triangle when swapping vectors and covectors.
Let $(R,C) = \left( ( \vtx_0 \mid  \vtx_1  \mid \vtx_2 ) , ( \covtx_0 \mid\mid \covtx_1 \mid \mid \covtx_2  )  \right) $ be a concrete decorated triangle and denote by $a_{ij}$ the edge parameter of the edge oriented from vertex $\vtx_i$ to vertex $\vtx_j$, where $i\neq j\in \{0,1,2\}$. 
The triangle parameter of $(R,C)$ is denoted by $A_{012}$.
Now the \emph{dual concrete decorated triangle} is given by the pair $(C^t,R^t) = \left( ( \covtx_0 \mid\mid \covtx_1 \mid \mid \covtx_2  )^t , ( \vtx_0 \mid  \vtx_1  \mid \vtx_2 )^t  \right)$.
If we denote by $a_{ij}^\perp$ the edge parameter of the edge in $(C^t,R^t)$ oriented from vertex $\covtx_i$ to $\covtx_j$, then we have
\begin{equation}\label{eq:dual_edges}
a_{ij}^\perp = \vtx_i^t \cdot \covtx_j^t = \covtx_j \cdot \vtx_i = a_{ij} \enspace .
\end{equation}
The triangle parameter $A_{012}^\perp$ of the dual concrete decorated triangle $(C^t,R^t)$ is, by definition, given as the determinant of $R^t$.
Now since  
\[
R \cdot C = 
\left(
\begin{array}{ccc}
0 & a_{01} & a_{02} \\
a_{10} & 0 & a_{12} \\
a_{20} & a_{21} & 0
\end{array}
\right)
\]
it follows that 
\begin{equation}\label{eq:dual_triangles}
A_{012}^\perp = \det (R^t) = \det (R) = \frac{a_{01}a_{12}a_{20} + a_{10}a_{21}a_{02}}{\det (C)} = \frac{a_{01}a_{12}a_{20} + a_{10}a_{21}a_{02}}{A_{012}} \enspace .
\end{equation}
Summarizing, the description of the involution $\sigma^\ddagger$ in $\A$--coordinates is given by \eqref{eq:dual_edges} and \eqref{eq:dual_triangles}.
We wish to emphasize that our convention for choosing a dual decoration currently has no \emph{intrinsic} justification, except for the case of hyperbolic structures.


\subsection{Penner's Decorated Teichm\"uller Space}\label{subsec:penner_cell_dec}

As before, let $S=S_{g,n}$ be a surface of genus $g \geq 0$ having $n>0$ punctures and Euler characteristic $\chi(S)=2-2g-n$ negative.
In \cite{Penner-Decorated-1987} Penner proves that Epstein-Penner's convex hull construction induces a cell decomposition of the decorated Teichm\"uller space $\decTeich(S)$. 
In the following we will embed Penner's decorated Teichm\"uller space in the space of $\mathcal{A}$-coordinates.

We briefly recall Penner's parametrization of $\decTeich(S)$. 
A \emph{marked hyperbolic structure} on $S$ is a marked properly convex projective structure $(\Omega , \Gamma , \phi) \in \Ttf(S)$ of finite volume whose convex domain $\Omega$ is a conic. 
By projective equivalence we may assume that this conic is defined Minkowski bilinear form
\[
\langle X,Y \rangle = X_0Y_0 + X_1Y_1 - X_2Y_2 \enspace ,
\] 
where $X=(X_0,X_1,X_2), ~Y=(Y_0,Y_1,Y_2) \in \R^3$.
In this case $\Omega$ equals the projectivization of the set 
\[
\HH^2 = \{ X \in \R^3 \mid \langle X,X \rangle = -1 ~,~ X_2>0 \} \enspace ,
\]
the hyperboloid model of the hyperbolic plane.
Using the affine chart at height $X_2=1$ we may identify $\Omega$ with the standard Klein disc, $\Omega = \{ (X_0,X_1,1) \in \R^3 \mid X_0^2+X_1^2<1\}$.
It follows that for a hyperbolic structure the holonomy group $\Gamma$ is a discrete subgroup of $\SO^+(2,1) < \SL_3(\R)$, and as such $\Gamma$ is isomorphic to a Fuchsian group in $\PSL_2(\R)$.

Denote by 
\[
\mathcal{L}^+ = \R_{>0} \cdot \partial \Omega = \{  V \in \R^3 \mid \langle V,V \rangle = 0 ~,~ V_2 > 0 \} 
\]
the associated positive light-cone.

Let $(\Omega, \Gamma, \phi) \in \decTeich(S)$.
As in \eqref{def:vec_dec} we may choose a vector decoration $\vec_dec \subset \mathcal{L}^+$, i.e. a $\Gamma$-invariant set of light-cone representatives for all parabolic fixed points.

The choice of a marked hyperbolic structure on $S$ together with a vector decoration forms the data of a point in the \emph{decorated Teichm\"uller space} $\decTeich(S)$.

Let $(\Omega , \Gamma , \phi , \vec_dec) \in \decTeich(S)$.
Fix some triangulation $\Delta \subset S$ consisting of E edges.
Let $\widetilde{\Delta}$ denote a lift of $\Delta$ to the hyperbolic plane $\HH$ such that edges of $\Delta$ are lifted to geodesic arcs. 
Let $\widetilde{e} \in \widetilde{\Delta}$ be a lift of an edge $e \in \Delta$, and denote the two ideal vertices at the ends of $\widetilde{e}$ by $p,q$.
Since $e$ connects two (not necessarily distinct) cusps, $p$ and $q$ are two parabolic fixed points with associated light-cone representatives $V_p, V_q \in \vec_dec$ .
Following Penner~\cite{Penner-book-2012}, we define the \emph{$\lambda$-length} of the edge $e$ by
\begin{equation}\label{eq:def_lambda}
\lambda_e = \sqrt{ - \langle V_p , V_q \rangle} \enspace .
\end{equation}
The constant $-1$ in the above formula is somewhat arbitrary; it affects certain formulae that are irrelevant to this paper. The original choice by Penner~\cite{Penner-Decorated-1987} was
\(
\lambda_e = \sqrt{ - \tfrac{1}{2} \langle V_p , V_q \rangle}.
\)

Penner proves that this defines a homeomorphism $\lambda: \decTeich(S) \mapsto \R_{>0}^E$, yielding  \emph{Penner's $\lambda$-coordinates} on the decorated Teichm\"uller space.

Let us investigate Penner's $\lambda$-coordinates in view of the $\mathcal{A}$-coordinates defined in Section \ref{sec:decorated_higher_teichmueller_space} .
Given a decorated marked hyperbolic structure $(\Omega , \Gamma , \phi , \mathcal{B}) \in \decTeich(S)$ and 
an ideal triangulation $\Delta \subset S$, the only data that is missing to specify $\mathcal{A}$-coordinates is a covector decoration.
However, in case of a hyperbolic structures the Minkowski bilinear form allows us to make a somewhat canonical choice for this covector decoration.
For a light-cone vector $V \in \mathcal{L}^+$ consider the linear form $\covtx_V \in (\R^3)^\ast$ given by  $\covtx_V(X) = - \langle V , X \rangle$.  
This choice of covector decoration yields symmetric edge parameters by symmetry of the Minkowki form, i.e.
\begin{equation} \label{eq:edge_coords_penner}
a^+ = - \covtx_{V_0}(V_1) = - \langle V_0 , V_1 \rangle = - \langle V_1 , V_0 \rangle = - \covtx_{V_1}(V_0) = a^- \enspace , 
\end{equation}
where $a^\pm$ are the two edge parameters along an edge $a \in \Delta$.
Furthermore, consider a concrete decorated triangle $(R,C)$ with this choice of covector decoration. 
Denote the three adjacent edge parameters around $(R,C)$  by $a=\langle \vtx_0 , \vtx_1 \rangle$, $b=\langle \vtx_0 , \vtx_2 \rangle$ and $c=\langle \vtx_1 , \vtx_2 \rangle$.
The associated positive counter-diagonal matrix $R \cdot C$ is the negative Gram matrix of the Minkowski bilinear form with respect to the basis $C=(\vtx_0 , \vtx_1 , \vtx_2)$, i.e.
\[
R \cdot C = \begin{pmatrix} 0 & -a & -b \\ -a  & 0 & -c  \\ -b  & -c  & 0 \end{pmatrix} = C^T  \begin{pmatrix} 1 & 0 & 0 \\ 0 & 1 & 0 \\ 0 & 0 & -1 \end{pmatrix} C \enspace .
\] 
It follows that the triangle parameter $A= \det(C)$ associated to the concrete decorated triangle $(R,C)$ satisfies
\begin{equation} \label{eq:triang_coords_penner}
A^2  = 2abc \enspace . 
\end{equation}
To summarize,  Penner's decorated Teichm\"uller space $\decTeich(S)$ can be identified as an $E$-dimensional subvariety $\A_\hyp(S)$
in the space of $\A$-coordinates $\A \cong \R_{>0}^{T+2E}$ subject to the conditions \eqref{eq:edge_coords_penner} and \eqref{eq:triang_coords_penner} for the triangle and edge parameters.

By construction, the  subvariety $\A_\hyp(S)$ is invariant under the duality map $\sigma^\ddagger$ in $\A$--coordinates given by \eqref{eq:dual_edges} and \eqref{eq:dual_triangles}. A simple calculation shows that it is also invariant under the change of coordinates given by equations \eqref{eq:flip_triang} and \eqref{eq:flip_edge}, and hence under the action of the mapping class group $\MCG(S).$

It follows from the defining equations that $\A_\hyp$ is the image of an embedding $\R_{>0}^E \hookrightarrow \A$.
Its inverse map is given by the projection $\pi: \A \mapsto \R_{>0}^E$ that forgets about all triangle parameters and only chooses one orientation for each edge, say
\[
\pi (\trCoords , \eCoords) = \eCoords^+ \in \R_{>0}^E \enspace .
\]
We identify $\A_\hyp$ with its diffeomorphic image $\pi( \A_\hyp) \cong \R_{>0}^E$. A point in our decorated Teichm\"uller space $\A_\hyp$ is thus specified by a choice of one positive parameter per unoriented edge.
Note that, by definition, the edge parameters we use for decorated hyperbolic structures are related to Penner's $\lambda$-lengths by
\begin{equation}\label{eq:A_to_lambda}
e = \lambda^2_e \enspace .
\end{equation}

Let $\eCoords \in \A_\hyp$ be a decorated hyperbolic structure and consider an edge $e \in \Delta$ with surrounding $\A$-cordinates as in Figure \ref{fig:outitude_A-hyp}.
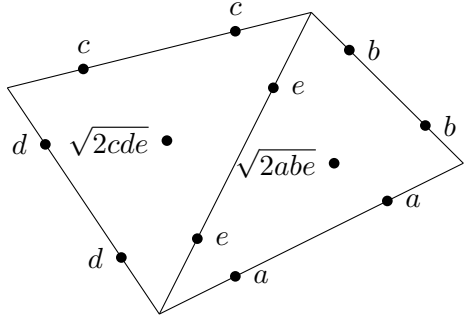
\begin{figure}
	\center
	\begin{tikzpicture}
	\draw (4, 4) edge  (6, 2); 
	\draw (6, 2) edge (2,0); 
	\draw (2, 0) edge (4, 4); 
	\draw (4, 4) edge (0, 3); 
	\draw (0, 3) edge (2, 0); 
	\node [circle, fill = black, inner sep = 0pt, minimum size = 4pt, label = left : {$\sqrt{2abe}$} ] (A) at (4.3, 2) {};	
	\node [circle, fill = black, inner sep = 0pt, minimum size = 4pt, label = left : {$\sqrt{2cde}$} ] (B) at (2.1, 2.3) {};
	\node [circle, fill = black, inner sep = 0pt, minimum size = 4pt, label = right : {$e$} ] (ep) at (2.5, 1) {};
	\node [circle, fill = black, inner sep = 0pt, minimum size = 4pt, label = right : {$e$} ] (em) at (3.5, 3) {};
	\node [circle, fill = black, inner sep = 0pt, minimum size = 4pt, label = right : {$a$} ] (am) at (3, 0.5) {};
	\node [circle, fill = black, inner sep = 0pt, minimum size = 4pt, label = right : {$a$} ] (ap) at (5, 1.5) {};
	\node [circle, fill = black, inner sep = 0pt, minimum size = 4pt, label = right : {$b$} ] (bp) at (4.5, 3.5) {};
	\node [circle, fill = black, inner sep = 0pt, minimum size = 4pt, label = right : {$b$} ] (bm) at (5.5, 2.5) {};
	\node [circle, fill = black, inner sep = 0pt, minimum size = 4pt, label = above : {$c$} ] (cp) at (3, 3.75) {};
	\node [circle, fill = black, inner sep = 0pt, minimum size = 4pt, label = above : {$c$} ] (cm) at (1, 3.25) {};
	\node [circle, fill = black, inner sep = 0pt, minimum size = 4pt, label = left : {$d$} ] (dp) at (0.5, 2.25) {};
	\node [circle, fill = black, inner sep = 0pt, minimum size = 4pt, label = left : {$d$} ] (dm) at (1.5, 0.75) {};
	\end{tikzpicture}
	\caption{The edge and triangle parameters for $\A_\hyp$ around an edge $e \in \Delta$.} \label{fig:outitude_A-hyp}
\end{figure}
In this case the outitude \eqref{eq:A_Out} along an edge $e$ specializes to
\[
\Out_{\eCoords }(e) = \sqrt{2abe} (ec + ed - e^2 ) + \sqrt{2cde} ( ea + eb - e^2) \enspace .
\]
It follows that $\Out_{\eCoords}(e) > 0$ if and only if
\begin{equation}\label{eq:hyp_Out}
\sqrt{ab} (c + d - e ) + \sqrt{cd} ( a + b - e) > 0 \enspace .
\end{equation}

It follows from Penner's work that the outitude conditions give $\A_\hyp$ a cell-decomposition. In particular, the subvariety $\A_\hyp(S)$ meets each cell $\mathring{ \mathcal{C} } ( \Delta )$ in a properly embedded cell.


\subsection{Outitude and duality}
\label{subsec:Outitude and duality}

We give an example of mutually dual doubly-decorated convex projective structures that have different canonical cell decompositions. The example is of a once punctured surface, so the canonical cell decomposition is independent of the double decoration of the dual convex projective structure chosen by our convention.

\begin{exm}\label{ex:dual_doubletorus}
Figure \ref{fig:dual_doubletorus} shows a triangulation of a once-punctured double torus $S_{2,1}$ together with edge coordinates (blue) and triangle coordinates (red) that form the data of $\A$-coordinates on this triangulation.
Denote by $\alpha \in \A$ the doubly-decorated convex projective structure whose $\A$-coordinates are given by
\[
b_0^+=b_2^+=b_5^-=b_6^-=b_7^+=2 \enspace , \quad b_6^+=3 \enspace ,
\]
and where all other edge and triangle parameters are equal to $1$ (cf. Figure \ref{fig:dual_doubletorus}).
Computing the outitutes for $\alpha$ we obtain
\[
\begin{array}{lllll}
	\Out_{\alpha}(b_0) = 4 \enspace , & \Out_{\alpha}(b_1) = 3 \enspace , & \Out_{\alpha}(b_2) = 3 \enspace , & \Out_{\alpha}(b_3) = 3 \enspace , & \Out_{\alpha}(b_4) = 3 \\
	\Out_{\alpha}(b_5) = 3 \enspace , & \Out_{\alpha}(b_6) = 8 \enspace , & \Out_{\alpha}(b_7) = 4 \enspace , & \Out_{\alpha}(b_8) = 3 \enspace . & ~ 
\end{array} 
\]
Since all outitude values are positive the triangulation $\tri$ is the canonical for $\alpha$, i.e. $\alpha \in \mathring{ \mathcal{C}}(\tri)$.
We will now show that the projectively dual structure $\alpha^\perp$ is not contained in $\mathring{ \mathcal{C}}(\tri)$.
The projectively dual structure $\alpha^\perp$ is given by 
\[
b_0^-=b_2^-=b_5^+=b_6^+=b_7^-=2 \enspace , \quad b_6^-=3 \enspace ,\quad  A_0=A_1=A_4=A_5=3 \enspace , \quad A_2=A_3=10 \enspace ,
\]
and again all other parameters equal to $1$.
This results in a negative outitude value $\Out_{\alpha^\perp}=-20$, hence $\alpha^\perp \notin \mathring{ \mathcal{C}}(\tri)$.
\begin{figure}
\centering
\begin{tikzpicture}
\node at (-3,0) { \begin{tikzpicture}
\begin{scriptsize}
\draw (4, 6) -- (6, 2)
node [circle, fill = blue1,  inner sep = 0pt, minimum size = 4pt,  label = below: {\color{blue1} $b^-_8$}, pos = 0.3] () {}
node [circle, fill = blue1,  inner sep = 0pt, minimum size = 4pt,  label = below: {\color{blue1} $b^+_8$}, pos = 0.7] () {};
\draw (4, 0) -- (2, 6)
node [circle, fill = blue1,  inner sep = 0pt, minimum size = 4pt,  label = left: {\color{blue1} $b^+_6$}, pos = 0.3] () {}
node [circle, fill = blue1,  inner sep = 0pt, minimum size = 4pt,  label = right: {\color{blue1} $b^-_6$}, pos = 0.7] () {};
\draw (4, 0) -- (4, 6)
node [circle, fill = blue1,  inner sep = 0pt, minimum size = 4pt,  label = below right: {\color{blue1} $b^+_7$}, pos = 0.3] () {}
node [circle, fill = blue1,  inner sep = 0pt, minimum size = 4pt,  label = below right: {\color{blue1} $b^-_7$}, pos = 0.7] () {};
\draw (2, 0) -- (2, 6)
node [circle, fill = blue1,  inner sep = 0pt, minimum size = 4pt,  label = left: {\color{blue1} $b^+_5$}, pos = 0.3] () {}
node [circle, fill = blue1,  inner sep = 0pt, minimum size = 4pt,  label = left: {\color{blue1} $b^-_5$}, pos = 0.7] () {};
\draw (2, 0) -- (0, 4)
node [circle, fill = blue1,  inner sep = 0pt, minimum size = 4pt,  label = right: {\color{blue1} $b^+_4$}, pos = 0.3] () {}
node [circle, fill = blue1,  inner sep = 0pt, minimum size = 4pt,  label = right: {\color{blue1} $b^-_4$}, pos = 0.7] () {};
\node [circle, fill = red1, inner sep = 0pt, minimum size = 4pt, label = below : {\color{red1} $A_2$} ] at (2.75, 1.5) {};
\node [circle, fill = red1, inner sep = 0pt, minimum size = 4pt, label = above right : {\color{red1} $A_4$} ] at (4.75, 2) {};
\node [circle, fill = red1, inner sep = 0pt, minimum size = 4pt, label = below : {\color{red1} $A_5$} ] at (5.5, 4) {};
\node [circle, fill = red1, inner sep = 0pt, minimum size = 4pt, label = above : {\color{red1} $A_3$} ] at (3.25, 4.5) {};
\node [circle, fill = red1, inner sep = 0pt, minimum size = 4pt, label = above left: {\color{red1} $A_1$} ] at (1.25, 3.5) {};
\node [circle, fill = red1, inner sep = 0pt, minimum size = 4pt, label = above : {\color{red1} $A_0$} ] at (0.5, 2) {};

\draw (2, 0) -- (4, 0)
node [circle, fill = blue1,  inner sep = 0pt, minimum size = 4pt,  label = below: {\color{blue1} $b^-_0$}, pos = 0.3] () {}
node [circle, fill = blue1,  inner sep = 0pt, minimum size = 4pt,  label = below: {\color{blue1} $b^+_0$}, pos = 0.7] () {};
\draw (4, 0) -- (6, 2)
node [circle, fill = blue1,  inner sep = 0pt, minimum size = 4pt,  label = below right: {\color{blue1} $b^-_1$}, pos = 0.3] () {}
node [circle, fill = blue1,  inner sep = 0pt, minimum size = 4pt,  label = below right: {\color{blue1} $b^+_1$}, pos = 0.7] () {};
\draw (6, 2) -- (6, 4)
node [circle, fill = blue1,  inner sep = 0pt, minimum size = 4pt,  label = right: {\color{blue1} $b^+_2$}, pos = 0.3] () {}
node [circle, fill = blue1,  inner sep = 0pt, minimum size = 4pt,  label = right: {\color{blue1} $b^-_2$}, pos = 0.7] () {};
\draw (6, 4) -- (4, 6)
node [circle, fill = blue1,  inner sep = 0pt, minimum size = 4pt,  label = above right: {\color{blue1} $b^+_3$}, pos = 0.3] () {}
node [circle, fill = blue1,  inner sep = 0pt, minimum size = 4pt,  label = above right: {\color{blue1} $b^-_3$}, pos = 0.7] () {};
\draw (4, 6) -- (2, 6)
node [circle, fill = blue1,  inner sep = 0pt, minimum size = 4pt,  label = above: {\color{blue1} $b^-_2$}, pos = 0.3] () {}
node [circle, fill = blue1,  inner sep = 0pt, minimum size = 4pt,  label = above: {\color{blue1} $b^+_2$}, pos = 0.7] () {};
\draw (2, 6) -- (0, 4)
node [circle, fill = blue1,  inner sep = 0pt, minimum size = 4pt,  label = above: {\color{blue1} $b^-_3$}, pos = 0.3] () {}
node [circle, fill = blue1,  inner sep = 0pt, minimum size = 4pt,  label = above: {\color{blue1} $b^+_3$}, pos = 0.7] () {};
\draw (0, 4) -- (0, 2)
node [circle, fill = blue1,  inner sep = 0pt, minimum size = 4pt,  label = left: {\color{blue1} $b^+_0$}, pos = 0.3] () {}
node [circle, fill = blue1,  inner sep = 0pt, minimum size = 4pt,  label = left: {\color{blue1} $b^-_0$}, pos = 0.7] () {};
\draw (0, 2) -- (2, 0)
node [circle, fill = blue1,  inner sep = 0pt, minimum size = 4pt,  label = below left: {\color{blue1} $b^+_1$}, pos = 0.3] () {}
node [circle, fill = blue1,  inner sep = 0pt, minimum size = 4pt,  label = below left: {\color{blue1} $b^-_1$}, pos = 0.7] () {};
\end{scriptsize}	
\end{tikzpicture}};
\node at (5,0) { \begin{tikzpicture}
	\draw (4, 6) -- (6, 2)
	node [circle, fill = blue1,  inner sep = 0pt, minimum size = 4pt,  label = below: {}, pos = 0.3] () {}
	node [circle, fill = blue1,  inner sep = 0pt, minimum size = 4pt,  label = below: {}, pos = 0.7] () {};
	\draw (4, 0) -- (2, 6)
	node [circle, fill = blue1,  inner sep = 0pt, minimum size = 4pt,  label = left: {\color{blue1} $3$}, pos = 0.3] () {}
	node [circle, fill = blue1,  inner sep = 0pt, minimum size = 4pt,  label = right: {\color{blue1} $2$}, pos = 0.7] () {};
	\draw (4, 0) -- (4, 6)
	node [circle, fill = blue1,  inner sep = 0pt, minimum size = 4pt,  label = right: {\color{blue1} $2$}, pos = 0.3] () {}
	node [circle, fill = blue1,  inner sep = 0pt, minimum size = 4pt,  label = below right: {}, pos = 0.7] () {};
	\draw (2, 0) -- (2, 6)
	node [circle, fill = blue1,  inner sep = 0pt, minimum size = 4pt,  label = left: {}, pos = 0.3] () {}
	node [circle, fill = blue1,  inner sep = 0pt, minimum size = 4pt,  label = left: {\color{blue1} $2$}, pos = 0.7] () {};
	\draw (2, 0) -- (0, 4)
	node [circle, fill = blue1,  inner sep = 0pt, minimum size = 4pt,  label = right: {}, pos = 0.3] () {}
	node [circle, fill = blue1,  inner sep = 0pt, minimum size = 4pt,  label = right: {}, pos = 0.7] () {};
	\node [circle, fill = red1, inner sep = 0pt, minimum size = 4pt, label = below : {} ] at (2.75, 1.5) {};
	\node [circle, fill = red1, inner sep = 0pt, minimum size = 4pt, label = above right : {} ] at (4.75, 2) {};
	\node [circle, fill = red1, inner sep = 0pt, minimum size = 4pt, label = below : {} ] at (5.5, 4) {};
	\node [circle, fill = red1, inner sep = 0pt, minimum size = 4pt, label = above : {} ] at (3.25, 4.5) {};
	\node [circle, fill = red1, inner sep = 0pt, minimum size = 4pt, label = above left: {} ] at (1.25, 3.5) {};
	\node [circle, fill = red1, inner sep = 0pt, minimum size = 4pt, label = above : {} ] at (0.5, 2) {};
	
	\draw (2, 0) -- (4, 0)
	node [circle, fill = blue1,  inner sep = 0pt, minimum size = 4pt,  label = below: {}, pos = 0.3] () {}
	node [circle, fill = blue1,  inner sep = 0pt, minimum size = 4pt,  label = below: {\color{blue1} $2$}, pos = 0.7] () {};
	\draw (4, 0) -- (6, 2)
	node [circle, fill = blue1,  inner sep = 0pt, minimum size = 4pt,  label = below right: {}, pos = 0.3] () {}
	node [circle, fill = blue1,  inner sep = 0pt, minimum size = 4pt,  label = below right: {}, pos = 0.7] () {};
	\draw (6, 2) -- (6, 4)
	node [circle, fill = blue1,  inner sep = 0pt, minimum size = 4pt,  label = right: {\color{blue1} $2$}, pos = 0.3] () {}
	node [circle, fill = blue1,  inner sep = 0pt, minimum size = 4pt,  label = right: {}, pos = 0.7] () {};
	\draw (6, 4) -- (4, 6)
	node [circle, fill = blue1,  inner sep = 0pt, minimum size = 4pt,  label = above right: {}, pos = 0.3] () {}
	node [circle, fill = blue1,  inner sep = 0pt, minimum size = 4pt,  label = above right: {}, pos = 0.7] () {};
	\draw (4, 6) -- (2, 6)
	node [circle, fill = blue1,  inner sep = 0pt, minimum size = 4pt,  label = above: {}, pos = 0.3] () {}
	node [circle, fill = blue1,  inner sep = 0pt, minimum size = 4pt,  label = above: {\color{blue1} $2$}, pos = 0.7] () {};
	\draw (2, 6) -- (0, 4)
	node [circle, fill = blue1,  inner sep = 0pt, minimum size = 4pt,  label = above: {}, pos = 0.3] () {}
	node [circle, fill = blue1,  inner sep = 0pt, minimum size = 4pt,  label = above: {}, pos = 0.7] () {};
	\draw (0, 4) -- (0, 2)
	node [circle, fill = blue1,  inner sep = 0pt, minimum size = 4pt,  label = left: {\color{blue1} $2$}, pos = 0.3] () {}
	node [circle, fill = blue1,  inner sep = 0pt, minimum size = 4pt,  label = left: {}, pos = 0.7] () {};
	\draw (0, 2) -- (2, 0)
	node [circle, fill = blue1,  inner sep = 0pt, minimum size = 4pt,  label = below left: {}, pos = 0.3] () {}
	node [circle, fill = blue1,  inner sep = 0pt, minimum size = 4pt,  label = below left: {}, pos = 0.7] () {};
	\end{tikzpicture}};
\end{tikzpicture}
\caption{$\A$-coordinates on a triangulation of a once-punctured double torus $S_{2,1}$ (left). The right figure shows the specific choice of coordinates $\alpha \in \A$ from Example \ref{ex:dual_doubletorus} where we omit to label those triangle and edge parameters with value $1$.}
\label{fig:dual_doubletorus}
\end{figure}
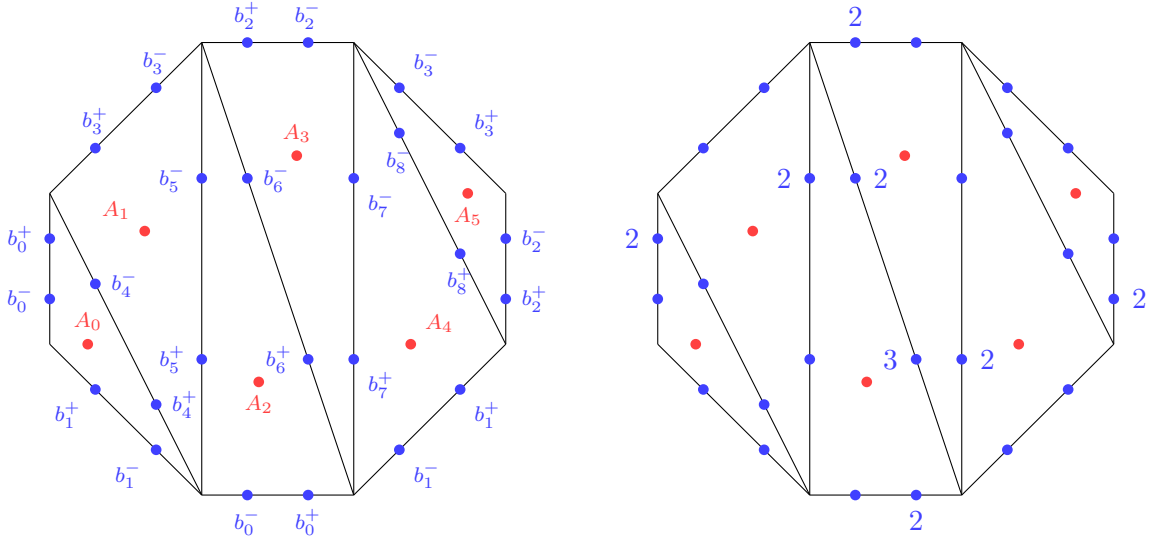
\end{exm}

The above example can easily be generalised to show that for all $g\ge 2,$ there are mutually dual structures of $S_{g,1}$ with different canonical cell decompositions. The situation is different for $S_{1,1}$ as the following example shows.

\begin{exm}\label{ex:dual_torus}
Consider $\A$-coordinates $\alpha=(A,B,a^+,a^-,b^+,b^-,c^+,c^-)$ on the triangulation $\tri$ of a once-punctured torus depicted on the left in Figure \ref{fig:dual_torus}.
The projectively dual structure has $\A$-coordinates $\alpha^\perp=(\tfrac{a^+b^+c^+ + a^-b^-c^-}{A},\tfrac{a^+b^+c^+ + a^-b^-c^-}{B},a^-,a^+,b^-,b^+,c^-,c^+)$, where we keep the ordering of triangle and edge coordinates as depicted in the left of Figure \ref{fig:dual_torus}.
The outitude of edge $c \in \tri$ in the projectively dual structure $\alpha^\perp$ is thus given by
\[
\Out_{\alpha^\perp}(c)= \tfrac{a^+b^+c^+ ~+~ a^-b^-c^-}{A}( a^-c^- + b^+c^+ - c^+c^- ) + \tfrac{a^+b^+c^+ ~+~ a^-b^-c^-}{B}( a^+c^+ + b^-c^- - c^+c^- ) \enspace .
\] 
Multiplying with the positive scalar $\tfrac{AB}{a^+b^+c^+ ~+~ a^-b^-c^-}$ we see that $\Out_{\alpha^\perp}(c) > 0$ if and only if 
\[
B( a^-c^- + b^+c^+ - c^+c^- ) + A( a^+c^+ + b^-c^- - c^+c^- ) > 0 \enspace .
\]
By symmetry the same holds for the other two edges $a,b \in \tri$.
It follows that in this particular example, $\mathring{ \mathcal{C}}(\tri)$ is invariant under projective duality.

\begin{figure}
\centering
\begin{tikzpicture}[scale=1.25]
\node [circle, fill = red1, inner sep = 0pt, minimum size = 4pt, label = above : {\color{red1} $A$} ] (A) at (1,2) {};
\node [circle, fill = red1, inner sep = 0pt, minimum size = 4pt, label = below : {\color{red1} $B$} ] (B) at (2,1) {};	
\draw (0,0) -- (0,3) -- (3,3) -- (3,0) -- cycle;
\draw (0,0) -- (3,3);
\node [circle, fill = blue1, inner sep = 0pt, minimum size = 4pt, label = below : {\color{blue1} $a^+$}] at ( 1, 0 ) {};
\node [circle, fill = blue1, inner sep = 0pt, minimum size = 4pt, label = below : {\color{blue1} $a^-$}] at ( 2, 0 ) {};
\node [circle, fill = blue1, inner sep = 0pt, minimum size = 4pt, label = left : {\color{blue1} $b^+$}] at ( 0, 1 ) {}; 
\node [circle, fill = blue1, inner sep = 0pt, minimum size = 4pt, label = left : {\color{blue1} $b^-$}] at ( 0, 2 ) {}; 
\node [circle, fill = blue1, inner sep = 0pt, minimum size = 4pt, label = above : {\color{blue1} $a^+$}] at ( 1, 3 ) {};
\node [circle, fill = blue1, inner sep = 0pt, minimum size = 4pt, label = above : {\color{blue1} $a^-$}] at ( 2, 3 ) {};
\node [circle, fill = blue1, inner sep = 0pt, minimum size = 4pt, label = right : {\color{blue1} $b^+$}] at ( 3, 1 ) {}; 
\node [circle, fill = blue1, inner sep = 0pt, minimum size = 4pt, label = right : {\color{blue1} $b^-$}] at ( 3, 2 ) {}; 
\node [circle, fill = blue1, inner sep = 0pt, minimum size = 4pt, label = right : {\color{blue1} $c^+$}] at ( 2, 2 ) {};
\node [circle, fill = blue1, inner sep = 0pt, minimum size = 4pt, label = left : {\color{blue1} $c^-$}] at ( 1, 1 ) {};	

\draw[<->] (3.8,1.5) -- (5.2,1.5) node [label = above: {dualizing}, pos = 0.5] () {};

\node [circle, fill = red1, inner sep = 0pt, minimum size = 4pt, label = above : {\scriptsize \color{red1} $\tfrac{a^+b^+c^+ ~+~ a^-b^-c^-}{A}$} ] (A) at (7,2) {};
\node [circle, fill = red1, inner sep = 0pt, minimum size = 4pt, label = below : {\scriptsize \color{red1} $\tfrac{a^+b^+c^+ ~+~ a^-b^-c^-}{B}$} ] (B) at (8,1) {};	
\draw (6,0) -- (6,3) -- (9,3) -- (9,0) -- cycle;
\draw (6,0) -- (9,3);
\node [circle, fill = blue1, inner sep = 0pt, minimum size = 4pt, label = below : {\color{blue1} $a^-$}] at ( 7, 0 ) {};
\node [circle, fill = blue1, inner sep = 0pt, minimum size = 4pt, label = below : {\color{blue1} $a^+$}] at ( 8, 0 ) {};
\node [circle, fill = blue1, inner sep = 0pt, minimum size = 4pt, label = left : {\color{blue1} $b^-$}] at ( 6, 1 ) {}; 
\node [circle, fill = blue1, inner sep = 0pt, minimum size = 4pt, label = left : {\color{blue1} $b^+$}] at ( 6, 2 ) {}; 
\node [circle, fill = blue1, inner sep = 0pt, minimum size = 4pt, label = above : {\color{blue1} $a^-$}] at ( 7, 3 ) {};
\node [circle, fill = blue1, inner sep = 0pt, minimum size = 4pt, label = above : {\color{blue1} $a^+$}] at ( 8, 3 ) {};
\node [circle, fill = blue1, inner sep = 0pt, minimum size = 4pt, label = right : {\color{blue1} $b^-$}] at ( 9, 1 ) {}; 
\node [circle, fill = blue1, inner sep = 0pt, minimum size = 4pt, label = right : {\color{blue1} $b^+$}] at ( 9, 2 ) {}; 
\node [circle, fill = blue1, inner sep = 0pt, minimum size = 4pt, label = right : {\color{blue1} $c^-$}] at ( 8, 2 ) {};
\node [circle, fill = blue1, inner sep = 0pt, minimum size = 4pt, label = left : {\color{blue1} $c^+$}] at ( 7, 1 ) {};		
\end{tikzpicture}	
\caption{A triangulation $\tri$ of a once-punctured torus with primal (left) and dual (right) $\A$-coordinates.}
\label{fig:dual_torus}	
\end{figure}
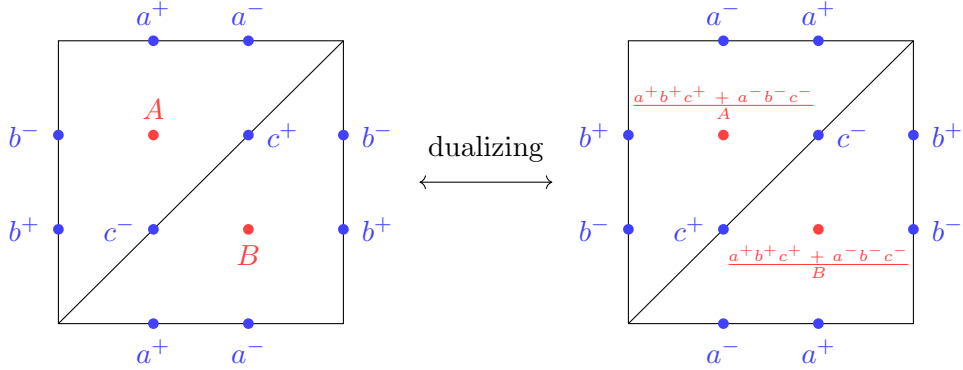
\end{exm}



\section{Other cell decompositions}
\label{sec:Other cell decompositions}

Labourie \cite{Labourie:2007} and Loftin \cite{Loftin:2001} described a bijection between convex projective structures on a compact surface $S$ and pairs $(\Sigma,U)$ of a hyperbolic structure $\Sigma$ on $S$ together with a cubic holomorphic differential $U$ on $\Sigma$.
For noncompact surfaces $S$, Benoist and Hulin \cite{BenoistHulin:2013} gave a similar bijection between finite area convex projective structures on $S$ and pairs $(\Sigma,U)$ of a finite area hyperbolic structure $\Sigma$ on $S$ and a cubic holomorphic differential $U$ on $\Sigma$ with poles of order at most 2 at the cusps.  
In both cases the authors showed that these bijections describe homeomorphism between the corresponding sets equipped with their natural topologies. Denote $V$ the vector space of cubic holomorphic differential $U$ on $\Sigma$ with poles of order at most 2 at the cusps.

Now given \emph{any} $\MCG(S)$--invariant cell decomposition of Teichm\"uller space ${\T}_\hyp(S),$ one obtains a natural $\MCG(S)$--invariant cell decomposition of ${\Tt}^f(S) \cong {\T}_\hyp(S) \times V.$ We claim that it is also invariant under duality. Loftin~\cite{Loftin:2001} shows that for $x=(\Sigma_x,U_x) \in \Ttf(S)$ the dual convex projective structure $x^\perp \in \Ttf(S)$ is obtained by replacing the cubic holomorphic differential $U_x$ by its negative, i.e. $x^\perp = (\Sigma_{x^\perp},U_{x^\perp})=(\Sigma_x,-U_x)$. In particular, the hyperbolic structures $\Sigma_x$ and $\Sigma_{x^\perp}$ coincide, and hence $x$ and $x^\perp$ lie in the same cell in $\Ttf(S).$ 

We now focus on the case of $S=S_{g,1}$, a once-punctured genus $g$ surface, since in this case the canonical cell decomposition of $( \Omega , \Gamma , \phi , ( \covec_dec , \vec_dec ) ) \in \decTtf(S)$ is independent of $( \covec_dec , \vec_dec ).$ In particular, one obtains a $\MCG(S)$--invariant cell decomposition of $\Ttf(S)$ from the outitude conditions. Example~\ref{ex:dual_doubletorus} implies that for $S_{2,1}$ this cell decomposition is not invariant under duality, and hence it cannot arise by appealing to the product structure due to Benoist and Hulin for a suitably chosen cell decomposition of Teichm\"uller space. The same applies to all $S=S_{g,1},$ where $g\ge 2.$

We complete this paper with another observation for $S=S_{g,1}.$ Denote by $\Sigma_x$ the hyperbolic structure associated to $x\in\decTtf(S)$ using the bijection of Benoist-Hulin. The convex hull construction of Epstein-Penner associates a canonical ideal cell decomposition $\tri(\Sigma_x) \subset S$ to the hyperbolic structure $\Sigma_x$, and the cells from the product structure of $\Ttf(S) \cong {\T}_\hyp(S) \times V$
are given by
\[
\mathring{ \mathcal{B}}(\tri) = \{ x \in \decTtf(S) ~|~ \tri(\Sigma_x)=\tri \} \enspace .
\] 
We wish to study how the two cells $\mathring{ \mathcal{B}}(\tri)$ and $\mathring{ \mathcal{C}}(\tri)$ are related.
If $x \in \widetilde{\Tt}^f(S)$ is hyperbolic, then the Benoist--Hulin correpondence yields $\Sigma_x=x$ and $U=0$, and therefore $\tri(\Sigma_x)=\tri(x)$.
It follows that the intersections of the two cells $\mathring{ \mathcal{B}}(\tri)$ and $\mathring{ \mathcal{C}}(\tri)$ with hyperbolic structures coincide, i.e.
\[
\mathring{\mathcal{B}}_\hyp(\tri) = \mathring{ \mathcal{C}}_\hyp(\tri) \enspace .
\]  

Example~\ref{ex:dual_torus} shows that in the case of the once-punctured torus $S_{1,1}$ the cell decomposition of $\decTtf(S_{1,1})$
is invariant under projective duality. The question whether its cells $\mathring{ \mathcal{C}}(\tri)$ coincide with the cells $\mathring{ \mathcal{B}}(\tri)$ remains open.

\section{Centres of cells}
\label{sec:cells}

Penner~\cite{Penner-Decorated-1987} describes a natural centre for each of the cells $\mathring{\mathcal{C}}(\Delta)$, where $\Delta \subset S$ is an ideal cell decomposition of $S.$ These centres are obtained by the assignment of the $\lambda$--length $1$ to each edge in $\Delta$. Suppose $\Delta$ is refined by the ideal triangulation $\Delta^\prime \subset S$. Then the $\lambda$--lengths for the edges in $\Delta^\prime \setminus \Delta$ are obtained by identifying each $n$--gon in $\Delta$ with the regular Euclidean $n$--gon of side length one, and assigning the lengths of the corresponding diagonals as $\lambda$--lengths. (See \cite[\S2.1.1]{Penner-book-2012} for a detailed discussion.) For this structure, the topological symmetry group of the ideal cell decomposition is the hyperbolic isometry group of the corresponding surface.

Penner shows that the set of all centres is invariant under the action of the mapping class group. Denote the image of the holonomy representation of the centre of $\mathring{\mathcal{C}}(\Delta)$ by $\Gamma(\Delta)\le \SL_3(\R).$ Penner shows that if $\Delta$ is a triangulation, then the centre of $\mathring{\mathcal{C}}(\Delta)$ is an arithmetic Fuchsian group $\Gamma(\Delta).$ Here, we recover this result and add some new observations concerning the remaining centres.

Our cell decomposition of $\widetilde{\mathcal{T}}(S)$ is $\MCG(S)$ invariant, and $\A_\hyp(S)$ is an $\MCG(S)$ invariant subset. It follows that Penner's centres are also natural centres for the cells in $\widetilde{\mathcal{T}}(S).$ The $\A$--coordinates of these centres are obtained from the above description of their $\lambda$--lengths, since \eqref{eq:A_to_lambda} determines the edge invariants and \eqref{eq:triang_coords_penner} subsequently the triangle invariants. 

\begin{figure}
	\centering
	\begin{tikzpicture}
	\node at (-3,0) { \begin{tikzpicture}[scale=1.8]
		\begin{scriptsize}
		\draw (0, 0) -- (0.26, 1.2)
		node [circle, fill = blue1,  inner sep = 0pt, minimum size = 4pt,  label = left: {\color{blue1} $1$}, pos = 0.5] () {};
		\draw (0, 0) -- (0.26, -1.2)
		node [circle, fill = blue1,  inner sep = 0pt, minimum size = 4pt,  label = left: {\color{blue1} $1$}, pos = 0.5] () {};
		\draw (1.2, 1.73) -- (0.26, 1.2)
		node [circle, fill = blue1,  inner sep = 0pt, minimum size = 4pt,  label = above: {\color{blue1} $1$}, pos = 0.5] () {};	
		\draw (1.2, -1.73) -- (0.26, -1.2)
		node [circle, fill = blue1,  inner sep = 0pt, minimum size = 4pt,  label = below: {\color{blue1} $1$}, pos = 0.5] () {};	
		\draw (1.2, 1.73) -- (2.1, 1.9)
		node [circle, fill = blue1,  inner sep = 0pt, minimum size = 4pt,  label = above: {\color{blue1} $1$}, pos = 0.5] () {};	
		\draw (1.2, -1.73) -- (2.1, -1.9)
		node [circle, fill = blue1,  inner sep = 0pt, minimum size = 4pt,  label = below: {\color{blue1} $1$}, pos = 0.5] () {};	
		\draw (3.73, 1) -- (4, 0)
		node [circle, fill = blue1,  inner sep = 0pt, minimum size = 4pt,  label = right: {\color{blue1} $1$}, pos = 0.5] () {};
		\draw (3.73, -1) -- (4, 0)
		node [circle, fill = blue1,  inner sep = 0pt, minimum size = 4pt,  label = right: {\color{blue1} $1$}, pos = 0.5] () {};
		\draw (3.73, 1) -- (3, 1.73)
		node [circle, fill = blue1,  inner sep = 0pt, minimum size = 4pt,  label = above: {\color{blue1} $1$}, pos = 0.5] () {};
		\draw (3.73, -1) -- (3, -1.73)
		node [circle, fill = blue1,  inner sep = 0pt, minimum size = 4pt,  label = below: {\color{blue1} $1$}, pos = 0.5] () {};
		\draw[dashed, gray!60] (2.1, 1.9) -- (3, 1.73);
		\draw[dashed, gray!60] (2.1, -1.9) -- (3, -1.73);
		\draw (0, 0) -- (3.73, 1)
		node [circle, fill = blue1,  inner sep = 0pt, minimum size = 4pt,  label = above: {\color{blue1} $d_{k-1}^2$}, pos = 0.5] () {};
		\draw (0, 0) -- (3.73, -1)
		node [circle, fill = blue1,  inner sep = 0pt, minimum size = 4pt,  label = below: {\color{blue1} $d_{k+1}^2$}, pos = 0.5] () {};
		\draw (0, 0) -- (4, 0)
		node [circle, fill = blue1,  inner sep = 0pt, minimum size = 4pt, pos = 0.5] () {};
		\draw (1.75, 0) node [circle, fill = white,  inner sep = 0pt, minimum size = 10pt,  label = center: {\color{blue1} $d_k^2$}] () {};
		\draw (0, 0) -- (1.2, 1.73)
		node [circle, fill = blue1,  inner sep = 0pt, minimum size = 4pt,  label = right: {\color{blue1} $d_1^2$}, pos = 0.5] () {};
		\draw (0, 0) -- (1.2, -1.73)
		node [circle, fill = blue1,  inner sep = 0pt, minimum size = 4pt,  label = right: {\color{blue1} $d_{n-3}^2$}, pos = 0.5] () {};
		\draw (2.6, 0.35) node [circle, fill = red1,  inner sep = 0pt, minimum size = 4pt,  label = right: {\color{red1} $\sqrt{2} d_{k-1}d_{k} $}] () {};
		\draw (2.6, -0.35) node [circle, fill = red1,  inner sep = 0pt, minimum size = 4pt,  label = right: {\color{red1} $\sqrt{2} d_{k}d_{k+1}$}] () {};
		\draw (0.4, 1) node [circle, fill = red1,  inner sep = 0pt, minimum size = 4pt] () {};
		\draw (0.3, 1.25) node [circle, fill = white,  inner sep = 0pt, minimum size = 20pt, label = center: {\color{red1} $\sqrt{2} d_1$}] () {};
		\draw (0.4, -1) node [circle, fill = red1,  inner sep = 0pt, minimum size = 4pt] () {};
		\draw (0.3, -1.25) node [circle, fill = white,  inner sep = 0pt, minimum size = 20pt, label = center: {\color{red1} $\sqrt{2} d_{n-3}$}] () {};		
		\end{scriptsize}	
		\end{tikzpicture}};
\end{tikzpicture}
\caption{Labeling of triangle and edge parameters of a standard triangulation of a hyperbolic center of a cell.}
\label{fig:Standard_Triangulation-Diagos}
\end{figure}
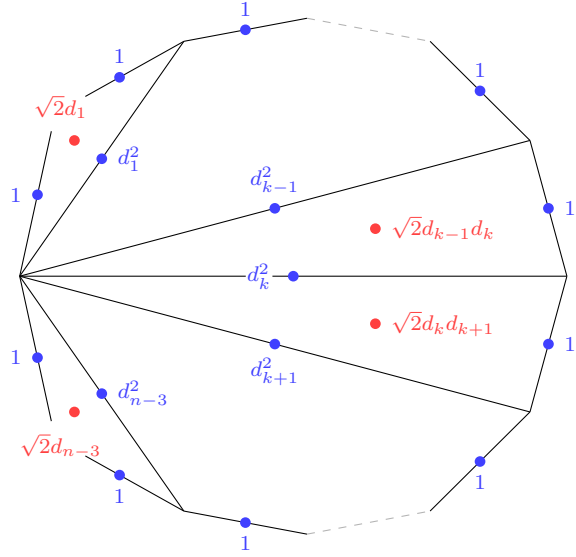

Using the standard triangulation and the assignment of $\A$--coordinates as in Figure~\ref{fig:Standard_Triangulation-Diagos}, the shown edge and triangle parameters satisfy the following recurrence relations, arising from the outitude conditions for all but the last diagonal. Note that the shortest diagonal in a regular $n$--gon with side lengths one has length $2\cos(\tfrac{\pi}{n}).$
For $2\le k \le n-3$ we have:
\begin{align*}
d_1&= 2\cos(\tfrac{\pi}{n}) & && d'_1&= 2\cos(\tfrac{\pi}{n})\\
d_2&= d_1^2-1 & && d'_2&= (d'_1)^2-1\\
d_k&=\frac{d_{k-1}^2-1}{d_{k-2}} &&& d'_k&= d'_1 d'_{k-1} - d'_{k-2}
\end{align*}
Here, the recurrence for $d_k$ follows by applying the outitude condition, and the equality $d_1 = d'_1$ follows by complete induction starting with the first three terms. It was shown by Lehmer~\cite{Lehmer-note-1933} that $2\cos(\tfrac{\pi}{n})$ is an algebraic integer, and it follows from the sequence $(d'_k)$ that all edge parameters of the $n$--gon are algebraic integers for \emph{any} subdivision of the $n$--gon into triangles. The field $\Q(2\cos(\tfrac{\pi}{n}))$ is the maximal real subfield of $\Q(\xi_n),$ where $\xi_n$ is a primitive root of unity. In particular, the edge invariants of the centre of $\mathring{\mathcal{C}}(\Delta)$ are algebraic integers.

Let $\text{Poly}(\Delta)\subset \N$ be the set of all $n\ge 4$ with the property that $\Delta$ contains at least one $n$--gon.  Let $\ell = \text{lcm}(n \mid n \in \text{Poly}(\Delta),$ and denote $\Q(\Delta) = \Q(\cos(\tfrac{\pi}{\ell})).$ Then $\Q(\Delta)$ is a totally real number field and for each $n \in \text{Poly}(\Delta)$ we have $2\cos(\tfrac{\pi}{n}) \in \Z[\cos(\tfrac{\pi}{\ell})]$ by application of Chebycheff polynomials. If $\Delta$ is a triangulation, then $\text{Poly}(\Delta)=\emptyset$ and $\Q(\Delta) = \Q.$

Referring to Figure~\ref{fig:A_to_X_coords} and letting $a^+=a^- = a^2,$ $b^+=b^- = b^2,$ $c^+=c^- = c^2,$ $d^+=d^- = d^2,$ $e^+=e^- = e^2,$ $A = \sqrt{2}abe$ and $B = \sqrt{2}dce$, it follows that every element of $\Gamma(\Delta)$ is a product of matrices of the form
\begin{equation}
T = \begin{pmatrix} 0 & 0 & 1 \\ 0 & -1 & -1 \\ 1 & 2 & 1 \end{pmatrix} \enspace \text{ and  } E = \begin{pmatrix} 0 & 0 & \tfrac{bd}{ac} \\ 0 & -1 & 0 \\ \frac{ac}{bd} & 0 & 1 \end{pmatrix} \enspace.
\end{equation}
where the edge invariants $a, b, c, d \in \Q(\Delta)$  are algebraic integers.

If $\Delta$ is a triangulation, then all edge invariants equal one. We therefore directly recover Penner's observation that in this case $\Gamma(\Delta)\le \SL_3(\Z)$, so $\Gamma(\Delta)$ is arithmetic,  We may ask whether centres of lower dimensional cells give rise to \emph{semi-arithmetic groups} in the sense of Schmutz Schaller and Wolfart~\cite{Schaller-semi-2000}. Here, a cofinite Fuchsian group $\Gamma \le \PSL_2(\R)$ is \emph{semi-arithmetic} if all traces of elements in $\Gamma^2,$ the subgroup generated by all squares of elements of $\Gamma$, are algebraic integers of a totally real number field. In particular, $\Gamma(\Delta)$ corresponds to a semi-arithmetic group if $\Gamma(\Delta)\le \SL_3(\mathcal{O}),$ where $\mathcal{O}$ is the ring of algebraic integers of a totally real number field. 

It follows from the description of the holonomy that $\Gamma(\Delta)$ is semi-arithmetic if the reciprocal of each each edge invariant is also an algebraic integer. It follows from equation (5) in \cite{Steinbach-golden-1997} that if $n$ is odd, the reciprocal of $2\cos(\tfrac{\pi}{n})$ is also an algebraic integer. If $n$ is even, this is not always the case. For instance, inspection of the minimal polynomials yields that for $n$ even in the range $[70]$, the reciprocal of $2\cos(\tfrac{\pi}{n})$ is an algebraic integer if and only if $n \in \{ 12, 20, 24, 28, 30, 36, 40, 42, 44, 48, 52, 56, 60, 66, 68, 70\}.$



\bibliographystyle{plain}
\bibliography{00_Cell_Decomp_main}

\begin{thebibliography}{10}

\bibitem{BenoistHulin:2013}
Yves Benoist and Dominique Hulin.
\newblock Cubic differentials and finite volume convex projective surfaces.
\newblock {\em Geom. Topol.}, 17(1):595--620, 2013.

\bibitem{BoEp-natural-1988}
Brian~H. Bowditch and David B.~A. Epstein.
\newblock {Natural Triangulations Associated to a Surface}.
\newblock {\em Topology}, 27(1):91--117, 1988.

\bibitem{Burger-higher-2014}
Marc Burger, Alessandra Iozzi, and Anna Wienhard.
\newblock Higher {T}eichm\"{u}ller spaces: from {${\rm SL}(2,\Bbb R)$} to other
  {L}ie groups.
\newblock In {\em Handbook of {T}eichm\"{u}ller theory. {V}ol. {IV}}, volume~19
  of {\em IRMA Lect. Math. Theor. Phys.}, pages 539--618. Eur. Math. Soc.,
  Z\"{u}rich, 2014.

\bibitem{CTT-moduli-2018}
Alex Casella, Dominic Tate, and Stephan Tillmann.
\newblock {\em {Moduli Spaces of Real Projective Structures on Surfaces}},
  volume~38 of {\em MSJ Memoirs}.
\newblock Mathematical Society of Japan, Tokyo, 2019.

\bibitem{Cooper-marked-2010}
Daryl Cooper and Kelly Delp.
\newblock The marked length spectrum of a projective manifold or orbifold.
\newblock {\em Proc. Amer. Math. Soc.}, 138(9):3361--3376, 2010.

\bibitem{CoLo-Epstein-Penner-2015}
Daryl Cooper and Darren Long.
\newblock {A generalization of the Epstein-Penner Construction to Projective
  Manifolds}.
\newblock {\em Proc. Amer. Math. Soc.}, 143(10):4561--4569, 2015.

\bibitem{EpPe-Euclidean-1988}
D.~B.~A. Epstein and R.~C. Penner.
\newblock {Euclidean Decompositions of Noncompact Hyperbolic Manifolds}.
\newblock {\em J. Differential Geom.}, 27(1):67--80, 1988.

\bibitem{FoGo-moduli-2006}
V.~V. Fock and A.~B. Goncharov.
\newblock {Moduli Spaces of Local Systems and Higher Teichmueller Theory}.
\newblock {\em Publ. Math. IHES}, 103:1--211, 2006.

\bibitem{FoGo-moduli-2007}
V.~V. Fock and A.~B. Goncharov.
\newblock {Moduli Spaces of Convex Projective Structures on Surfaces}.
\newblock {\em Adv. Math.}, 208(1):249--273, 2007.

\bibitem{polymake}
Ewgenij Gawrilow and Michael Joswig.
\newblock \texttt{polymake}: a framework for analyzing convex polytopes.
\newblock In {\em Polytopes---combinatorics and computation (Oberwolfach,
  1997)}, volume~29 of {\em DMV Sem.}, pages 43--73. Birk\-h\"au\-ser, Basel,
  2000.

\bibitem{Goldman-convex-1990}
William Goldman.
\newblock {Convex Real Projective Structures on Compact Surfaces}.
\newblock {\em J. Differential Geometry}, 31(3):791--845, 1990.

\bibitem{Goldman-geometric-1988}
William~M. Goldman.
\newblock Geometric structures on manifolds and varieties of representations.
\newblock In {\em Geometry of group representations ({B}oulder, {CO}, 1987)},
  volume~74 of {\em Contemp. Math.}, pages 169--198. Amer. Math. Soc.,
  Providence, RI, 1988.

\bibitem{HaTi-tessellating-2017}
Robert~C. Haraway, III and Stephan Tillmann.
\newblock Tessellating the moduli space of strictly convex projective
  structures on the once-punctured torus.
\newblock {\em Exp. Math.}, 28(3):369--384, 2019.

\bibitem{Harer-cohomological-1986}
John Harer.
\newblock {The Virtual Cohomological Dimension of the Mapping Class Group of an
  Orientable Surface}.
\newblock {\em Invent. Math.}, 84(1):157--176, 1986.

\bibitem{Hatcher-triangulations-1991}
Allen Hatcher.
\newblock {On Triangulations of Surfaces}.
\newblock {\em Top. and its Appl.}, 40(2):189--194, 1991.

\bibitem{Hirsch-differential-1976}
Morris~W. Hirsch.
\newblock {\em Differential topology}.
\newblock Springer-Verlag, New York-Heidelberg, 1976.
\newblock Graduate Texts in Mathematics, No. 33.

\bibitem{Labourie:2007}
Fran\c{c}ois Labourie.
\newblock Flat projective structures on surfaces and cubic holomorphic
  differentials.
\newblock {\em Pure Appl. Math. Q.}, 3(4, Special Issue: In honor of Grigory
  Margulis. Part 1):1057--1099, 2007.

\bibitem{Lehmer-note-1933}
D.~H. Lehmer.
\newblock Questions, {D}iscussions, and {N}otes: {A} {N}ote on {T}rigonometric
  {A}lgebraic {N}umbers.
\newblock {\em Amer. Math. Monthly}, 40(3):165--166, 1933.

\bibitem{Loftin:2001}
John~C. Loftin.
\newblock Affine spheres and convex {$\Bbb{RP}^n$}-manifolds.
\newblock {\em Amer. J. Math.}, 123(2):255--274, 2001.

\bibitem{Marquis-espaces-2010}
Ludovic Marquis.
\newblock {Espaces des Modules Marqu\'es des Surfaces Projectives Convexes de
  Volume Fini}.
\newblock {\em Geom. Topol.}, 14(4):2103--2149, 2010.

\bibitem{Mosher-tiling-1988}
Lee Mosher.
\newblock Tiling the projective foliation space of a punctured surface.
\newblock {\em Trans. Amer. Math. Soc.}, 306(1):1--70, 1988.

\bibitem{Penner-Decorated-1987}
R.C. Penner.
\newblock { The Decorated Teichm\"uller Space of Punctured Surfaces}.
\newblock {\em Commun. Math. Phys.}, 113:299--339, 1987.

\bibitem{Penner-book-2012}
Robert~C. Penner.
\newblock {\em Decorated {T}eichm\"{u}ller theory}.
\newblock QGM Master Class Series. European Mathematical Society (EMS),
  Z\"{u}rich, 2012.
\newblock With a foreword by Yuri I. Manin.

\bibitem{Schaller-semi-2000}
Paul Schmutz~Schaller and J\"{u}rgen Wolfart.
\newblock Semi-arithmetic {F}uchsian groups and modular embeddings.
\newblock {\em J. London Math. Soc. (2)}, 61(1):13--24, 2000.

\bibitem{Steinbach-golden-1997}
Peter Steinbach.
\newblock Golden fields: a case for the heptagon.
\newblock {\em Math. Mag.}, 70(1):22--31, 1997.

\bibitem{TiWo-algorithm-2016}
Stephan Tillmann and Sampson Wong.
\newblock {An Algorithm for the Euclidean Cell Decomposition of a Cusped
  Strictly Convex Projective Surface}.
\newblock {\em Journal of Computational Geometry}, 7(1):237--255, 2016.

\bibitem{Vinberg-theory-1963}
\`Ernest~B. Vinberg.
\newblock The theory of homogeneous convex cones.
\newblock {\em Trudy Moskov. Mat. Ob\v s\v c.}, 12:303--358, 1963.

\bibitem{Wienhard-invitation-2019}
Anna Wienhard.
\newblock An invitation to higher {T}eichm\"uller theory.
\newblock In {\em Proceedings of the International Congress of Mathematicians
  (ICM 2018)}, pages 1013--1039. World Scientific, 2019.

\end{thebibliography}


\address{Robert Haraway\\ Department of Mathematics, Oklahoma State University, Stillwater, OK 74078-1058, USA\\{robert.haraway@okstate.edu}\\-----}

\address{Robert L\"owe\\Technische Universit\"at Berlin, Institut f\"ur Mathematik, Str. des 17. Juni 136, 10623 Berlin, Germany\\{loewe@math.tu-berlin.de}\\----- }

\address{Dominik Tate\\School of Mathematics and Statistics F07, The University of Sydney, NSW 2006 Australia\\{D.Tate@maths.usyd.edu.au}\\-----}

\address{Stephan Tillmann\\School of Mathematics and Statistics F07, The University of Sydney, NSW 2006 Australia\\{stephan.tillmann@sydney.edu.au}}

\Addresses

\end{document}